\documentclass[smallextended]{svjour3}       
\smartqed  

\usepackage{graphicx}

\usepackage{amsmath,amssymb}
\usepackage{color}
\usepackage{algorithm}
\usepackage{epstopdf}
\usepackage{marvosym} 
\usepackage{geometry}
\usepackage{lineno}

\numberwithin{equation}{section}
\numberwithin{theorem}{section}
\numberwithin{lemma}{section}
\numberwithin{remark}{section}
\numberwithin{definition}{section}
\numberwithin{algorithm}{section}
\numberwithin{figure}{section}
\numberwithin{table}{section}
\renewenvironment{proof}{{\indent \indent \it Proof~}}{\hfill $\square$\par}

\begin{document}
\title{A Posteriori Estimates of Taylor-Hood Element for Stokes Problem Using Auxiliary Subspace Techniques}

\titlerunning{A Posteriori Estimates of Taylor-Hood Element for Stokes Problem Using Auxiliary Subspace Techniques}        

\author{Jiachuan Zhang        \and
        Ran Zhang        \and
        Xiaoshen Wang
}

\authorrunning{Jiachuan Zhang \and Ran Zhang\and Xiaoshen Wang} 

\institute{Jiachuan Zhang \at
              School of Physical and Mathematical Sciences, Nanjing Tech University, Nanjing 211816, P.~R.~China. \\
              \email{zhangjc@njtech.edu.cn}
           \and
           Ran Zhang (\Letter) \at
               School of Mathematics, Jilin University,
               Changchun 130012, P.~ R.~ China.\\
               \email{zhangran@jlu.edu.cn}
           \and
           Xiaoshen Wang\at
           Department of Mathematics and Statistics, University of Arkansas at Little Rock, Arkansas 72204, USA.\\
           \email{xxwang@ualr.edu}
}

\date{}

\maketitle

\begin{abstract}
Based on the auxiliary subspace techniques, a hierarchical basis \textit{a posteriori} error estimator is proposed
for the Stokes problem in two and three dimensions. For the error
estimator, we need to solve only two global diagonal linear systems
corresponding to the degree of freedom of velocity and pressure respectively, which reduces the
computational cost sharply. The upper and lower bounds up to an oscillation term of the
error estimator are also shown to address the reliability of the
adaptive method without saturation assumption. Numerical simulations are performed to demonstrate
the effectiveness and robustness of our algorithm.
\end{abstract}


\section{Introduction}\label{sec:introduction}

In this paper, we propose an \textit{a posteriori} error estimator based on the auxiliary subspace techniques for the Taylor-Hood finite
element method (FEM) \cite{BOFFI1994,BOFFI1997} to solve Stokes
equations \cite{Huang2011,Verfurth2013} with Dirichlet boundary condition

\begin{align}
\label{eqn:original eqn 1}
-\Delta\boldsymbol{u}+\nabla p=\boldsymbol{f}& \quad\mbox{ in }\Omega,\\
\label{eqn:original eqn 2}
\nabla\cdot \boldsymbol{u}=0&  \quad\mbox{ in }\Omega,\\
\label{eqn:original eqn 3} \boldsymbol{u}=\boldsymbol{g}&  \quad\mbox{ on
}\Gamma,
\end{align}
where $\Omega \subset \mathbb{R}^d (d=2,3)$ is a bounded polygonal or
polyhedral domain with the boundary
$\Gamma$. The function $\boldsymbol{u}$
is a vector velocity field and $p$ is the pressure.
The functions $\boldsymbol{f}$ and $\boldsymbol{g}$ are given Lebesgue square-integrable functions on $\Omega$ and $\Gamma$, respectively.
The problem
(\ref{eqn:original eqn 1})-(\ref{eqn:original eqn 3}) has a unique
solution in the sense that $p$ is only determined up to an
additive constant. In the later sections, we will analyze the case of $\boldsymbol{g}=\boldsymbol{0}$, and the case $\boldsymbol{g}\neq\boldsymbol{0}$ is similar.

\textit{A posteriori} error estimators and adaptive FEM can be used to solve the problems with local singularities effectively. Hierarchical basis \textit{a posteriori} estimator is a popular approach and has been proven to be robust and efficient, whose origins can be traced back to \cite{Zienkiewicz1986,Zienkiewicz1982}. In this approach, let $V_k$ and $W_{k+d}$ be the approximation space and auxiliary space, respectively, where $V_k\cap W_{k+d}=\{0\}$ (to be specified in Section \ref{sec:FEM space}). The solution of approximation problem (\ref{eqn:approximation problem}) is denoted by $(\hat{\boldsymbol{u}},\hat{p})\in V_k$. Then the approximation error $\|(\boldsymbol{u}-\hat{\boldsymbol{u}},p-\hat{p})\|_V$ can be estimated in auxiliary space $W_{k+d}$ with the help of the error problem (\ref{eqn:error problem}). Traditionally, the upper and lower bounds of error estimations need to make use of a saturation assumption, i.e. the best approximation of $(\boldsymbol{u},p)$ in $V_k\cup W_{k+d}$ is strictly better than its best approximation in $V_k$. Although saturation assumption is widely accepted in \textit{a posteriori} error analysis \cite{HainReduced2019,Antonietti2013} and satisfied in the case of small data oscillation \cite{Dorfler2002}, it is not difficult to construct counter-examples for particular problems on particular meshes \cite{Bornemann1996}. To remove the saturation assumption, Araya et al. presented an adaptive stabilized FEM combined with a hierarchical basis \textit{a posteriori} error estimator in a special auxiliary bubble function spaces for generalized Stokes problem and Navier-Stokes equations. The error analysis of upper and lower bounds avoids the use of saturation assumption. Although the construction of auxiliary space needs a transformation operator in the reference element, it provides a novel idea for removing saturation assumption in reliability analysis \cite{Araya2008,Araya2005,Araya2012,Araya2018}. Hakula et al. constructed the auxiliary space directly on each element for the second order elliptic problem and elliptic eigenvalue problem and proved that the error is bounded by the error estimator up to oscillation terms without the saturation assumption \cite{Hakula2017,Giani2021}.

The contribution of this paper is twofold. Firstly, we extend the auxiliary subspace techniques in \cite{Hakula2017} to the Stokes problem in two and three dimensions. More specifically, we construct auxiliary spaces for velocity and pressure, respectively and prove that these auxiliary spaces satisfy the inf-sup condition shown in Lemma~\ref{lem:inf-sup one layer in com space}. The error $\|(\boldsymbol{u}-\hat{\boldsymbol{u}},p-\hat{p})\|_V$ can be bounded by the solution of the error problem (\ref{eqn:error problem}), the term $\|\nabla\cdot \hat{\boldsymbol{u}}\|$ and the oscillation term $osc(\boldsymbol{f})$ (Theorem~\ref{thm:equivalence hat e}). We emphasize that the error analysis does not use the saturation assumption. The other contribution of the present work is the diagonalization of the error problem to reduce the computational cost. Considering that the Stokes problem is a saddle point problem, we replace part of the matrix, which is related to velocity only, with a diagonal matrix in (\ref{eqn:define Du}) to construct the second error problem shown in (\ref{eqn:error in com space easily}). Then the solution of (\ref{eqn:error in com space easily}) combined with the term $\|\nabla\cdot \hat{\boldsymbol{u}}\|$ and the oscillation term $osc(\boldsymbol{f})$ can be used to bound the error $\|(\boldsymbol{u}-\hat{\boldsymbol{u}},p-\hat{p})\|_V$ (Theorem~\ref{theo:diag velocity approx error}). Here, obtaining the pressure and velocity requires solving a non-diagonal and diagonal linear system, respectively. To further reduce the computation, the diagonal matrix is obtained by multiplying the diagonal matrix of pressure correlation matrix by a constant $c_s$ related to the number of the bases of pressure in each element. Now, the linear systems of pressure and velocity are both diagonal, which is the third error problem shown in (\ref{eqn:the third error problem}) whose solution combined with the term $\|\nabla\cdot \hat{\boldsymbol{u}}\|$ and the oscillation term $osc(\boldsymbol{f})$ can be used to bound the error $\|(\boldsymbol{u}-\hat{\boldsymbol{u}},p-\hat{p})\|_V$ (Theorem~\ref{theo:diag pressure approx error}).

The rest of the work is organized as follows. In Section 2, the FEM spaces, the approximation problem, and the first error problem are introduced. Section 3
presents a quasi-interpolant based on moment conditions and develops \textit{a posteriori} error estimation related to the first error problem for the Stokes equation. In Section 4, to reduce the computational cost, the system diagonalization techniques are developed for velocity (the second error problem) and pressure (the third error problem), respectively. The \textit{a posteriori} error estimates of the corresponding error problems are shown. In Section 5, we
obtained the local and global \textit{a posteriori} error estimators, and an adaptive FEM is
proposed based on the solution of the third error problem and term $\|\nabla\cdot \hat{\boldsymbol{u}}\|$. In Section 6, numerical
experiment results are presented to verify the effectiveness of our
adaptive algorithm. The last section is devoted to some concluding
remarks.

\section{Approximation Problem and Error Problem}\label{sec:FEM space}
The following notations are used in this paper

\vspace{-10pt}
\begin{align*}
a(\boldsymbol{w},\boldsymbol{v})&=\int_\Omega \nabla\boldsymbol{w}:\nabla\boldsymbol{v},\\
b(\boldsymbol{v},q)&=\int_\Omega q\nabla\cdot\boldsymbol{v},\\
a_1((\boldsymbol{w},r),(\boldsymbol{v},q))&=a(\boldsymbol{w},\boldsymbol{v})
-b(\boldsymbol{v},r)+b(\boldsymbol{w},q),\\
f(\boldsymbol{v})&=\int_\Omega \boldsymbol{f}\cdot \boldsymbol{v},
\end{align*}
for all $(\boldsymbol{v},q),(\boldsymbol{w},r)\in V:=[H_0^1(\Omega)]^d\times L_0^2(\Omega)$, where

\vspace{-10pt}
\begin{align*}
&H^{1}_{0}(\Omega)=\{v\in H^1(\Omega)~\big|~ v=0~on~ \Gamma\},\\
&L_0^2(\Omega)=\{q\in L^2(\Omega)~\big|~\int_{\Omega}q=0\}.
\end{align*}

The variational formulation of (\ref{eqn:original eqn
1})-(\ref{eqn:original eqn 3}) is: Find $(\boldsymbol{u},p)\in
V$ such that

\vspace{-10pt}
\begin{align}\label{eqn:variational form}
a_1((\boldsymbol{u},p),(\boldsymbol{v},q))=f(\boldsymbol{v}),
\end{align}
for all $(\boldsymbol{v},q)\in V$.

We denote by $\|\cdot\|_{m,\Omega}$ and $|\cdot|_{m,\Omega}$ the standard norm and semi-norm of Sobolev space with $m\geq0$, respectively. For the sake of convenience, we will use $\|\cdot\|$ and $|\cdot|$ for $\|\cdot\|_{0,\Omega}$ and $|\cdot|_{0,\Omega}$, respectively. For the coupling space $V$, we define

\vspace{-10pt}
\begin{align}\label{eqn:norm V}
\|(\boldsymbol{v},q)\|_V=\sqrt{\|\nabla\boldsymbol{v}\|^2+\|q\|^2}.
\end{align}
From Cauchy-Schwarz inequality, $a_1(\cdot,\cdot)$ is continuous, i.e.

\vspace{-10pt}
\begin{align}\label{eqn:B continous}
|a_1((\boldsymbol{w},r),(\boldsymbol{v},q))|\leq\mathfrak{C}_1\|(\boldsymbol{w},r)\|_V\|(\boldsymbol{v},q)\|_V.
\end{align}
From Proposition 4.69 in \cite{Verfurth2013}, $a_1(\cdot,\cdot)$ satisfies the estimates

\vspace{-10pt}
\begin{align}\label{eqn:a1 orign inf-sup}
\inf_{(\boldsymbol{v},q)\in V\backslash\{0\}}
\sup_{(\boldsymbol{w},r)\in V\backslash\{0\}}
\frac{a_1((\boldsymbol{v},q),(\boldsymbol{w},r))}
{\|(\boldsymbol{v},q)\|_V\|(\boldsymbol{w},r)\|_V}
\geq \mathfrak{c}_1.
\end{align}
We refer to $\mathfrak{C}_1$ and $\mathfrak{c}_1$ as the continuity and inf-sup constant, respectively.

\subsection{Approximation Problem}

Let $\mathcal{T}$ be a family of conforming, shape-regular simplicial partition of $\Omega$. Let $\mathcal{F}$ denote the set of $(d-1)$-dimensional sub-simplices, the ``faces" of $\mathcal{T}$, and further decompose it as $\mathcal{F}=\mathcal{F}_I\cup\mathcal{F}_D$, where $\mathcal{F}_I$ comprises those faces in the interior of $\Omega$, and $\mathcal{F}_D$ comprises those faces in $\Gamma$. To ensure that the Taylor-Hood element satisfies the stability condition (inf-sup condition), we make the following assumptions for $\mathcal{T}$:
\begin{itemize}
\item[] \textbf{Assumption 1}. $\mathcal{T}$ contains at least three triangles in the case of $d=2$.
\item[] \textbf{Assumption 2}. Every element $T\in\mathcal{T}$ has at least one vertex in the interior of $\Omega$ in the case of $d=3$.
\end{itemize}

In our scheme, in order to have a conforming approximation we shall choose the finite-dimensional spaces $VV_{k+1}$ and $VP_k$ with $k\geq 1$ (called Hood-Taylor or Taylor-Hood element)

\vspace{-10pt}
\begin{align}
\label{eqn:space VV}
VV_{k+1}&=\{\hat{\boldsymbol{v}}\in[H_{0}^1(\Omega)]^d~\big|~ \hat{\boldsymbol{v}}_{|T}\in [P_{k+1}]^d, \forall T\in \mathcal{T} \}\subset [H_{0}^1(\Omega)]^d,\\
\label{eqn:space VP}
VP_k&=\{\hat{q}\in H^1(\Omega)~\big|~\hat{q}_{|T}\in P_k(K), \forall T \in \mathcal{T}, \int_{\Omega}\hat{q}=0\}\subset L_0^2(\Omega),\\
\label{eqn:space Vk}
V_k&=VV_{k+1}\times VP_k.
\end{align}

A mixed finite element method to approximate (\ref{eqn:variational form}) is called an \textbf{approximation problem}: Find $(\hat{\boldsymbol{u}},\hat{p})\in V_k$ such that

\vspace{-10pt}
\begin{align}\label{eqn:approximation problem}
a_1((\hat{\boldsymbol{u}},\hat{p}),(\hat{\boldsymbol{v}},\hat{q}))=f(\hat{\boldsymbol{v}}),
\end{align}
for all $(\hat{\boldsymbol{v}},\hat{q})\in V_{k}$.
\begin{remark}
The solvability of the approximation problem (\ref{eqn:approximation problem}) can be found in \cite{BOFFI1994,BOFFI1997,Brezzi2006}.
\end{remark}

\subsection{Error Problem}
Given a simplex $T\subset\mathbb{R}^d$ of diameter $h_T$, we define $\mathcal{S}_j(T), 0\leq j\leq d$ to be the set of sub-simplices of $T$ of dimension $j$. The cardinality is $|\mathcal{S}_j(T)|=\binom{d+1}{j+1}$. We denote by $\mathcal{S}_j$ the set of sub-simplices of the triangulation of dimension $j$, in particular, $\mathcal{S}_{d-1}=\mathcal{F}_I\cup\mathcal{F}_D$ and $\mathcal{S}_d=\mathcal{T}$. Recall that $P_m(S)$ is the set of polynomials of total degree $\leq m$ with domain $S$, and  note that dim $P_m(S)=\binom{m+j}{j}$ for $S\in\mathcal{S}_j(T)$. Denoting the vertices of $T$ by $\{z_0,\cdots,z_d\}$, we let $\lambda_i\in P_1(T), 0\leq i\leq d$, be the corresponding barycentric coordinates, uniquely defined by the relation $\lambda_i(z_i)=\delta_{ij}$. We denote by $F_j\in \mathcal{S}_{d-1}(T)$ the sub-simplex not containing $z_j$.

The fundamental element and face bubbles for $T$ are given by $(j=0,1,\cdots,d)$

\vspace{-10pt}
\begin{align*}
b_T=\prod_{k=0}^{d}\lambda_k\in P_{d+1}(T),\quad b_{F_j}=\prod_{k=0 \atop k\neq j}^d \lambda_k\in P_d(T).
\end{align*}
We also define general element and face bubbles of degree $m$,

\vspace{-10pt}
\begin{align}\label{eqn:general element bubbles}
&Q_{m}(T)=\{\hat{v}=b_T\hat{w}\in P_m(T)~|~ \hat{w}\in P_{m-d-1}(T)\},\\
\label{eqn:general face bubbles}
&Q_{m}(F_j)=\{\hat{v}=b_{F_j}\hat{w}\in P_m(T)~|~ \hat{w}\in P_{m-d}(T)\}\ominus Q_{m}(T).
\end{align}
From now on, we use the shorthand $W_1\ominus W_2=span\{W_1\backslash W_2\}$ for vector spaces $W_1$ and $W_2$. So $W_1\ominus W_2$ is the largest subspace of $W_1$ such that $W_1\cap W_2=\{0\}$. The functions in $Q_m(T)$ are precisely those in $P_m(T)$ that vanish on $\partial T$, and the functions in $Q_m(F_j)$ are precisely those in $P_m(T)$ that vanish on $\partial T\backslash F_j$. It is clear that $Q_{m}(T)\cap Q_{m}(F_j)=\{0\}$ and $Q_m(F_i)\cap Q_m(F_j)=\{0\}$ for $i\neq j$.
The collection of face bubbles of degree $m$ can be denoted by

\vspace{-10pt}
\begin{align*}
Q_{m}(\partial T)=\bigoplus_{j=0}^d Q_{m}(F_j).
\end{align*}
Then we define the local space

\vspace{-10pt}
\begin{align*}
R_m(T)=Q_m(T)\oplus Q_{m}(\partial T),
\end{align*}
which contains all element and face bubbles of degree $m$ related to $T$ defined in (\ref{eqn:general element bubbles}) and (\ref{eqn:general face bubbles}), and the corresponding global finite element spaces

\vspace{-10pt}
\begin{align*}
R_m=\{\hat{v}\in H_{0}^1(\Omega) ~|~ \hat{v}_{|T}\in R_m(T) \mbox{~for each~}T\in\mathcal{T}\}.
\end{align*}

\begin{lemma}
A function $\hat{v}\in R_{m}(T)$ is uniquely determined by the moments

\vspace{-10pt}
\begin{align}\label{eqn:RT monment}
\int_S \hat{v}\kappa,\quad \forall \kappa\in P_{m-\ell-1}(S),\quad \forall S\in S_{\ell}(T), \quad d-1\leq \ell\leq d.
\end{align}
\end{lemma}
\begin{proof}
As is shown in \cite{Arnold2013}, a function $v\in P_m(T)$ is uniquely determined by the moments

\vspace{-10pt}
\begin{align*}
\int_S \hat{v}\kappa,\quad \forall \kappa\in P_{m-\ell-1}(S),\quad \forall S\in S_{\ell}(T),\quad 0\leq\ell\leq d,
\end{align*}
where $\int_S \hat{v}\kappa$ with $S\in S_0(T)$ is understood to be the evaluation of $\hat{v}$ at the vertex $S$. Since $\hat{v}\in R_{m}(T)$ is uniquely determined by its moments on $T$ and $F_j, j=0,\cdots,d$, the result is clear.
\end{proof}

Given $k\in \mathbb{N}$, we define the local error space for
velocity by element and face bubbles

\vspace{-10pt}
\begin{align*}
WV_{k+d+1}(T)&=[R_{k+d+1}(T)\ominus R_{k+1}(T)]^d,
\end{align*}
and for pressure by element bubbles

\vspace{-10pt}
\begin{align*}
WP_{k+d}(T)&=Q_{k+d}(T)\ominus Q_k(T).
\end{align*}
The velocity and pressure error spaces are constructed this way to satisfy the inf-sup condition shown in Lemma \ref{lem:inf-sup in com space}.

The corresponding global finite element spaces, defined by the degrees of freedom and local spaces, are given by

\vspace{-10pt}
\begin{align}
\label{eqn:space WV}
WV_{k+d+1}&=\{\hat{\boldsymbol{w}}\in [H_{0}^1(\Omega)]^d~|~ \hat{\boldsymbol{w}}_{|T}\in WV_{k+d+1}(T) \mbox{~for each~}T\in\mathcal{T}\},\\
\label{eqn:space WP}
WP_{k+d}&=\{\hat{r}\in L_0^2(\Omega)\cap H^1(\Omega)~|~ \hat{r}_{|T}\in WP_{k+d}(T) \mbox{~for each~}T\in\mathcal{T}\},\\
\label{eqn:space W(k+d)}
W_{k+d}&=WV_{k+d+1}\times WP_{k+d},
\end{align}
where $V_k\cap W_{k+d}=\{0\}$. Then the \textbf{error problem} is: Find $(\hat{\boldsymbol{e}}_{u},\hat{e}_{p})\in W_{k+d}$ such that

\vspace{-10pt}
\begin{align}\label{eqn:error problem}
a_1((\hat{\boldsymbol{e}}_{u},\hat{e}_{p}),(\hat{\boldsymbol{v}},\hat{q}))=f(\hat{\boldsymbol{v}})-a_1((\hat{\boldsymbol{u}},\hat{p}),(\hat{\boldsymbol{v}},\hat{q})),
\end{align}
for any $(\hat{\boldsymbol{v}},\hat{q})\in W_{k+d}$.

\subsection{Solvability of Error Problem}
The error problem is stable (in the sense of inf-sup condition) in $W_{k+d}$ from the following Lemma \ref{lem:inf-sup one layer in com space} and Lemma \ref{lem:inf-sup in com space}.
Let $\overline{P}_k$ be the set of homogeneous polynomials of degree $k\geq 1$. Then we define

\vspace{-10pt}
\begin{align*}
\overline{WV}_{k+j+1}=WV_{k+d+1}\cap[\overline{P}_{k+j+1}]^d,\quad \overline{WP}_{k+j}=WP_{k+d}\cap\overline{P}_{k+j},
\end{align*}
where $1\leq j\leq d$.

\begin{lemma}\label{lem:inf-sup one layer in com space}
Under Assumptions 1 and 2, there exist positive constants $\mu_j~(1\leq j\leq d)$ independent of $h$ such that

\vspace{-10pt}
\begin{align*}
\sup_{\hat{\boldsymbol{v}}\in \overline{WV}_{k+j+1}}\frac{b(\hat{\boldsymbol{v}},\hat{q})}{\| \hat{\boldsymbol{v}}\|_{1,\Omega}}\geq\mu_j\|\hat{q}\|,\quad \forall \hat{q}\in \overline{P}_{k+j},
\end{align*}
where $k\geq 1$.
\end{lemma}
\begin{proof}
The proof can be found in Appendix A.
\end{proof}

\begin{lemma}\label{lem:inf-sup in com space}
There exists a positive constant $\mu$ independent of $h$ such that

\vspace{-10pt}
\begin{align}\label{eqn:inf-sup in com space}
\sup_{\hat{\boldsymbol{v}}\in WV_{k+d+1}}\frac{b(\hat{\boldsymbol{v}},\hat{q})}{\| \hat{\boldsymbol{v}}\|_{1,\Omega}}\geq\mu\|\hat{q}\|,\quad \forall \hat{q}\in WP_{k+d},
\end{align}
where $k\geq 1$.
\end{lemma}
\begin{proof}
It follows from $\overline{WP}_{k+j}\subset\overline{P}_{k+j}$ and Lemma \ref{lem:inf-sup one layer in com space} that

\vspace{-10pt}
\begin{align*}
\sup_{\hat{\boldsymbol{v}}\in \overline{WV}_{k+j+1}}\frac{b(\hat{\boldsymbol{v}},\hat{q})}{\| \hat{\boldsymbol{v}}\|_{1,\Omega}}\geq\mu_j\|\hat{q}\|,\quad \forall \hat{q}\in \overline{WP}_{k+j},
\end{align*}
for $k\geq 1$ and $1\leq j\leq d$. Then set $\mu=\min\limits_{1\leq j\leq d}\mu_j$ and complete the proof from the facts

\vspace{-10pt}
\begin{align*}
WV_{k+d+1}=\bigoplus_{j=1}^{d} \overline{WV}_{k+j+1},\quad WP_{k+d}=\bigoplus_{j=1}^{d} \overline{WP}_{k+j}.
\end{align*}
\end{proof}

From Lemma \ref{lem:inf-sup in com space}, the proof of the following lemma is similar to that of Proposition 4.69 in \cite{Verfurth2013}. For the sake of completeness, we give the proof here.
\begin{lemma}\label{lem:inf-sup B}
The bilinear form $a_1((\hat{\boldsymbol{v}},\hat{q}),(\hat{\boldsymbol{w}},\hat{r}))$ satisfies the estimate

\vspace{-10pt}
\begin{align}\label{eqn:inf-sub B error problem}
\inf_{(\hat{\boldsymbol{v}},\hat{q})\in W_{k+d}\backslash\{0\}}
\sup_{(\hat{\boldsymbol{w}},\hat{r})\in W_{k+d}\backslash\{0\}}
\frac{a_1((\hat{\boldsymbol{v}},\hat{q}),(\hat{\boldsymbol{w}},\hat{r}))}
{\|(\hat{\boldsymbol{v}},\hat{q})\|_V\|(\hat{\boldsymbol{w}},\hat{r})\|_V}
\geq\frac{\mu^2}{(1+\mu)^2}
\end{align}
where $\mu$ is a constant defined in Lemma \ref{lem:inf-sup in com space}.
\end{lemma}
\begin{proof}
Let $(\hat{\boldsymbol{v}},\hat{q})\in W_{k+d}\backslash\{0\}$ be an arbitrary but fixed function. The definition of $a_1(\cdot,\cdot)$ immediately implies that

\vspace{-10pt}
\begin{align*}
a_1((\hat{\boldsymbol{v}},\hat{q}),(\hat{\boldsymbol{v}},\hat{q}))=\|\nabla \hat{\boldsymbol{v}}\|.
\end{align*}
Due to Lemma \ref{lem:inf-sup in com space}, there is a velocity field $\hat{\boldsymbol{w}}_{\hat{q}}\in WV_{k+d+1}$ with $\|\nabla \hat{\boldsymbol{w}}_{\hat{q}}\|=1$ such that

\vspace{-10pt}
\begin{align*} \int_{\Omega}\hat{q}\nabla\cdot\hat{\boldsymbol{w}}_{\hat{q}}\geq\mu\|\hat{q}\|.
\end{align*}
We therefore obtain for every $\delta>0$

\vspace{-10pt}
\begin{align*}
a_1((\hat{\boldsymbol{v}},\hat{q}),(\hat{\boldsymbol{v}}-\delta\|\hat{q}\|\hat{\boldsymbol{w}}_{\hat{q}},\hat{q}))=&a_1((
\hat{\boldsymbol{v}},\hat{q}),(\hat{\boldsymbol{v}},\hat{q}))-\delta\|\hat{q}\|a_1((\hat{\boldsymbol{v}},\hat{q}),(\hat{\boldsymbol{w}}_{\hat{q}},0))\\
=&\|\nabla\hat{\boldsymbol{v}}\|^2-\delta\|\hat{q}\|\int_{\Omega}\nabla\hat{\boldsymbol{v}}:\nabla\hat{\boldsymbol{w}}_{\hat{q}}
+\delta\|\hat{q}\|\int_{\Omega}\hat{q}\nabla\cdot\hat{\boldsymbol{w}}_{\hat{q}}\\
\geq&\|\nabla\hat{\boldsymbol{v}}\|^2-\delta\|\nabla\hat{\boldsymbol{v}}\|\|\hat{q}\|+\delta\mu\|\hat{q}\|^2\\
\geq&(1-\frac{\delta}{2\mu})\|\nabla\hat{\boldsymbol{v}}\|^2+\frac{1}{2}\delta\mu\|\hat{q}\|^2.
\end{align*}
The choice of $\delta=\frac{2\mu}{1+\mu^2}$ yields

\vspace{-10pt}
\begin{align*}
a_1((\hat{\boldsymbol{v}},\hat{q}),(\hat{\boldsymbol{v}}-\delta\|\hat{q}\|\hat{\boldsymbol{w}}_{\hat{q}},\hat{q}))\geq\frac{\mu^2}{1+\mu^2}
\|(\hat{\boldsymbol{v}},\hat{q})\|_V^2.
\end{align*}
On the other hand, we have

\vspace{-10pt}
\begin{align*}
\|(\hat{\boldsymbol{v}}-\delta\|\hat{q}\|\hat{\boldsymbol{w}}_{\hat{q}},\hat{q})\|_V
\leq&\|(\hat{\boldsymbol{v}},\hat{q})\|_V
+\|(\delta\|\hat{q}\|\hat{\boldsymbol{w}}_{\hat{q}},0)\|_V\\
=&\|(\hat{\boldsymbol{v}},\hat{q})\|_V+\delta\|\hat{q}\|\|\nabla\hat{\boldsymbol{w}}_{\hat{q}}\|\\
=&\|(\hat{\boldsymbol{v}},\hat{q})\|_V+\delta\|\hat{q}\|\\
\leq&(1+\delta)\|(\hat{\boldsymbol{v}},\hat{q})\|_V\\
=&\frac{1+\mu^2+2\mu}{1+\mu^2}\|(\hat{\boldsymbol{v}},\hat{q})\|_V.
\end{align*}
Combining these estimates we arrive at

\vspace{-10pt}
\begin{align*}
\sup_{(\hat{\boldsymbol{w}},\hat{r})\in W_{k+d}\backslash\{0\}}
\frac{a_1((\hat{\boldsymbol{v}},\hat{q}),(\hat{\boldsymbol{w}},\hat{r}))}
{\|(\hat{\boldsymbol{v}},\hat{q})\|_V\|(\hat{\boldsymbol{w}},\hat{r})\|_V}
\geq&\frac{a_1((\hat{\boldsymbol{v}},\hat{q}),(\hat{\boldsymbol{v}}-\delta\|\hat{q}\|\hat{\boldsymbol{w}}_{\hat{q}},\hat{q}))}
{\|(\hat{\boldsymbol{v}},\hat{q})\|_V\|(\hat{\boldsymbol{v}}
-\delta\|\hat{q}\|\hat{\boldsymbol{w}}_{\hat{q}},\hat{q})\|_V}\\
\geq&\frac{\mu^2}{1+\mu^2}.
\end{align*}
Since $(\hat{\boldsymbol{v}},\hat{q})\in W_{k+d}\backslash\{0\} $ was arbitrary, this completes the proof.
\end{proof}

\begin{theorem}\label{thm:unique of error problem}
The error problem (\ref{eqn:error problem}) has a unique solution.
\end{theorem}
\begin{proof}
For the system (\ref{eqn:error problem}), one can easily check that $a_1(\cdot,\cdot)$ is a continuous bilinear form on $W_{k+d}\times W_{k+d}\subset V\times V$ by (\ref{eqn:B continous}) and satisfies the inf-sup condition by Lemma \ref{lem:inf-sup B}. In addition, $f(\hat{\boldsymbol{v}})-a_1((\hat{\boldsymbol{u}},\hat{p}),(\hat{\boldsymbol{v}},\hat{q}))$ is a continuous linear functional on $W_{k+d}$ and the bilinear form $a_1(\cdot,\cdot)$ satisfies

\vspace{-10pt}
\begin{align*}
a_1((\hat{\boldsymbol{v}},\hat{q}),(\hat{\boldsymbol{v}},\hat{q}))=\|\nabla \hat{\boldsymbol{v}}\|^2\geq C\|\hat{\boldsymbol{v}}\|^2>0,\quad \hat{\boldsymbol{v}}\neq 0,
\end{align*}
by Poincare's inequalities.
So by Theorem 5.2.1 in \cite{Babuska1972}, the scheme (\ref{eqn:error problem}) has a unique solution.
\end{proof}

\section{\textit{A Posteriori} Error Estimation}

In this section, a quasi-interpolant based on moment conditions will be shown in Lemma~\ref{lem:quasi-interpolant v}, which is used to get the \textit{a posteriori} error estimate shown in Theorem~\ref{thm:equivalence hat e}.
\begin{lemma}\label{lem:quasi-interpolant v}
Given $\boldsymbol{v}\in [H^1(\Omega)]^d$, there exits a $\hat{\boldsymbol{v}}\in VV_{k+1}$ and $\hat{\boldsymbol{w}}\in WV_{k+d+1}$ such that
\begin{itemize}
\item[(1)]$\int_T(\boldsymbol{v}-\hat{\boldsymbol{v}}-\hat{\boldsymbol{w}})\cdot\boldsymbol{\kappa}=0 \mbox{~for all~} \boldsymbol{\kappa}\in [P_k(T)]^d \mbox{~and~} T\in\mathcal{T}$.
\item[(2)]$\int_F(\boldsymbol{v}-\hat{\boldsymbol{v}}-\hat{\boldsymbol{w}})\cdot\boldsymbol{\kappa}=0 \mbox{~for all~} \boldsymbol{\kappa}\in [P_{k+1}(F)]^d \mbox{~and~} F\in\mathcal{F}_I$.
\item[(3)] $|\boldsymbol{v}-\hat{\boldsymbol{v}}-\hat{\boldsymbol{w}}|_{m,T}\leq C_{\mathcal{T}}h_T^{1-m}|\boldsymbol{v}|_{1,\Omega_T}$ for $m=0,1$, where $\Omega_T$ is a local patch of elements containing $T$.
\item[(4)] $|\boldsymbol{v}-\hat{\boldsymbol{v}}-\hat{\boldsymbol{w}}|_{0,F}\leq C_{\mathcal{T}}h_F^{1/2}|\boldsymbol{v}|_{1,\Omega_F}$, where $h_F$ is the diameter of $F\in\mathcal{F}$, and $\Omega_F=\Omega_T$ for some $T\in\mathcal{T}$ with $F\subset\partial T$.
\item[(5)] $|\hat{\boldsymbol{w}}|_{1,T}\leq C_{\mathcal{T}}|\boldsymbol{v}|_{1,\Omega_T}$ for each $T\in\mathcal{T}$.
\end{itemize}
where $C_{\mathcal{T}}$ depends only on the dimension $d$, polynomial degree $k$, and the shape-regularity of $\mathcal{T}$.
\end{lemma}
\begin{proof}
Since functions in $R_{k+d+1}(T)$ are uniquely determined by the moments (\ref{eqn:RT monment}), for $m=0,1$ the function $\langle\!\langle\cdot\rangle\!\rangle_{m,T}: [R_{k+d+1}(T)]^d\rightarrow \mathbb{R}^{+}$ defined by

\vspace{-10pt}
\begin{align*}
\langle\!\langle\boldsymbol{\phi}\rangle\!\rangle_{m,T}=\max_{S\in S_\ell(T)\atop d-1\leq \ell\leq d}\sup_{\boldsymbol{\kappa}\in [P_{k+d-\ell}(S)]^d}\frac{h_T^{d/2-\ell/2-m}}{\|\boldsymbol{\kappa}\|_{0,S}}
\int_{S}\boldsymbol{\phi}\cdot\boldsymbol{\kappa}
\end{align*}
is a norm on $[R_{k+d+1}(T)]^d$.

Let $\tilde{T}=\{y=h_T^{-1}x:x\in T\}$, and for each $\psi: T\rightarrow \mathbb{R}$, define $\tilde{\psi}: \tilde{T}\rightarrow \mathbb{R}$ by $\tilde{\psi}(y)=\psi(h_Ty)$. Analogous definitions are given for the sub-simplices of $T$ and $\tilde{T}$ and functions defined on them. It is clear that $|\boldsymbol{\phi}|_{m,T}=h_T^{d/2-m}|\tilde{\boldsymbol{\phi}}|_{m,\tilde{T}}$, where $|\cdot|_{0,T}=\|\cdot\|_{0,T}$. We also have for any $S\in S_{\ell}(T)$

\vspace{-10pt}
\begin{align*}
\frac{h_T^{d/2-\ell/2-m}}{\|\boldsymbol{\kappa}\|_{0,S}}\int_{S}\boldsymbol{\phi}
\cdot\boldsymbol{\kappa}
=\frac{h_T^{d/2-\ell/2-m}}{h_T^{\ell/2}\|\tilde{\boldsymbol{\kappa}}\|_{0,\tilde{S}}}
\int_{\tilde{S}}\tilde{\boldsymbol{\phi}}\cdot\tilde{\boldsymbol{\kappa}}h_T^{\ell}=
\frac{h_T^{d/2-m}}{\|\tilde{\boldsymbol{\kappa}}\|_{0,\tilde{S}}}
\int_{\tilde{S}}\tilde{\boldsymbol{\phi}}\cdot\tilde{\boldsymbol{\kappa}}.
\end{align*}

Since $h_{\tilde{T}}=1$, we set that $\langle\!\langle\boldsymbol{\phi}\rangle\!\rangle_{m,T}=h_T^{d/2-m}
\langle\!\langle\tilde{\boldsymbol{\phi}}\rangle\!\rangle_{m,\tilde{T}}$. Therefore there exists a scale-invariant constant $C_{\mathcal{T}}>0$ that depends solely on $k,d$, and $m$ such that

\vspace{-10pt}
\begin{align}\label{eqn:norm leq temp norm}
|\boldsymbol{\phi}|_{m,T}=h_T^{d/2-m}|\tilde{\boldsymbol{\phi}}|_{m,\tilde{T}}\leq C_{\mathcal{T}}h_T^{d/2-m}\langle\!\langle\tilde{\boldsymbol{\phi}}\rangle\!\rangle_{m,\tilde{T}}=C_{\mathcal{T}}
\langle\!\langle\boldsymbol{\phi}\rangle\!\rangle_{m,T}.
\end{align}
At this stage, we see that the local constant $C_{\mathcal{T}}$ in (\ref{eqn:norm leq temp norm}) may depend on the shape of $T$, but not its diameter. For the rest of the argument, we make a shape-regularity assumption on $\mathcal{T}$.

Next, denote by $\hat{\boldsymbol{v}}_{1}\in VV_{k+1}$ the Scott-Zhang interpolant of $\boldsymbol{v}$ satisfying \cite{Scott1990}

\vspace{-10pt}
\begin{align}
\label{eqn:Scott-Zhang interpolant 1}
&\|\boldsymbol{v}-\hat{\boldsymbol{v}}_1\|_{m,T}\leq C_{\mathcal{T}}h_T^{1-m}|\boldsymbol{v}|_{1,\Omega_T},\quad m=0,1,\\
\label{eqn:Scott-Zhang interpolant 2}
&\|\boldsymbol{v}-\hat{\boldsymbol{v}}_1\|_{0,\partial T}\leq C_{\mathcal{T}}h_T^{1/2}|\boldsymbol{v}|_{1,\Omega_T},
\end{align}
on each $T\in\mathcal{T}$. Set $\hat{\boldsymbol{v}}_2\in [R_{k+d+1}]^d$ such that

\vspace{-10pt}
\begin{align*}
\int_S\hat{\boldsymbol{v}}_2\cdot\boldsymbol{\kappa}
=\int_S(\boldsymbol{v}-\hat{\boldsymbol{v}}_1)\cdot\boldsymbol{\kappa},\quad \forall \boldsymbol{\kappa}\in[P_{k+d-\ell}(S)]^d,\quad \forall S\in S_{\ell},\quad d-1\leq \ell\leq d.
\end{align*}
By (\ref{eqn:norm leq temp norm})-(\ref{eqn:Scott-Zhang interpolant 2}) we get

\vspace{-10pt}
\begin{align*}
|\hat{\boldsymbol{v}}_2|_{m,T}&\leq C_{\mathcal{T}}\max_{S\in S_{\ell}(T)\atop d-1\leq \ell\leq d}\sup_{\boldsymbol{\boldsymbol{\kappa}}\in[P_{k+d-\ell}(S)]^d}\frac{h_T^{d/2-\ell/2-m}}{\|\boldsymbol{\kappa}\|_{0,S}}
\int_S\hat{\boldsymbol{v}}_2\cdot\boldsymbol{\kappa}\\
&=C_{\mathcal{T}}\max_{S\in S_{\ell}(T)\atop d-1\leq \ell\leq d}\sup_{\boldsymbol{\kappa}\in[P_{k+d-\ell}(S)]^d}\frac{h_T^{d/2-\ell/2-m}}{\|\boldsymbol{\kappa}\|_{0,S}}
\int_S(\boldsymbol{v}-\hat{\boldsymbol{v}}_1)\cdot\boldsymbol{\kappa}\\
&\leq C_{\mathcal{T}}(h_T^{1/2-m}\|\boldsymbol{v}-\hat{\boldsymbol{v}}_1\|_{0,\partial T}+h_T^{-m}\|\boldsymbol{v}-\hat{\boldsymbol{v}}_1\|_{0,T})\leq Ch_T^{1-m}|\boldsymbol{v}|_{1,\Omega_T}.
\end{align*}

Uniquely decomposing $\hat{\boldsymbol{v}}_2$ as $\hat{\boldsymbol{v}}_2=\hat{\boldsymbol{v}}_3+\hat{\boldsymbol{w}}$ with $\hat{\boldsymbol{v}}_3\in VV_{k+1}$ and $\hat{\boldsymbol{w}}\in WV_{k+d+1}$, and setting $\hat{\boldsymbol{v}}=\hat{\boldsymbol{v}}_1+\hat{\boldsymbol{v}}_3$ so that $\hat{\boldsymbol{v}}+\hat{\boldsymbol{w}}=\hat{\boldsymbol{v}}_1+\hat{\boldsymbol{v}}_2$, we see that properties (1)-(2) clearly hold, and

\vspace{-10pt}
\begin{align*}
\|\boldsymbol{v}-\hat{\boldsymbol{v}}-\hat{\boldsymbol{w}}\|_{m,T}
\leq\|\boldsymbol{v}-\hat{\boldsymbol{v}}_1\|_{m,T}+\|\hat{\boldsymbol{v}}_2\|_{m,T}
\leq C_{\mathcal{T}}h_T^{1-m}|\boldsymbol{v}|_{1,\Omega_T}.
\end{align*}
Therefore by the standard trace inequalities and the shape regularity of the mesh, we also have on $F\subset\partial T$

\vspace{-10pt}
\begin{align*}
\|\boldsymbol{v}-\hat{\boldsymbol{v}}-\hat{\boldsymbol{w}}\|_{0,F}\leq C_{\mathcal{T}}(h_F^{-1/2}\|\boldsymbol{v}-\hat{\boldsymbol{v}}-\hat{\boldsymbol{w}}\|_{0,T}
+h_F^{1/2}\|\boldsymbol{v}-\hat{\boldsymbol{v}}-\hat{\boldsymbol{w}}\|_{1,T})\leq C_{\mathcal{T}}h_F^{1/2}|\boldsymbol{v}|_{1,\Omega_F}.
\end{align*}
Hence, properties (3)-(4) are satisfied.

Finally, since $VV_{k+1}(T)\cap WV_{k+d+1}(T)=\{0\}$, the strengthened Cauchy-Schwarz inequality \cite{Eijkhout1991} gives the existence of a constant $\gamma\in[0,1)$ such that

\vspace{-10pt}
\begin{align*}
\int_T\nabla \hat{\boldsymbol{w}}\cdot\nabla\hat{\boldsymbol{v}}_3\leq\gamma|\hat{\boldsymbol{w}}|_{1,T}|\hat{\boldsymbol{v}}_3|_{1,T}.
\end{align*}
Consequently, we have

\vspace{-10pt}
\begin{align*}
|\hat{\boldsymbol{v}}_2|_{1,T}^2&=|\hat{\boldsymbol{w}}|_{1,T}^2+|\hat{\boldsymbol{v}}_3|_{1,T}^2
+2\int_T\nabla \hat{\boldsymbol{w}}:\nabla\hat{\boldsymbol{v}}_3\\
&\geq |\hat{\boldsymbol{w}}|_{1,T}^2+|\hat{\boldsymbol{v}}_3|_{1,T}^2
-2\gamma|\hat{\boldsymbol{w}}|_{1,T}|\hat{\boldsymbol{v}}_3|_{1,T}
\geq(1-\gamma^2)|\hat{\boldsymbol{w}}|_{1,T}^2.
\end{align*}
Therefore we find $|\hat{\boldsymbol{w}}|_{1,T}\leq \sqrt{(1-\gamma^2)^{-1}}|\hat{\boldsymbol{v}}_2|_{1,T}\leq C_{\mathcal{T}}|\boldsymbol{v}|_{1,\Omega_T}$.
\end{proof}

For $(\boldsymbol{v},q)\in V$, we have

\vspace{-10pt}
\begin{align}\label{eqn:u and u hat}
a_1((\boldsymbol{u}-\hat{\boldsymbol{u}},p-\hat{p}),(\boldsymbol{v},q))=f(\boldsymbol{v})-a_1((\hat{\boldsymbol{u}},\hat{p}),(\boldsymbol{v},q)),
\end{align}
where $(\boldsymbol{u},p)$ and $(\hat{\boldsymbol{u}},\hat{p})$ are the solutions of (\ref{eqn:variational form}) and (\ref{eqn:approximation problem}), respectively. So,

\vspace{-10pt}
\begin{align*}
a_1((\boldsymbol{u}-\hat{\boldsymbol{u}},p-\hat{p}),(\boldsymbol{v},q))&=\sum_{T\in \mathcal{T}}\int_T(\boldsymbol{f}\cdot\boldsymbol{v}
-\nabla\hat{\boldsymbol{u}}:\nabla\boldsymbol{v}
+\nabla\cdot\boldsymbol{v}\hat{p}-\nabla\cdot\hat{\boldsymbol{u}} q)\\
&=\sum_{T\in \mathcal{T}}\int_T(\boldsymbol{f}\cdot\boldsymbol{v}-(-\Delta\hat{\boldsymbol{u}}\cdot\boldsymbol{v}
+\nabla \hat{p}\cdot\boldsymbol{v}-\nabla\cdot\hat{\boldsymbol{u}} q))\\
&\quad+\sum_{T\in \mathcal{T}}\int_{\partial T}(-\nabla\hat{\boldsymbol{u}}\cdot\boldsymbol{n}_T\cdot\boldsymbol{v}
+\hat{p}\boldsymbol{v}\cdot\boldsymbol{n}_T)
\end{align*}

\begin{lemma}\label{lem:u-u hat and e hat}
For any $(\boldsymbol{v},q)\in V$, $(\hat{\boldsymbol{w}},\hat{r})\in W_{k+d}$, and $(\hat{\boldsymbol{v}},\hat{q})\in V_k$, it holds that

\vspace{-10pt}
\begin{align}\label{eqn:u-u hat and e hat}
a_1((\boldsymbol{u}-\hat{\boldsymbol{u}},p-\hat{p}),(\boldsymbol{v},q))=a_1((\hat{\boldsymbol{e}}_{u},\hat{e}_{p}),(\hat{\boldsymbol{w}},\hat{r}))
+\mathcal{R}(\boldsymbol{v}-\hat{\boldsymbol{w}}-\hat{\boldsymbol{v}},q-\hat{r}-\hat{q})
\end{align}
where $(\boldsymbol{u},p)$ and $(\hat{\boldsymbol{u}},\hat{p})$ are the solutions of (\ref{eqn:variational form}) and (\ref{eqn:approximation problem}), respectively, and

\vspace{-10pt}
\begin{align*}
\mathcal{R}(\boldsymbol{w},r)&=f(\boldsymbol{w})-a_1((\hat{\boldsymbol{u}},\hat{p}),(\boldsymbol{w},r))\\
&=\sum_{T\in\mathcal{T}}\int_{T}((\boldsymbol{f}-\boldsymbol{R}_T)\cdot\boldsymbol{w}
+\nabla\cdot\hat{\boldsymbol{u}}r)+\sum_{F\in \mathcal{F}_I}\int_{F}\boldsymbol{r}_F\cdot\boldsymbol{w},
\end{align*}
for any $(\boldsymbol{w},r)\in [H_0^1(\Omega)]^d\times L_0^2(\Omega)$ and

\vspace{-10pt}
\begin{align*}
&\boldsymbol{R}_T=(-\Delta\hat{\boldsymbol{u}}
+\nabla \hat{p})_{|T},\\
&\boldsymbol{r}_F=(-\nabla\hat{\boldsymbol{u}}\cdot\boldsymbol{n}_T
+\hat{p}\boldsymbol{n}_T)_{|T}-(-\nabla\hat{\boldsymbol{u}}\cdot\boldsymbol{n}_{T^{\prime}}
+\hat{p}\boldsymbol{n}_{T^{\prime}})_{|T^{\prime}}.
\end{align*}
Here, $T$ and $T^{\prime}$ are the simplices sharing the face $F$, and $\boldsymbol{n}_T$ and $\boldsymbol{n}_{T^{\prime}}$ are their outward unit normals.
\end{lemma}
\begin{proof}
From (\ref{eqn:approximation problem}), (\ref{eqn:error problem}), and (\ref{eqn:u and u hat}), we obtain

\vspace{-10pt}
\begin{align*}
&\mathcal{R}(\boldsymbol{v}-\hat{\boldsymbol{w}}-\hat{\boldsymbol{v}},q-\hat{r}-\hat{q})\\
=&\mathcal{R}(\boldsymbol{v},q)
-\mathcal{R}(\hat{\boldsymbol{w}},\hat{r})
-\mathcal{R}(\hat{\boldsymbol{v}},\hat{q})\\
=&f(\boldsymbol{v})-a_1((\hat{\boldsymbol{u}},\hat{p}),(\boldsymbol{v},q))
-(f(\hat{\boldsymbol{w}})-a_1((\hat{\boldsymbol{u}},\hat{p}),(\hat{\boldsymbol{w}},\hat{r})))
-(f(\hat{\boldsymbol{v}})-a_1((\hat{\boldsymbol{u}},\hat{p}),(\hat{\boldsymbol{v}},\hat{q})))\\
=&a_1((\boldsymbol{u}-\hat{\boldsymbol{u}},p-\hat{p}),(\boldsymbol{v},q))-
a_1((\hat{\boldsymbol{e}}_{u},\hat{e}_{p}),(\hat{\boldsymbol{w}},\hat{r})),
\end{align*}
which completes the proof.
\end{proof}

We define the local oscillation for each $T\in\mathcal{T}$ by

\vspace{-10pt}
\begin{align*}
osc(\boldsymbol{f}, T)^2=h_T^2\inf_{\boldsymbol{\kappa}\in [P_k(T)]^d}\|\boldsymbol{f}-\boldsymbol{\kappa}\|_{0,T}^2.
\end{align*}
Then define

\vspace{-10pt}
\begin{align}\label{eqn:residual oscillation}
osc(\boldsymbol{f})^2=\sum_{T\in\mathcal{T}}osc(\boldsymbol{f},T)^2.
\end{align}

\begin{theorem}\label{thm:equivalence hat e}
Let $(\boldsymbol{u},p)$, $(\hat{\boldsymbol{u}},\hat{p})$, and
$(\hat{\boldsymbol{e}}_u,\hat{e}_p)$ be the solutions
of (\ref{eqn:variational form}),(\ref{eqn:approximation problem}), and (\ref{eqn:error problem}), respectively. There are constants $\hat{\mathfrak{C}}_*=\frac{\mu^2}{2\mathfrak{C}_1(1+\mu)^2}$ and $\hat{\mathfrak{C}}^*=\frac{\mathfrak{C}_1}{\mathfrak{c}_1}$ such that

\vspace{-10pt}
\begin{align}\label{eqn:equivalence hat e}
\hat{\mathfrak{C}}_*\|(\hat{\boldsymbol{e}}_{u},\hat{e}_p)\|_V
+\frac{1}{2\sqrt{d}}\|\nabla\cdot\hat{\boldsymbol{u}}\|
\leq\|(\boldsymbol{u}-\hat{\boldsymbol{u}},p-\hat{p})\|_V\leq \hat{\mathfrak{C}}^*\|(\hat{\boldsymbol{e}}_{u},\hat{e}_p)\|_V
+\frac{1}{\mathfrak{c}_1}\|\nabla\cdot\hat{\boldsymbol{u}}\|
+\frac{C_{\mathcal{T}}}{\mathfrak{c}_1}osc(\boldsymbol{f}),
\end{align}
where constants $\mathfrak{C}_1, \mathfrak{c}_1, \mu$, and $C_{\mathcal{T}}$ are defined in  (\ref{eqn:B continous}), (\ref{eqn:a1 orign inf-sup}), (\ref{eqn:inf-sup in com space}), and Lemma~\ref{lem:quasi-interpolant v}.
\end{theorem}
\begin{proof}
Given $q\in L_0^2(\Omega)$, there exists $\hat{r}\in WP_{k+d}$ such that $\|q-\hat{r}\|\leq\|q\|$ since $WP_{k+d}\subset L_0^2(\Omega)$. Then combining Lemma \ref{lem:quasi-interpolant v},
Lemma \ref{lem:u-u hat and e hat}, and noting $\boldsymbol{R}_T\in[P_k(T)]^2, \boldsymbol{r}_F\in[P_k(F)]^2$, we determine that

\vspace{-10pt}
\begin{align*}
|a_1((\boldsymbol{u}-\hat{\boldsymbol{u}},p-\hat{p}),(\boldsymbol{v},q))|
\leq&|a_1((\hat{\boldsymbol{e}}_{u},\hat{e}_{p}),(\hat{\boldsymbol{w}},\hat{r}))|
+\sum_{T\in\mathcal{T}}\|\boldsymbol{v}-\hat{\boldsymbol{v}}-\hat{\boldsymbol{w}}\|_{0,T}
\inf_{\boldsymbol{\kappa}\in[P_k(T)]^d}\|\boldsymbol{f}-\boldsymbol{\kappa}\|_{0,T}\\
&+\sum_{T\in\mathcal{T}}\|\boldsymbol{v}-\hat{\boldsymbol{v}}-\hat{\boldsymbol{w}}\|_{0,T}
\inf_{\boldsymbol{\kappa}\in[P_k(T)]^d}\|\boldsymbol{R}_T-\boldsymbol{\kappa}\|_{0,T}\\
&+\big|\sum_{T\in\mathcal{T}}\int_T(q-\hat{q}-\hat{r})\nabla\cdot\hat{\boldsymbol{u}}\big|
+\sum_{F\in \mathcal{F}_I}\|\boldsymbol{v}-\hat{\boldsymbol{v}}-\hat{\boldsymbol{w}}\|_{0,F}
\inf_{\boldsymbol{\kappa}\in[P_{k+1}(F)]^d}\|\boldsymbol{r}_F-\boldsymbol{\kappa}\|_{0,F}\\
\leq&\mathfrak{C}_1\|(\hat{\boldsymbol{e}}_{u},\hat{e}_{p})\|_V
\|(\hat{\boldsymbol{w}},\hat{r})\|_V
+C_{\mathcal{T}}\sum_{T\in\mathcal{T}}h_T\|\boldsymbol{v}\|_{1,\Omega_T}
\inf_{\boldsymbol{\kappa}\in[P_k(T)]^d}\|\boldsymbol{f}-\boldsymbol{\kappa}\|_{0,T}\\
&+\sum_{T\in\mathcal{T}}\|q\|\|\nabla\cdot\hat{\boldsymbol{u}}\|\\
\leq&\mathfrak{C}_1\|(\hat{\boldsymbol{e}}_{u},\hat{e}_{p})\|_V\|(\boldsymbol{v},q)\|_V
+C_{\mathcal{T}} osc(\boldsymbol{f})\|(\boldsymbol{v},q)\|_V+\|\nabla\cdot\hat{\boldsymbol{u}}\|
\|(\boldsymbol{v},q)\|_V,
\end{align*}
for any $\hat{\boldsymbol{w}}\in WV_{k+d+1}$ and $(\hat{\boldsymbol{v}},\hat{q})\in W_{k+d}$. Then the right inequality of (\ref{eqn:equivalence hat e}) follows from the inf-sup condition (\ref{eqn:a1 orign inf-sup}) of continuous problem:

\vspace{-10pt}
\begin{align*}
\mathfrak{c}_1\|(\boldsymbol{u}-\hat{\boldsymbol{u}},p-\hat{p})\|_V
\leq& \sup_{(\boldsymbol{w},r)\in V\backslash\{0\}}
\frac{a_1((\boldsymbol{u}-\hat{\boldsymbol{u}},p-\hat{p}),(\boldsymbol{w},r))}
{\|(\boldsymbol{w},r)\|_V}.
\end{align*}

From (\ref{eqn:inf-sub B error problem}),(\ref{eqn:error problem}), and (\ref{eqn:variational form}),

\vspace{-10pt}
\begin{align*}
\frac{\mu^2}{(1+\mu)^2}\|(\hat{\boldsymbol{e}}_{u},\hat{e}_{p})\|_V
\leq&
\sup_{(\hat{\boldsymbol{w}},\hat{r})\in W_{k+d}\backslash\{0\}}
\frac{a_1((\hat{\boldsymbol{e}}_{u},\hat{e}_{p}),(\hat{\boldsymbol{w}},\hat{r}))}
{\|(\hat{\boldsymbol{w}},\hat{r})\|_V}\\
=&\sup_{(\hat{\boldsymbol{w}},\hat{r})\in W_{k+d}\backslash\{0\}}
\frac{f(\hat{\boldsymbol{v}})-a_1((\hat{\boldsymbol{u}},\hat{p}),(\hat{\boldsymbol{w}},\hat{r}))}
{\|(\hat{\boldsymbol{w}},\hat{r})\|_V}\\
=&\sup_{(\hat{\boldsymbol{w}},\hat{r})\in W_{k+d}\backslash\{0\}}
\frac{a_1((\boldsymbol{u}-\hat{\boldsymbol{u}},p-\hat{p}),(\hat{\boldsymbol{w}},\hat{r}))}
{\|(\hat{\boldsymbol{w}},\hat{r})\|_V}\\
\leq&\mathfrak{C}_1\|(\boldsymbol{u}-\hat{\boldsymbol{u}},p-\hat{p})\|_V.
\end{align*}
Since $\nabla\cdot\boldsymbol{u}=0$, we have

\vspace{-10pt}
\begin{align*}
\|\nabla\cdot\hat{\boldsymbol{u}}\|=\|\nabla\cdot(\boldsymbol{u}-\hat{\boldsymbol{u}})\|
\leq\sqrt{d}\|\nabla(\boldsymbol{u}-\hat{\boldsymbol{u}})\|
\leq\sqrt{d}\|(\boldsymbol{u}-\hat{\boldsymbol{u}},p-\hat{p})\|_V.
\end{align*}
Then we get the left inequality of (\ref{eqn:equivalence hat e}).
\end{proof}

\section{System Diagonalization}
As stated, the computation of $(\hat{\boldsymbol{e}}_u,\hat{e}_p)$ requires the formation and solution of a global system, so one might naturally be concerned that this approach is too expensive for practical consideration. Generally speaking, the hierarchical basis for $W_{k+d}$ is typically made up of highly oscillatory functions with compact support, therefore we may approximate the stiffness matrix by a diagonal matrix, which reduces the cost of computation.

\subsection{Diagonalization with respect to Velocity}
Let $\{\phi_j\}_{j=1}^N$ be the bases for $W_{k+d}$, i.e.

\vspace{-10pt}
\begin{align*}
W_{k+d}=span\{\phi_j\}_{j=1}^N.
\end{align*}
Let $\{\varphi_j\}_{j=1}^{N_{v}}$ and $\{\psi_j\}_{j=1}^{N_{p}}$  be
the bases in $W_{k+d}$ for velocity and pressure,
respectively. It is clear that $N=N_v+N_p$ and
$\{\phi_j\}_{j=1}^N=\{\varphi_j\}_{j=1}^{N_{v}}\cup
\{\psi_j\}_{j=1}^{N_{p}}$.

Define an $N_v\times N_v$ matrix $A$ by
$A_{\ell,j}=a(\varphi_j,\varphi_{\ell})$ and an $N_v\times N_p$ matrix
$B$ by $B_{\ell,j}=-b(\psi_j,\varphi_{\ell})$. Then we can rewrite
(\ref{eqn:error problem}) in a matrix form

\vspace{-10pt}
\begin{align}\label{eqn:define A}
\begin{bmatrix}
A & B  \\
-B^T & 0
\end{bmatrix}
\begin{bmatrix}
\boldsymbol{x}_u\\
\boldsymbol{x}_p
\end{bmatrix}
=
\begin{bmatrix}
F_v \\
F_p
\end{bmatrix},
\end{align}
where $\boldsymbol{x}_u$ and $\boldsymbol{x}_p$ are the coefficients of
$\hat{\boldsymbol{e}}_u$ and $\hat{e}_p$ with respect to the bases,
respectively; $F_v$ and $F_p$  are the vectors formed by the right-hand function of (\ref{eqn:error problem}) acting on the bases
of velocity and pressure, respectively. For any
$(\hat{\boldsymbol{v}},\hat{q})=\sum\limits_{j=1}^Nx_j\phi_j$,
$(\hat{\boldsymbol{w}},\hat{r})=\sum\limits_{j=1}^Ny_j\phi_j\in W_{k+d}$, we have

\vspace{-10pt}
\begin{align}\label{eqn:a bilinear}
a_1((\hat{\boldsymbol{v}},\hat{q}),(\hat{\boldsymbol{w}},\hat{r}))=\boldsymbol{y}^TM\boldsymbol{x},
\end{align}
where

\vspace{-10pt}
\begin{align}
\boldsymbol{x}=(x_1,\cdots,x_N)^{T}, \quad
\boldsymbol{y}=(y_1,\cdots,y_N)^{T}, \quad \mbox{and} \quad
M=\begin{bmatrix}
A & B  \\
-B^T & 0
\end{bmatrix}.
\end{align}

Let $\boldsymbol{x}_v$ be a vector composed of elements related to
velocity in $\boldsymbol{x}$, then it holds

\vspace{-10pt}
\begin{align*}
\|(\hat{\boldsymbol{v}},\hat{q})\|_V^2=| \hat{\boldsymbol{v}}|_{1,\Omega}^2+\|\hat{q}\|^2=\boldsymbol{x}_v^T A
\boldsymbol{x}_v+\|\hat{q}\|^2.
\end{align*}
Let $D_v$ be the diagonal matrix with the same diagonal as $A$ and
$M_v$ be

\vspace{-10pt}
\begin{align}\label{eqn:define Du}
M_v=\begin{bmatrix}
D_v & B  \\
-B^T & 0
\end{bmatrix}.
\end{align}
Define

\vspace{-10pt}
\begin{align}\label{eqn:a3 bilinear}
a_2((\hat{\boldsymbol{v}},\hat{q}),(\hat{\boldsymbol{w}},\hat{r}))=\boldsymbol{y}^TM_v\boldsymbol{x}
\end{align}
and norms

\vspace{-10pt}
\begin{align}
\label{eqn:norm associate with a3}
\|(\hat{\boldsymbol{v}},\hat{q})\|_D^2&=\boldsymbol{x}_v^T D_v \boldsymbol{x}_v+\|\hat{q}\|^2,\\
\label{eqn:norm associate with a3 about velocity}
|\hat{\boldsymbol{v}}|_D^2&=\boldsymbol{x}_v^T D_v \boldsymbol{x}_v.
\end{align}

Now, we are at the stage to present \textbf{the second error problem}: Find
$(\tilde{\boldsymbol{e}}_u,\tilde{e}_p)\in W_{k+d}$ such
that

\vspace{-10pt}
\begin{align}\label{eqn:error in com space easily}
a_2(\tilde{\boldsymbol{e}}_u,\tilde{e}_p),(\hat{\boldsymbol{v}},\hat{q}))
=f(\hat{\boldsymbol{v}})-a_1((\hat{\boldsymbol{u}},\hat{p}),(\hat{\boldsymbol{v}},\hat{q})), \quad
\forall~(\hat{\boldsymbol{v}},\hat{q})\in W_{k+d},
\end{align}
where $a_2(\cdot,\cdot)$ is specified in (\ref{eqn:a3 bilinear}).

For any $T \in \mathcal{T}$ and $(\hat{\boldsymbol{v}},\hat{q}) \in W_{k+d}$,
denote by $\{\varphi_{T,j}\}_{j=1}^{N_{v,T}}$ the basis functions of
velocity related to $T$, then
$\hat{\boldsymbol{v}}_T:=\hat{\boldsymbol{v}}_{|T}=\sum\limits_{j=1}^{N_{v,T}}x_{T,j}\varphi_{T,j}$
with $\{x_{T,j}\}_{j=1}^{N_{v,T}}$ being the coefficients. Let
$\hat{\boldsymbol{v}}_{T,j}:=x_{T,j}\varphi_{T,j}$, then
$\hat{\boldsymbol{v}}_T=\sum\limits_{j=1}^{N_{v,T}}\hat{\boldsymbol{v}}_{T,j}$. We can
rewrite $|\hat{\boldsymbol{v}}|_{1,\Omega}$ and $|\hat{\boldsymbol{v}}|_D$ as follows:

\vspace{-10pt}
\begin{align*}
|\hat{\boldsymbol{v}}|_{1,\Omega}^2&=\sum\limits_{T\in\mathcal{T}}|\hat{\boldsymbol{v}}_T|_{1,T}^2,\\
|\hat{\boldsymbol{v}}|_D^2&=\sum\limits_{T\in\mathcal{T}}
\sum_{j=1}^{N_{v,T}}|\hat{\boldsymbol{v}}_{T,j}|_{1,T}^2.
\end{align*}
We define the local norm of $|\cdot|_{D}$ by

\vspace{-10pt}
\begin{align}
\label{eqn:local norm D}
|\hat{\boldsymbol{v}}|_{D,T}=\sqrt{\sum_{j=1}^{N_{v,T}}|\hat{\boldsymbol{v}}_{T,j}|_{1,T}^2},
\end{align}
where $N_{v,T}$ is the number of basis functions of velocity in element $T$.
\begin{lemma}\label{lem:A and Du}
There exist two positive constants $\beta_1$ and $\beta_2$ independent
of $h$ such that

\vspace{-10pt}
\begin{align}\label{eqn:equivalence relation}
\beta_1\leq \frac{|\hat{\boldsymbol{w}}|_{1,T}^2}{|\hat{\boldsymbol{w}}|_{D,T}^2}\leq \beta_2, \qquad
\beta_1\leq \frac{|\hat{\boldsymbol{w}}|_{1,\Omega}^2}{|\hat{\boldsymbol{w}}|_D^2}\leq \beta_2,
\end{align}
for all $T\in\mathcal{T}$ and $\hat{\boldsymbol{w}}\in WV_{k+d+1}$.
\end{lemma}
\begin{proof}
We claim that there exist two positive constants $\beta_{1T}$ and
$\beta_{2T}$ independent of $h$ such that

\vspace{-10pt}
\begin{align}\label{eqn:equivalence relation in T}
\beta_{1T}\sum_{j=1}^{N_{v,T}}|\hat{\boldsymbol{w}}_{T,j}|_{1,T}^2\leq |\hat{\boldsymbol{w}}_T|_{1,T}^2\leq
\beta_{2T}\sum_{j=1}^{N_{v,T}}|\hat{\boldsymbol{w}}_{T,j}|_{1,T}^2, \quad T\in \mathcal{T}_h.
\end{align}
where $N_{v,T}$ is the number of basis functions of velocity in element $T$.

For the first inequality in (\ref{eqn:equivalence relation in T}),
divide $\Lambda=\{j\in N^{+}|1\leq j \leq N_{v,T}\}$ into two subsets $\Lambda=\Lambda_1
\cup \Lambda_2$ with $\Lambda_1\cap\Lambda_2=\emptyset$.
From Theorem 1 in \cite{Eijkhout1991}, it gets
that

\vspace{-10pt}
\begin{align}\label{eqn:stren Cauchy inequality 1}
(\sum_{j_1\in\Lambda_1}\nabla\hat{\boldsymbol{w}}_{T,j_1},\sum_{j_2\in\Lambda_2}\nabla\hat{\boldsymbol{w}}_{T,j_2}) \leq
\gamma_{v,T}|\sum_{j_1\in\Lambda_1}\hat{\boldsymbol{w}}_{T,j_1}|_{1,T}|\sum_{j_2\in\Lambda_2}\hat{\boldsymbol{w}}_{T,j_2}|_{1,T},
\end{align}
where $0\leq \gamma_{v,T}<1$ is independent of $h$. Using the  strengthened Cauchy inequality (\ref{eqn:stren Cauchy inequality 1}) and
Cauchy-Schwarz inequality, we deduce

\vspace{-10pt}
\begin{align*}
|\hat{\boldsymbol{w}}_T|_{1,T}^2&=|\sum_{j=1}^{N_{v,T}}\hat{\boldsymbol{w}}_{T,j}|_{1,T}^2
=(\sum_{j=1}^{N_{v,T}}\nabla\hat{\boldsymbol{w}}_{T,j},\sum_{j=1}^{N_{v,T}}\nabla\hat{\boldsymbol{w}}_{T,j})\\
&=|\hat{\boldsymbol{w}}_{T,1}|_{1,T}^2+|\sum_{j=2}^{N_{v,T}}\hat{\boldsymbol{w}}_{T,j}|_{1,T}^2
+2(\nabla\hat{\boldsymbol{w}}_{T,1},\sum_{j=2}^{N_{v,T}}\nabla\hat{\boldsymbol{w}}_{T,j})\\
&\geq |\hat{\boldsymbol{w}}_{T,1}|_{1,T}^2+|\sum_{j=2}^{N_{v,T}}\hat{\boldsymbol{w}}_{T,j}|_{1,T}^2
-2\gamma_{v,T}|\hat{\boldsymbol{w}}_{T,1}|_{1,T}|
\sum_{j=2}^{N_{v,T}}\hat{\boldsymbol{w}}_{T,j}|_{1,T}\\
&\geq (1-\gamma_{v,T})|\hat{\boldsymbol{w}}_{T,1}|_{1,T}^2+(1-\gamma_{v,T})
|\sum_{j=2}^{N_{v,T}}\hat{\boldsymbol{w}}_{T,j}|_{1,T}^2.
\end{align*}
By a similar argument, we obtain

\vspace{-10pt}
\begin{align*}
|\hat{\boldsymbol{w}}_T|_{1,T}^2=|\sum_{j=1}^{N_{v,T}}\hat{\boldsymbol{w}}_{T,j}|_{1,T}^2\geq
\sum_{j=1}^{N_{v,T}}(1-\gamma_{v,T})^{j}|\hat{\boldsymbol{w}}_{T,j}|_{1,T}^2 \geq
(1-\gamma_{v,T})^{N_{v,T}}\sum_{j=1}^{N_{v,T}}|\hat{\boldsymbol{w}}_{T,j}|_{1,T}^2,
\end{align*}
which implies the first inequality in (\ref{eqn:equivalence relation
in T}) with $\beta_{1T}=(1-\gamma_{v,T})^{N_{v,T}}$.

The second inequality in (\ref{eqn:equivalence relation in T})
follows from the Cauchy-Schwarz inequality with $\beta_{2T}=N_{v,T}$.
Therefore, the claim (\ref{eqn:equivalence relation in T}) holds.
Summing up (\ref{eqn:equivalence relation in T}) overall $T\in
\mathcal{T}$ and noting

\vspace{-10pt}
\begin{align*}
\frac{|\hat{\boldsymbol{w}}|_W^2}{|\hat{\boldsymbol{w}}|_D^2}
=\frac{\sum\limits_{T\in\mathcal{T}_h}|\hat{\boldsymbol{w}}_T|_{1,T}^2}
{\sum\limits_{T\in\mathcal{T}}\sum_{j=1}^{N_{v,T}}|\hat{\boldsymbol{w}}_{T,j}|_{1,T}^2},
\end{align*}
we arrive at the conclusion (\ref{eqn:equivalence relation}) with
$\beta_{1}=\min\limits_{T \in \mathcal{T}}(1-\gamma_{v,T})^{N_{v,T}}$ and
$\beta_{2}=\max\limits_{T\in\mathcal{T}}{N_{v,T}}$.
\end{proof}

\begin{lemma}\label{lem:a3 continuous}
For any $(\hat{\boldsymbol{v}},\hat{q}),(\hat{\boldsymbol{w}},\hat{r})\in W_{k+d}$, we have

\vspace{-10pt}
\begin{align}\label{eqn:Cauchy-Schwarz inequality for a3}
a_2((\hat{\boldsymbol{v}},\hat{q}),(\hat{\boldsymbol{w}},\hat{r}))\leq \mathfrak{C}_2\|(\hat{\boldsymbol{w}},\hat{r})\|_V
\|(\hat{\boldsymbol{v}},\hat{q})\|_V,
\end{align}
where $\mathfrak{C}_2$ is a positive constant.
\end{lemma}
\begin{proof}
For any $(\hat{\boldsymbol{v}},\hat{q})=\sum\limits_{j=1}^Nx_j\phi_j$,
$(\hat{\boldsymbol{w}},\hat{r})=\sum\limits_{j=1}^Ny_j\phi_j\in W_{k+d}$, define
$\boldsymbol{x}=(x_1,\cdots,x_N)^{T}$ and
$\boldsymbol{y}=(y_1,\cdots,y_N)^{T}$. Let $\boldsymbol{x}_v$ and
$\boldsymbol{y}_v$ be vectors composed of elements related to velocity
in $\boldsymbol{x}$ and $\boldsymbol{y}$, respectively. Similarly, Let
$\boldsymbol{x}_p$ and $\boldsymbol{y}_p$ be vectors composed of elements
related to pressure in $\boldsymbol{x}$ and $\boldsymbol{y}$, respectively.
Then using (\ref{eqn:a3 bilinear})$\sim$(\ref{eqn:norm associate
with a3 about velocity}), Cauchy-Schwarz inequality, and
Lemma~\ref{lem:A and Du}, we have

\vspace{-10pt}
\begin{align*}
a_2((\hat{\boldsymbol{v}},\hat{q}),(\hat{\boldsymbol{w}},\hat{r}))&=\boldsymbol{y}^TM_v\boldsymbol{x}
=\begin{bmatrix}
\boldsymbol{y}_v^T & \boldsymbol{y}_p^T
\end{bmatrix}
\begin{bmatrix}
D_v & B  \\
-B^T & 0
\end{bmatrix}
\begin{bmatrix}
\boldsymbol{x}_v\\
\boldsymbol{x}_p
\end{bmatrix}\\
&=\boldsymbol{y}_v^TD_v\boldsymbol{x}_v-\boldsymbol{y}_p^TB^T\boldsymbol{x}_v+\boldsymbol{y}_v^TB\boldsymbol{x}_p\\
&\leq |\hat{\boldsymbol{v}}|_D|\hat{\boldsymbol{w}}|_D-(\nabla\cdot\hat{\boldsymbol{v}},\hat{r})+(\nabla\cdot\hat{\boldsymbol{w}},\hat{q})\\
&\leq |\hat{\boldsymbol{v}}|_D|\hat{\boldsymbol{w}}|_D+\|\nabla\cdot\hat{\boldsymbol{v}}\|\|\hat{r}\|+\|\nabla\cdot\hat{\boldsymbol{w}}\|\|\hat{q}\|\\
&\leq |\hat{\boldsymbol{v}}|_D|\hat{\boldsymbol{w}}|_D+\sqrt{d}|\hat{\boldsymbol{v}}|_{1,\Omega}\|\hat{r}\|+\sqrt{d}|\hat{\boldsymbol{w}}|_{1,\Omega}\|\hat{q}\|\\
&\leq |\hat{\boldsymbol{v}}|_D|\hat{\boldsymbol{w}}|_D+\sqrt{d}\sqrt{\beta_2}|\hat{\boldsymbol{v}}|_D\|\hat{r}\|+\sqrt{d}\sqrt{\beta_2}|\hat{\boldsymbol{w}}|_D\|\hat{q}\|\\
&\leq \sqrt{d|\hat{\boldsymbol{v}}|_D^2|\hat{\boldsymbol{w}}|_D^2+d^2\beta_2(|\hat{\boldsymbol{v}}|_D^2\|\hat{r}\|^2+|\hat{\boldsymbol{w}}|_D^2\|\hat{q}\|^2)}\\
&\leq \mathfrak{C}_2\|(\hat{\boldsymbol{v}},\hat{q})\|_D\|(\hat{\boldsymbol{w}},\hat{r})\|_D,
\end{align*}
where $\mathfrak{C}_2=\max{(\sqrt{d},d\sqrt{\beta_2})}$.
\end{proof}

\begin{lemma}\label{lem:a3 inf sup}
The bilinear form $a_2((\hat{\boldsymbol{v}},\hat{q}),(\hat{\boldsymbol{w}},\hat{r}))$ satisfies the estimate

\vspace{-10pt}
\begin{align*}
\inf_{(\hat{\boldsymbol{v}},\hat{q})\in W_{k+d}\backslash
\{0\}}\sup_{(\hat{\boldsymbol{w}},\hat{r})\in W_{k+d}\backslash
\{0\}}\frac{a_2((\hat{\boldsymbol{v}},\hat{q}),(\hat{\boldsymbol{w}},\hat{r}))}
{\|(\hat{\boldsymbol{v}},\hat{q})\|_D\|(\hat{\boldsymbol{w}},\hat{r})\|_D}
\geq\frac{(\mu\beta_1)^2}{(1+\mu\beta_1)^2},
\end{align*}
where $\mu$ and $\beta_1$ are the constants in
Lemma~\ref{lem:inf-sup in com space} and Lemma~\ref{lem:A
and Du}, respectively.
\end{lemma}
\begin{proof}
To prove the inequality, we choose an arbitrary but fixed element
$(\hat{\boldsymbol{v}},\hat{q})\in WP_{k+d}\backslash \{0\}$. Due to Lemma~\ref{lem:inf-sup in com space},
there is a velocity field $\hat{\boldsymbol{w}}_{\hat{q}}\in
WV_{k+d+1}$ with $|\hat{\boldsymbol{w}}_{\hat{q}}|_D=1$ such that

\vspace{-10pt}
\begin{align*}
\sum_{T\in\mathcal{T}}\int_T\hat{q}\nabla\cdot
\hat{\boldsymbol{w}}_{\hat{q}}dx=\int_\Omega \hat{q}\nabla\cdot \hat{\boldsymbol{w}}_{\hat{q}}dx\geq \mu
\|\hat{q}\|.
\end{align*}
By using Cauchy-Schwartz inequality, Lemma
\ref{lem:A and Du}, Lemma \ref{lem:inf-sup in com space},
and noting $|\hat{\boldsymbol{w}}_{\hat{q}}|_D=1$, we therefore obtain for every $\delta>0$,

\vspace{-10pt}
\begin{align*}
&a_2((\hat{\boldsymbol{v}},\hat{q}),(\hat{\boldsymbol{v}}-\delta\|\hat{q}\|\hat{\boldsymbol{w}}_{\hat{q}},\hat{q}))\\
=&a_2((\hat{\boldsymbol{v}},\hat{q}),(\hat{\boldsymbol{v}},\hat{q}))
-\delta\|\hat{q}\|a_2((\hat{\boldsymbol{v}},\hat{q}),(\hat{\boldsymbol{w}}_{\hat{q}},0))\\
=&|\hat{\boldsymbol{v}}|_D^2-\delta\|\hat{q}\|\boldsymbol{y}_v^TD_v\boldsymbol{x}_v
+\delta\|\hat{q}\|\sum_{T\in\mathcal{T}}\int \hat{q} \nabla\cdot \hat{\boldsymbol{w}}_{\hat{q}}\\
\geq& |\hat{\boldsymbol{v}}|_D^2-\delta |\hat{\boldsymbol{v}}|_D\|\hat{q}\|+\delta \mu\|\hat{q}\|^2|\hat{\boldsymbol{w}}_{\hat{q}}|_{1,\Omega}\\
\geq& |\hat{\boldsymbol{v}}|_D^2-\delta |\hat{\boldsymbol{v}}|_D\|\hat{q}\|+\delta\mu\beta_1\|\hat{q}\|^2\\
\geq&
(1-\frac{\delta}{2\mu\beta_1})|\hat{\boldsymbol{v}}|_D^2+\frac{1}{2}\delta\mu\beta_1\|\hat{q}\|^2,
\end{align*}
where $\boldsymbol{x}_v=(x_1,x_2,\cdots,x_{N_v})^T$ and
$\boldsymbol{y}_v=(y_1,y_2,\cdots,y_{N_v})^T$ are such that
$\hat{\boldsymbol{v}}=\sum\limits_{j=1}^{N_{v}}x_j\varphi_j,
\hat{\boldsymbol{w}}_{\hat{q}}=\sum\limits_{j=1}^{N_{v}}y_j\varphi_j\in
WV_{k+d+1}$.

Similar to the proof in Lemma~\ref{lem:inf-sup B}, the choice of
$\delta=\frac{2\mu\beta_1}{1+(\mu\beta_1)^2}$ yields

\vspace{-10pt}
\begin{align*}
a_2((\hat{\boldsymbol{v}},\hat{q}),(\hat{\boldsymbol{v}}-\delta\|\hat{q}\|\hat{\boldsymbol{w}}_{\hat{q}},\hat{q}))\geq
\frac{(\mu\beta_1)^2}{1+(\mu\beta_1)^2}\|(\hat{\boldsymbol{v}},\hat{q})\|_D^2,
\end{align*}
and

\vspace{-10pt}
\begin{align*}
\|(\hat{\boldsymbol{v}}-\delta\|\hat{q}\|\hat{\boldsymbol{w}}_{\hat{q}},\hat{q})\|_D\leq
\frac{(1+\mu\beta_1)^2}{1+(\mu\beta_1)^2}\|(\hat{\boldsymbol{v}},\hat{q})\|_D.
\end{align*}
Then we arrive at

\vspace{-10pt}
\begin{align*}
\sup_{(\hat{\boldsymbol{w}},\hat{r})\in
W_{k+d}\backslash\{0\}}\frac{a_2((\hat{\boldsymbol{v}},\hat{q}),(\hat{\boldsymbol{w}},\hat{r}))}{\|(\hat{\boldsymbol{v}},\hat{q})\|_D\|(\hat{\boldsymbol{w}},\hat{r})\|_D}\geq
\frac{a_2((\hat{\boldsymbol{v}},\hat{q}),(\hat{\boldsymbol{v}}-\delta\|\hat{q}\|\hat{\boldsymbol{w}}_{\hat{q}},\hat{q}))}{\|(\hat{\boldsymbol{v}},\hat{q})\|_D
\|(\hat{\boldsymbol{v}}-\delta\|\hat{q}\|\hat{\boldsymbol{w}}_{\hat{q}},\hat{q})\|_D}
\geq\frac{(\mu\beta_1)^2}{(1+\mu\beta_1)^2}.
\end{align*}
Since $(\hat{\boldsymbol{v}},\hat{q})\in W_{k+d}\backslash\{0\}$ is arbitrary, this
completes the proof.
\end{proof}

Using Lemma \ref{lem:a3 continuous}, Lemma \ref{lem:a3 inf sup}, and a proof similar to that of Theorem
\ref{thm:unique of error problem}, we have the following conclusion.
\begin{theorem}
The finite element scheme (\ref{eqn:error in com space
easily}) has a unique solution.
\end{theorem}

\begin{lemma}\label{lem:e_h estimate e_h hat}
Let $(\hat{\boldsymbol{e}}_u,\hat{e}_p)$ and
$(\tilde{\boldsymbol{e}}_u,\tilde{e}_p)$ be the solutions
of (\ref{eqn:error problem}) and (\ref{eqn:error in com space
easily}), respectively.

\vspace{-10pt}
\begin{align}\label{eqn:two error estimations in com space}
\frac{(\mu\beta_1)^2}{\mathfrak{C}_1(1+\mu\beta_1)^2\sqrt{\beta_2+1}}\|(\tilde{\boldsymbol{e}}_u,\tilde{e}_p)\|_D\leq
\|(\hat{\boldsymbol{e}}_u,\hat{e}_{ph})\|_V
\leq\frac{\mathfrak{C}_2\sqrt{1+\beta_1}(1+\mu)^2}{\sqrt{\beta_1}\mu^2}\|(\tilde{\boldsymbol{e}}_u,\tilde{e}_p)\|_D,
\end{align}
where $\|\cdot\|_V$ and $\|\cdot\|_D$ are defined in (\ref{eqn:norm
V}) and (\ref{eqn:norm associate with a3}),
respectively. The constants $\mathfrak{C}_1,\mathfrak{C}_2,\beta_1,\beta_2$, and $\mu$ are defined in (\ref{eqn:B continous}), (\ref{eqn:Cauchy-Schwarz inequality for a3}), (\ref{eqn:equivalence relation in T}), and (\ref{eqn:inf-sup in com space}).
\end{lemma}
\begin{proof}
It follows from (\ref{eqn:error problem}) and (\ref{eqn:error
in com space easily}) that

\vspace{-10pt}
\begin{align}\label{eqn:bh=ah}
a_2((\tilde{\boldsymbol{e}}_{u},\tilde{e}_{p}),(\hat{\boldsymbol{v}},\hat{q}))
=a_1((\hat{\boldsymbol{e}}_{u},\hat{e}_{ph}),(\hat{\boldsymbol{v}},\hat{q})), \quad
\forall~(\hat{\boldsymbol{v}},\hat{q})\in W_{k+d}.
\end{align}
Using (\ref{eqn:bh=ah}), Lemma~\ref{lem:A and Du}, and Lemma
\ref{lem:a3 inf sup}, we obtain

\vspace{-10pt}
\begin{align*}
\frac{(\mu\beta_1)^2}{(1+\mu\beta_1)^2}\|(\tilde{\boldsymbol{e}}_u,\tilde{e}_p)\|_D&
\leq\sup_{(\hat{\boldsymbol{v}},\hat{q})\in W_{k+d}\backslash\{0\}}\frac{a_2((\tilde{\boldsymbol{e}}_u,\tilde{e}_p),
(\hat{\boldsymbol{v}},\hat{q}))}{\|(\hat{\boldsymbol{v}},\hat{q})\|_D}\\
&=\sup_{(\hat{\boldsymbol{v}},\hat{q})\in W_{k+d}\backslash\{0\}}\frac{a_{1}((\hat{\boldsymbol{e}}_u,\hat{e}_p),(\hat{\boldsymbol{v}},\hat{q}))}
{\|(\hat{\boldsymbol{v}},\hat{q})\|_D}\\
&\leq\sup_{(\hat{\boldsymbol{v}},\hat{q})\in W_{k+d}\backslash\{0\}}\frac{\mathfrak{C}_1\|(\hat{\boldsymbol{e}}_u,\hat{e}_p)\|_V\|(\hat{\boldsymbol{v}},\hat{q})\|_V}{\|(\hat{\boldsymbol{v}},\hat{q})\|_D}\\
&\leq\sup_{(\hat{\boldsymbol{v}},\hat{q})\in W_{k+d}\backslash\{0\}}\frac{\mathfrak{C}_1\|(\hat{\boldsymbol{e}}_u,\hat{e}_p)\|_V\sqrt{|\hat{\boldsymbol{v}}|_{1,\Omega}^2+\|\hat{q}\|^2}}
{\sqrt{|\hat{\boldsymbol{v}}|_D^2+\|\hat{q}\|^2}}\\
&\leq\sup_{(\hat{\boldsymbol{v}},\hat{q})\in W_{k+d}\backslash\{0\}}\frac{\mathfrak{C}_1\|(\hat{\boldsymbol{e}}_u,\hat{e}_p)\|_V\sqrt{\beta_2|\hat{\boldsymbol{v}}|_D^2+\|\hat{q}\|^2}}
{\sqrt{|\hat{\boldsymbol{v}}|_D^2+\|\hat{q}\|^2}}\\
&\leq\mathfrak{C}_1\sqrt{\beta_2+1}\|(\hat{\boldsymbol{e}}_u,\hat{e}_p)\|_V,
\end{align*}
which implies the first inequality in (\ref{eqn:two error
estimations in com space}).

Similarly, using (\ref{eqn:bh=ah}) and Lemma \ref{lem:inf-sup in com space}, we have

\vspace{-10pt}
\begin{align*}
\frac{\mu^2}{(1+\mu)^2}\|(\hat{\boldsymbol{e}}_u,\hat{e}_p)\|_V&
\leq\sup_{(\hat{\boldsymbol{v}},\hat{q})\in W_{k+d}\backslash\{0\}}\frac{a_{1}((\hat{\boldsymbol{e}}_u,\hat{e}_p),
(\hat{\boldsymbol{v}},\hat{q}))}{\|(\hat{\boldsymbol{v}},\hat{q})\|_V}\\
&=\sup_{(\hat{\boldsymbol{v}},\hat{q})\in W_{k+d}\backslash\{0\}}\frac{a_2((\tilde{\boldsymbol{e}}_u,\tilde{e}_p),(\hat{\boldsymbol{v}},\hat{q}))}
{\|(\hat{\boldsymbol{v}},\hat{q})\|_V}\\
&\leq\sup_{(\hat{\boldsymbol{v}},\hat{q})\in W_{k+d}\backslash\{0\}}\frac{\mathfrak{C}_2\|(\tilde{\boldsymbol{e}}_u,\tilde{e}_p)\|_D\|(\hat{\boldsymbol{v}},\hat{q})\|_D}
{\|(\hat{\boldsymbol{v}},\hat{q})\|_V}\\
&\leq\sup_{(\hat{\boldsymbol{v}},\hat{q})\in W_{k+d}\backslash\{0\}}\frac{\mathfrak{C}_2\|(\tilde{\boldsymbol{e}}_u,\tilde{e}_p)\|_D
\sqrt{|\hat{\boldsymbol{v}}|_D^2+\|\hat{q}\|^2}}
{\sqrt{|\hat{\boldsymbol{v}}|_{1,\Omega}^2+\|\hat{q}\|^2}}\\
&\leq\sup_{(\hat{\boldsymbol{v}},\hat{q})\in W_{k+d}\backslash\{0\}}\frac{\mathfrak{C}_2\|(\tilde{\boldsymbol{e}}_u,\tilde{e}_p)\|_D
\sqrt{\frac{|\hat{\boldsymbol{v}}|_{1,\Omega}^2}{\beta_1}+\|\hat{q}\|^2}}
{\sqrt{|\hat{\boldsymbol{v}}|_{1,\Omega}^2+\|\hat{q}\|^2}}\\
&\leq\frac{\mathfrak{C}_2\sqrt{1+\beta_1}}{\sqrt{\beta_1}}\|(\tilde{\boldsymbol{e}}_u,\tilde{e}_p)\|_D,
\end{align*}
which implies the second inequality in (\ref{eqn:two error
estimations in com space}).
\end{proof}

Combining Theorem \ref{thm:equivalence hat e} and Lemma
\ref{lem:e_h estimate e_h hat}, we obtain the following lower
and upper bounds related to
$\|(\tilde{\boldsymbol{e}}_u,\tilde{e}_p)\|_D$.
\begin{theorem}\label{theo:diag velocity approx error}
Let $(\boldsymbol{u},p)$, $(\hat{\boldsymbol{u}},\hat{p})$, and
$(\tilde{\boldsymbol{e}}_u,\tilde{e}_p)$ be the solutions
of (\ref{eqn:variational form}),(\ref{eqn:error problem}), and (\ref{eqn:error in com space
easily}), respectively. There are constants $\tilde{\mathfrak{C}}_*=\frac{\mu^4\beta_1^2}
{2\mathfrak{C}_1^2(1+\mu)^2(1+\beta_1\mu)^2\sqrt{1+\beta_2}}$ and
$\tilde{\mathfrak{C}}^*=\frac{\mathfrak{C}_1\mathfrak{C}_2\sqrt{1+\beta_1}(1+\mu)^2}
{\mathfrak{c}_1\sqrt{\beta_1}\mu^2}$ such that

\vspace{-10pt}
\begin{align}
\tilde{\mathfrak{C}}_*\|(\tilde{\boldsymbol{e}}_{u},\tilde{e}_p)\|_D
+\frac{1}{2\sqrt{d}}\|\nabla\cdot\hat{\boldsymbol{u}}\|
\leq\|(\boldsymbol{u}-\hat{\boldsymbol{u}},p-\hat{p})\|_V\leq \tilde{\mathfrak{C}}^*\|(\tilde{\boldsymbol{e}}_{u},\tilde{e}_p)\|_D
+\frac{1}{\mathfrak{c}_1}\|\nabla\cdot\hat{\boldsymbol{u}}\|
+\frac{C_{\mathcal{T}}}{\mathfrak{c}_1}osc(\boldsymbol{f}),
\end{align}
where $\|\cdot\|_V$ and $\|\cdot\|_D$ are defined in (\ref{eqn:norm
V}) and (\ref{eqn:norm associate with a3}),
respectively. The constants $\mathfrak{C}_1, \mathfrak{C}_2, \mathfrak{c}_1, \mu, \beta_1,\beta_2$, and $C_{\mathcal{T}}$ are defined in  (\ref{eqn:B continous}), (\ref{eqn:Cauchy-Schwarz inequality for a3}) (\ref{eqn:a1 orign inf-sup}), (\ref{eqn:inf-sup in com space}), (\ref{eqn:equivalence relation in T}), and Lemma~\ref{lem:quasi-interpolant v}, respectively.
\end{theorem}

\subsection{Diagonalization with respect to Pressure}
Recall that $\{\varphi_j\}_{j=1}^{N_{v}}$ and
$\{\psi_j\}_{j=1}^{N_{p}}$ are the bases in space $W_{k+d}$ for
velocity and pressure, respectively. For
$\tilde{\boldsymbol{e}}_u=\sum_{j=1}^{N_v}\tilde{x}_{u,j}\varphi_j$
and $\tilde{e}_p=\sum_{j=1}^{N_p}\tilde{x}_{p,j}\psi_j$,
rewrite (\ref{eqn:error in com space easily}) in a matrix form

\vspace{-10pt}
\begin{align}\label{eqn:matrix form of a3}
\begin{bmatrix}
D_v & B  \\
-B^T & 0
\end{bmatrix}
\begin{bmatrix}
\tilde{\boldsymbol{x}}_u  \\
\tilde{\boldsymbol{x}}_p
\end{bmatrix}=
\begin{bmatrix}
F_v   \\
F_p
\end{bmatrix},
\end{align}
where

\vspace{-10pt}
\begin{align*}
&\tilde{\boldsymbol{x}}_u=(\tilde{x}_{u,1},\tilde{x}_{u,2},\cdots,\tilde{x}_{u,N_v})^T,
&F_v=(F_{v,1},F_{v,2},\cdots,F_{v,N_v})^T,\\
&\tilde{\boldsymbol{x}}_p=(\tilde{x}_{p,1},\tilde{x}_{p,2},\cdots,\tilde{x}_{p,N_p})^T,
&F_p=(F_{p,1},F_{p,2},\cdots,F_{p,N_v})^T.
\end{align*}
Here, $F_{v,j}$ and $F_{p,j}$ are defined by

\vspace{-10pt}
\begin{align}
\label{eqn:Fvj}
F_{v,j}&=f(\varphi_j)-a_1((\hat{\boldsymbol{u}},\hat{p}),(\varphi_j,0)), \quad j=1,2,\cdots,N_v,\\
\label{eqn:Fpj}
F_{p,j}&=-a_1((\hat{\boldsymbol{u}},\hat{p}),(0,\psi_j)), \quad j=1,2,\cdots,N_p.
\end{align}
After a simple calculation, we have

\vspace{-10pt}
\begin{align}
\label{eqn:xu}
&D_v\tilde{\boldsymbol{x}}_u+B\tilde{\boldsymbol{x}}_p=F_v,\\
\label{eqn:xp}
&B^TD_v^{-1}B\tilde{\boldsymbol{x}}_p=F_p+B^TD_v^{-1}F_v.
\end{align}

The inverse of the matrix $D_v$ is easy to calculate because it is a
diagonal matrix. If we get $\tilde{\boldsymbol{x}}_p$ by solving
(\ref{eqn:xp}), $\tilde{\boldsymbol{x}}_u$ is easy to get by
(\ref{eqn:xu}). Let $D_p=diag(B^TD_v^{-1}B)$, which is the diagonal
matrix with the same diagonal as $B^TD_v^{-1}B$. Let
$c_s=\max\limits_{T\in \mathcal{T}}N_{p,T}$, which is the maximum
number of basis functions of pressure for each element. Then replacing
$B^TD_v^{-1}B$ with $c_sD_p$ in (\ref{eqn:xp}), we get

\vspace{-10pt}
\begin{align}
\label{eqn:xu hat}
&D_v\bar{\boldsymbol{x}}_u+B\bar{\boldsymbol{x}}_p=F_v,\\
\label{eqn:xp hat} &c_sD_p\bar{\boldsymbol{x}}_p=F_p+B^TD_v^{-1}F_v,
\end{align}
where
$\bar{\boldsymbol{x}}_u=(\bar{x}_{u,1},\bar{x}_{u,2},\cdots,\bar{x}_{u,N_v})$
and
$\bar{\boldsymbol{x}}_p=(\bar{x}_{p,1},\bar{x}_{p,2},\cdots,\bar{x}_{p,N_p})$.

Equations (\ref{eqn:xu hat}) and (\ref{eqn:xp hat}) are equivalent to

\vspace{-10pt}
\begin{align*}
&D_v\bar{\boldsymbol{x}}_u+B\bar{\boldsymbol{x}}_p=F_v,\\
&-B^T\bar{\boldsymbol{x}}_u+(c_sD_p-B^TD_v^{-1}B)\bar{\boldsymbol{x}}_p=F_p,
\end{align*}
whose matrix form is

\vspace{-10pt}
\begin{align*}
\begin{bmatrix}
D_v & B  \\
-B^T & c_sD_p-B^TD_v^{-1}B
\end{bmatrix}
\begin{bmatrix}
\bar{\boldsymbol{x}}_u  \\
\bar{\boldsymbol{x}}_p
\end{bmatrix}=
\begin{bmatrix}
F_v   \\
F_p
\end{bmatrix}.
\end{align*}

For any $(\hat{\boldsymbol{v}},\hat{q})=\sum\limits_{j=1}^Nx_j\phi_j$ and
$(\hat{\boldsymbol{w}},\hat{r})=\sum\limits_{j=1}^Ny_j\phi_j\in W_{k+d}$, we define

\vspace{-10pt}
\begin{align}\label{eqn:a4 bilinear}
a_3((\hat{\boldsymbol{v}},\hat{q}),(\hat{\boldsymbol{w}},\hat{r}))=\boldsymbol{y}^TM_{vp}\boldsymbol{x},
\end{align}
where

\vspace{-10pt}
\begin{align*}
\boldsymbol{y}=(y_1,\cdots,y_N)^{T}, \quad M_{vp}=
\begin{bmatrix}
D_v & B  \\
-B^T & c_sD_p-B^TD_v^{-1}B
\end{bmatrix}, \quad \mbox{and}  \quad
\boldsymbol{x}=(x_1,\cdots,x_N)^{T}.
\end{align*}

It is time to present \textbf{the third error problem}: Find $\{\bar{\boldsymbol{e}}_u,\bar{e}_p\}\in
W_{k+d}$ with
$\bar{\boldsymbol{e}}_u=\sum_{j=1}^{N_v}\bar{x}_{u,j}\varphi_j$
and $\bar{e}_p=\sum_{j=1}^{N_p}\bar{x}_{p,j}\psi_j$ such
that

\vspace{-10pt}
\begin{align}\label{eqn:the third error problem}
a_3((\bar{\boldsymbol{e}}_u,\bar{e}_p),(\hat{\boldsymbol{v}},\hat{q}))
=f(\hat{\boldsymbol{v}})-a_1((\hat{\boldsymbol{u}},\hat{p}),(\hat{\boldsymbol{v}},\hat{q})), \quad
\forall~(\hat{\boldsymbol{v}},\hat{q})\in W_{k+d}.
\end{align}

\begin{remark}
Equations (\ref{eqn:xu hat}) and (\ref{eqn:xp hat}) are equivalent to
(\ref{eqn:the third error problem}), but they are used in different
ways. Obviously, (\ref{eqn:xu hat}) and (\ref{eqn:xp hat}) are
easier to calculate. In section \ref{sec:Adaptive finite element method}, the global and local estimators will be
generated from (\ref{eqn:xu hat}) and (\ref{eqn:xp hat}). However,
(\ref{eqn:the third error problem}) is essential in the proof of
equivalence. Therefore, we use (\ref{eqn:xu hat}) and (\ref{eqn:xp
hat}) for the numerical computation and (\ref{eqn:the third error problem}) for the theoretical analysis.
\end{remark}

Because matrix $D_v$ and $D_p$ are diagonal matrices in (\ref{eqn:xu
hat}) and (\ref{eqn:xp hat}), the existence and uniqueness of finite element scheme (\ref{eqn:the third error problem})
are obvious.
\begin{theorem}
The finite element scheme (\ref{eqn:the third error problem}) has a unique solution.
\end{theorem}

Next, we turn our attention to the discrete pressure space. Two new norms will be defined.
We still use $\{\psi_j\}_{j=1}^{N_{p}}$ to denote the basis functions of pressure in $W_{k+d}$. For $\hat{q}=\sum_{j=1}^{N_p}x_p^j\psi_j$
and $\hat{r}=\sum_{j=1}^{N_p}y_p^j\psi_j$, define two bilinear forms

\vspace{-10pt}
\begin{align*}
E_{31}(\hat{q},\hat{r})=\boldsymbol{y}_p^TB^TD_v^{-1}B\boldsymbol{x}_p, \quad
E_{32}(\hat{q},\hat{r})=\boldsymbol{y}_p^TD_p\boldsymbol{x}_p,
\end{align*}
and norms

\vspace{-10pt}
\begin{align*}
\|\hat{q}\|_B^2=\sqrt{E_{31}(\hat{q},\hat{q})},\quad \|\hat{q}\|_P^2=\sqrt{E_{32}(\hat{q},\hat{q})},
\end{align*}
where $\boldsymbol{x}_p=(x_{p,1},\cdots,x_{p,N_p})^{T}$ and
$\boldsymbol{y}_p=(y_{p,1},\cdots,y_{p,N_p})^{T}$.
The next two lemmas will establish some inequalities related to the three pressure norms $\|\cdot\|_B, \|\cdot\|_P$, and $\|\cdot\|$.

\begin{lemma}\label{lem:B and P}
There exist two positive constants $c_i$ and $c_s$ independent of
$h$ such that

\vspace{-10pt}
\begin{align}\label{eqn:equivalence relation B and P}
c_i\leq \frac{\|\hat{q}\|_B^2}{\|\hat{q}\|_P^2}\leq c_s,\quad \forall \hat{q} \in WP_{k+d},
\end{align}
where $c_s$ is the same as in (\ref{eqn:xp hat}).
\end{lemma}
\begin{proof}
For any $T \in \mathcal{T}$, denote by
$\{\psi_{T,j}\}_{j=1}^{N_{p,T}}$ the basis functions of pressure related to $T$,
then $\hat{q}_T:=\hat{q}_{|T}=\sum\limits_{j=1}^{N_{p,T}}x_{T,j}\psi_{T,j}$ with
$\{x_{T,j}\}_{j=1}^{N_{p,T}}$ being the coefficients. Let
$\hat{q}_{T,j}:=x_{T,j}\psi_{T,j}$, then
$\hat{q}_T=\sum\limits_{j=1}^{N_{p,T}}\hat{q}_{T,j}$. We claim that there exist two
positive constants $c_{iT}$ and $c_{sT}$, independent of $h$,
such that

\vspace{-10pt}
\begin{align}\label{eqn:equivalence relation B and P in K}
c_{iT}\sum_{j=1}^{N_{p,T}}\|\hat{q}_{T,j}\|_B^2\leq \|\hat{q}_T\|_B^2\leq
c_{sT}\sum_{j=1}^{N_{p,T}}\|\hat{q}_{T,j}\|_B^2, \quad T\in \mathcal{T}.
\end{align}

For the first inequality in (\ref{eqn:equivalence relation B and P
in K}), devide $\Lambda=\{j\in N^{+}~\big|~1\leq j \leq N_{p,T}\}$ into two subsets
$\Lambda=\Lambda_1 \cup \Lambda_2$ with
$\Lambda_1\cap\Lambda_2=\emptyset$. From Theorem 1 in
\cite{Eijkhout1991}, it gets that

\vspace{-10pt}
\begin{align}\label{eqn:stren Cauchy inequality bar}
E_{31}(\sum_{j\in\Lambda_1}\hat{q}_{T,j},\sum_{\ell\in\Lambda_2}\hat{q}_{T,\ell}) \leq
\gamma_{p,T}\|\sum_{j\in\Lambda_1}\hat{q}_{T,j}\|_B\|\sum_{\ell\in\Lambda_2}\hat{q}_{T,\ell}\|_B,
\end{align}
where $0\leq \gamma_{p,T}<1$ is independent of $h$. Using the strengthened Cauchy inequality (\ref{eqn:stren Cauchy inequality bar}), we deduce

\vspace{-10pt}
\begin{align*}
\|\sum_{j=1}^{N_{p,T}}\hat{q}_{T,j}\|_B^2&=\|\hat{q}_T\|_B^2
=E_{31}(\sum_{j=1}^{N_{p,T}}\hat{q}_{T,j},\sum_{j=1}^{N_{p,T}}\hat{q}_{T,j})\\
&=\|\hat{q}_{T,1}\|_B^2+\|\sum_{j=2}^{N_{p,T}}\hat{q}_{T,j}\|_B^2
+2E_{31}(\hat{q}_{T,1},\sum_{j=2}^{N_{p,T}}\hat{q}_{T,j})\\
&\geq \|\hat{q}_{T,1}\|_B^2+\|\sum_{j=2}^{N_{p,T}}\hat{q}_{T,j}\|_B^2
-2\gamma_{p,T}\|\hat{q}_{T,1}\|_B\|
\sum_{j=2}^{N_{p,T}}\hat{q}_{T,j}\|_B\\
&\geq
(1-\gamma_{p,T})\|\hat{q}_{T,1}\|_B^2+(1-\gamma_{p,T})\|\sum_{j=2}^{N_{p,T}}\hat{q}_{T,j}\|_B^2.
\end{align*}
By a similar argument, we obtain

\vspace{-10pt}
\begin{align*}
\|\hat{q}_T\|_B^2=\|\sum_{j=1}^{N_{p,T}}\hat{q}_{T,j}\|_B^2\geq
\sum_{j=1}^{N_{p,T}}(1-\gamma_{p,T})^{j}\|\hat{q}_{T,j}\|_B^2 \geq
(1-\gamma_{p,T})^{N_{p,T}}\sum_{j=1}^{N_{p,T}}\|\hat{q}_{T,j}\|_B^2,
\end{align*}
which implies the first inequality in (\ref{eqn:equivalence relation B and P in K}) with $c_{iT}=(1-\gamma_{p,T})^{N_{p,T}}$.

The second inequality in (\ref{eqn:equivalence relation in T})
follows from the Cauchy-Schwarz inequality with $c_{sT}=N_{p,T}$.
Therefore, the claim (\ref{eqn:equivalence relation in T}) holds.
Summing up (\ref{eqn:equivalence relation in T}) over all $T\in
\mathcal{T}$ and noting

\vspace{-10pt}
\begin{align*}
\frac{\|\hat{q}\|_B^2}{\|\hat{q}\|_P^2}
=\frac{\sum\limits_{T\in\mathcal{T}}\|\hat{q}_T\|_B^2}
{\sum\limits_{T\in\mathcal{T}}\sum_{j=1}^{N_{p,T}}\|\hat{q}_{T,j}\|_B^2},
\end{align*}
we arrive at the conclusion (\ref{eqn:equivalence relation}) with
$c_{i}=\min\limits_{T \in \mathcal{T}}(1-\gamma_{p,T})^{N_{p,T}}$ and
$c_{s}=\max\limits_{T \in \mathcal{T}}N_{p,T}$.
\end{proof}

\begin{lemma}\label{lem:B and L2 norm}
For any $(\hat{\boldsymbol{v}},\hat{q})\in W_{k+d}$, we have

\vspace{-10pt}
\begin{align*}
\|\hat{q}\|_B\leq d(d+1)\beta_2\|\hat{q}\|
\end{align*}
where $d$ is the dimension and $\beta_2$ is defined in (\ref{eqn:equivalence relation}).
\end{lemma}
\begin{proof}
We continue to use $\{\varphi_j\}_{j=1}^{N_v}$ and
$\{\psi_j\}_{j=1}^{N_{p}}$ as the bases in $W_{k+d}$ for
velocity and pressure, respectively. Define a diagonal matrix
$\widetilde{D}$ whose elements are the square roots of the
corresponding elements of $D_v^{-1}$, and it is clear that
$D_v^{-1}=\widetilde{D}\widetilde{D}$. Let
$q=\sum_{j=1}^{N_p}x_j\psi_j$ and
$\boldsymbol{x}=(x_1,x_2,\cdots,x_{N_p})^T$, then

\vspace{-10pt}
\begin{align}\label{eqn:q B norm eqn 1}
\|\hat{q}\|_B^2&=\boldsymbol{x}^TB^TD_v^{-1}B\boldsymbol{x}=\boldsymbol{x}^TB^T\widetilde{D}\widetilde{D}B\boldsymbol{x}.
\end{align}
Let $\{d_j\}_{j=1}^{N_v}$ denote the diagonal elements of the matrix $\widetilde{D}$ whose dimension is $N_v$. Let $\boldsymbol{d}_j$ denote the $j$-th column of matrix $\widetilde{D}$ and $\hat{\boldsymbol{v}}_j=d_j\varphi_j$ and denoted by $T_{1},\cdots,T_{j_T}$ the elements related to $\varphi_j$ respectively. Then

\vspace{-10pt}
\begin{align}\label{eqn:basis and elements}
\sum_{i=1}^{j_T}|\hat{\boldsymbol{v}}_j|_{D,T_i}^2=|\hat{\boldsymbol{v}}_j|_D^2=\boldsymbol{d}_j^TD_v\boldsymbol{d}_j=1.
\end{align}

From (\ref{eqn:q B norm eqn 1}), (\ref{eqn:basis and elements}), Cauchy-Schwartz
inequality, and
Lemma~\ref{lem:A and Du}, we have

\vspace{-10pt}
\begin{align*}
\|\hat{q}\|_B^2
&=\sum_{j=1}^{N_v}(\boldsymbol{x}^TB^T\boldsymbol{d}_j)(\boldsymbol{d}_j^TB\boldsymbol{x})
=\sum_{j=1}^{N_v} (b(\hat{\boldsymbol{v}}_j,\hat{q}))^2
=\sum_{j=1}^{N_v}(\sum_{T\in \mathcal{T}_h}(\nabla\cdot \hat{\boldsymbol{v}}_j,\hat{q})_T)^2\\
&=\sum_{j=1}^{N_v}(\sum_{i=1}^{j_T}(\nabla\cdot\hat{\boldsymbol{v}}_j,\hat{q})_{T_i})^2
\leq\sum_{j=1}^{N_v}\sum_{i=1}^{j_T}\|\nabla\cdot\hat{\boldsymbol{v}}_j\|_{0,T_i}^2
\|\hat{q}\|_{0,T_i}^2\\
&\leq \sum_{j=1}^{N_v}\sum_{i=1}^{j_T}d|\hat{\boldsymbol{v}}_j|_{1,T_i}^2
\|\hat{q}\|_{0,T_i}^2\leq \sum_{j=1}^{N_v}\sum_{i=1}^{j_T}d\beta_2|\hat{\boldsymbol{v}}_j|_{D,T_i}^2
\|\hat{q}\|_{0,T_i}^2\\
&\leq \sum_{j=1}^{N_v}\sum_{i=1}^{j_T}d\beta_2
\|\hat{q}\|_{0,T_i}^2=d(d+1)\beta_2\sum_{T\in\mathcal{T}}
\|\hat{q}\|_{0,T}^2=d(d+1)\beta_2\|\hat{q}\|
\end{align*}
\end{proof}

\begin{lemma}\label{lem:a4 inf sup}
The bi-linear form $a_3((\hat{\boldsymbol{v}},\hat{q}),(\hat{\boldsymbol{w}},\hat{r}))$ satisfies
the estimates

\vspace{-10pt}
\begin{align*}
\inf_{(\hat{\boldsymbol{v}},\hat{q})\in W_{k+d}\backslash
\{0\}}\sup_{(\hat{\boldsymbol{w}},\hat{r})\in W_{k+d}\backslash
\{0\}}\frac{a_3((\hat{\boldsymbol{v}},\hat{q}),(\hat{\boldsymbol{w}},\hat{r}))}{\|(\hat{\boldsymbol{v}},\hat{q})\|_D\|(\hat{\boldsymbol{w}},\hat{r})\|_D}
\geq\frac{(\mu\beta_1)^2}{(1+\mu\beta_1)^2},
\end{align*}
where $\mu$ and $\beta_1$ are the constants in Lemma \ref{lem:inf-sup in com space} and Lemma \ref{lem:A and Du}, respectively.
\end{lemma}
\begin{proof}
To prove the inequality, we choose an arbitrary but fixed element
$(\hat{\boldsymbol{v}},\hat{q})\in W_{k+d}\backslash \{0\}$. Due to Lemma~\ref{lem:inf-sup in com space},
there is a velocity field $\hat{\boldsymbol{w}}_{\hat{q}}\in WP_{k+d}$ with $|\hat{\boldsymbol{w}}_{\hat{q}}|_D=1$ such that

\vspace{-10pt}
\begin{align*}
\sum_{T\in\mathcal{T}}\int_T\hat{q}\nabla\cdot
\hat{\boldsymbol{w}}_{\hat{q}}=\int_\Omega \hat{q}\nabla\cdot \hat{\boldsymbol{w}}_{\hat{q}}\geq \mu
\|\hat{q}\|.
\end{align*}
By using Cauchy-Schwartz inequality, Lemma
\ref{lem:A and Du}, Lemma \ref{lem:inf-sup in com space}, and
Lemma \ref{lem:B and P}, we therefore obtain for every $\delta>0$

\vspace{-10pt}
\begin{align*}
&a_3((\hat{\boldsymbol{v}},\hat{q}),(\hat{\boldsymbol{v}}-\delta\|\hat{q}\|\hat{\boldsymbol{w}}_{\hat{q}},\hat{q}))\\
=&a_3((\hat{\boldsymbol{v}},\hat{q}),(\hat{\boldsymbol{v}},\hat{q}))-\delta\|\hat{q}\|a_3((\hat{\boldsymbol{v}},\hat{q}),(\hat{\boldsymbol{w}}_{\hat{q}},0))\\
=&|\hat{\boldsymbol{v}}|_D^2+c_s\|\hat{q}\|_p-\|\hat{q}\|_B-\delta\|\hat{q}\|\boldsymbol{y}_v^TD_v\boldsymbol{x}_v
+\delta\|\hat{q}\|\sum_{T\in\mathcal{T}_h}\int_T \hat{q} \nabla\cdot \hat{\boldsymbol{w}}_{\hat{q}}\\
\geq& |\hat{\boldsymbol{v}}|_D^2-\delta |\hat{\boldsymbol{v}}|_D\|\hat{q}\|+\delta \mu\|\hat{q}\|^2|\hat{\boldsymbol{w}}_{\hat{q}}|_{1,\Omega}\\
\geq& |\hat{\boldsymbol{v}}|_D^2-\delta |\hat{\boldsymbol{v}}|_D\|\hat{q}\|+\delta \mu\beta_1\|\hat{q}\|^2\\
\geq&
(1-\frac{\delta}{2\mu\beta_1})|\hat{\boldsymbol{v}}|_D^2+\frac{1}{2}\delta\mu\beta_1\|\hat{q}\|^2,
\end{align*}
where $\boldsymbol{x}_v=(x_{v,1},\cdots,x_{v,N_v})^T$,
$\boldsymbol{x}_q=(x_{q,1},\cdots,x_{q,N_p})^T$, and
$\boldsymbol{y}_v=(y_{v,1},\cdots,y_{v,N_v})^T$. Let
$\hat{\boldsymbol{v}}=\sum\limits_{j=1}^{N_v}x_{v,j}\varphi_j$,
$\hat{q}=\sum\limits_{j=1}^{N_v}x_{q,j}\psi_j$, and
$\hat{\boldsymbol{w}}_{\hat{q}}=\sum\limits_{j=1}^{N_v}y_{v,j}\varphi_j$.

Similar to the proof of Lemma~\ref{lem:inf-sup B}, the choice
of $\delta=\frac{2\mu\beta_1}{1+(\mu\beta_1)^2}$
yields

\vspace{-10pt}
\begin{align*}
a_3((\hat{\boldsymbol{v}},\hat{q}),(\hat{\boldsymbol{v}}-\delta\|\hat{q}\|\hat{\boldsymbol{w}}_{\hat{q}},\hat{q}))\geq
\frac{(\mu\beta_1)^2}{1+(\mu\beta_1)^2}\|(\hat{\boldsymbol{v}},\hat{q})\|_D^2,
\end{align*}
and

\vspace{-10pt}
\begin{align*}
\|(\hat{\boldsymbol{v}}-\delta\|\hat{q}\|\hat{\boldsymbol{w}}_{\hat{q}},\hat{q})\|_D\leq\frac{(1+\mu\beta_1)^2}{1+(\mu\beta_1)^2}\|(\hat{\boldsymbol{v}},\hat{q})\|_D.
\end{align*}
Then we arrive at

\vspace{-10pt}
\begin{align*}
\sup_{(\hat{\boldsymbol{w}},\hat{r})\in
W_{k+d}\backslash\{0\}}\frac{a_3((\hat{\boldsymbol{v}},\hat{q}),(\hat{\boldsymbol{w}},\hat{r}))}{\|(\hat{\boldsymbol{v}},\hat{q})\|_D\|(\hat{\boldsymbol{w}},\hat{r})\|_D}\geq
\frac{a_3((\hat{\boldsymbol{v}},\hat{q}),(\hat{\boldsymbol{v}}-\delta\|\hat{q}\|\hat{\boldsymbol{w}}_{\hat{q}},\hat{q}))}{\|(\hat{\boldsymbol{v}},\hat{q})\|_D\|(\hat{\boldsymbol{v}}-\delta\|\hat{q}\|\hat{\boldsymbol{w}}_{\hat{q}},\hat{q})\|_D}
\geq\frac{(\mu\beta_1)^2}{(1+\mu\beta_1)^2}.
\end{align*}
Since $(\hat{\boldsymbol{v}},\hat{q})\in W_{k+d}\backslash\{0\}$ is arbitrary, this
completes the proof.
\end{proof}

\begin{lemma}
For any $(\hat{\boldsymbol{v}},\hat{q}),(\hat{\boldsymbol{w}},\hat{r})\in W_{k+d}$, we have

\vspace{-10pt}
\begin{align}\label{eqn:Cauchy-Schwarz inequality for a4}
a_3((\hat{\boldsymbol{v}},\hat{q}),(\hat{\boldsymbol{w}},\hat{r}))\leq \mathfrak{C}_3\|(\hat{\boldsymbol{w}},\hat{r})\|_D
\|(\hat{\boldsymbol{v}},\hat{q})\|_D,
\end{align}
where $\mathfrak{C}_3$ is a positive constant independent of $h$.
\end{lemma}
\begin{proof}
We continue to use $\{\varphi_j\}_{j=1}^{N_{v}}$ and
$\{\psi_j\}_{j=1}^{N_{p}}$ as the bases in space $W_{k+d}$ for
velocity and pressure, respectively.

Let

\vspace{-10pt}
\begin{align*}
&\hat{\boldsymbol{v}}=\sum_{j=1}^{N_v}x_{v,j}\varphi_j, \quad
\boldsymbol{x}_v=(x_{v,1},x_{v,2},\cdots,x_{v,N_v})^T,\\
&\hat{q}=\sum_{j=1}^{N_p}x_{p,j}\psi_j, \quad
\boldsymbol{x}_p=(x_{p,1},x_{p,2},\cdots,x_{p,N_p})^T,\\
&\hat{\boldsymbol{w}}=\sum_{j=1}^{N_v}y_{v,j}\varphi_j, \quad
\boldsymbol{y}_v=(y_{v,1},y_{v,2},\cdots,y_{v,N_v})^T,\\
&\hat{r}=\sum_{j=1}^{N_p}y_{p,j}\psi_j, \quad
\boldsymbol{y}_p=(x_{p,1},y_{p,2},\cdots,y_{p,N_p})^T.
\end{align*}
Then, using (\ref{eqn:a4 bilinear}), Cauchy-Schwarz inequality, Lemma~\ref{lem:B and P},
Lemma~\ref{lem:A and Du}, and Lemma~\ref{lem:B and L2 norm}

\vspace{-10pt}
\begin{align*}
&\quad a_3((\hat{\boldsymbol{v}},\hat{q}),(\hat{\boldsymbol{w}},\hat{r}))\\
&=\boldsymbol{y}_v^TD_v\boldsymbol{x}_v+\boldsymbol{y}_v^TB\boldsymbol{x}_p-
\boldsymbol{y}_p^TB^T\boldsymbol{x}_v+\boldsymbol{y}_p^T(c_sD_p-B^TD_v^{-1}B)\boldsymbol{x}_p\\
&\leq |\hat{\boldsymbol{v}}|_D|\hat{\boldsymbol{w}}|_D+\|\hat{q}\|\|\nabla\cdot \hat{\boldsymbol{w}}\|
+\|\hat{r}\|\|\nabla\cdot \hat{\boldsymbol{v}}\|
+\|\hat{q}\|_P\|\hat{r}\|_P+\|\hat{q}\|_B\|\hat{r}\|_B\\
&\leq |\hat{\boldsymbol{v}}|_D|\hat{\boldsymbol{w}}|_D+d\|\hat{q}\||\hat{\boldsymbol{w}}|_{1,\Omega}
+d\|\hat{r}\||\hat{\boldsymbol{v}}|_{1,\Omega}
+(c_i^{-1}+1)\|\hat{q}\|_B\|\hat{r}\|_B\\
&\leq |\hat{\boldsymbol{v}}|_D|\hat{\boldsymbol{w}}|_D+d\beta_2\|\hat{q}\||\hat{\boldsymbol{w}}|_D
+d\beta_2\|\hat{r}\||\hat{\boldsymbol{v}}|_D
+(c_i^{-1}+1)d(d+1)\beta_2\|\hat{q}\|\|\hat{r}\|\\
&\leq \mathfrak{C}_3\|(\hat{\boldsymbol{v}},\hat{q})\|_D\|(\hat{\boldsymbol{w}},\hat{r})\|_D,
\end{align*}
with $\mathfrak{C}_3=\sqrt{2}\max{\{1, d\beta_2, (c_i^{-1}+1)d(d+1)\beta_2\}}$.
\end{proof}

\begin{lemma}\label{lem:e_h tilde estimate e_h hat}
Let $(\tilde{\boldsymbol{e}}_u,\tilde{e}_p)$ and $(\bar{\boldsymbol{e}}_u,\bar{e}_p)$  be the solutions of (\ref{eqn:error in com space easily}) and (\ref{eqn:xu hat})-(\ref{eqn:xp hat}),
respectively. Then,

\vspace{-10pt}
\begin{align}\label{eqn:fourth error estimations}
\frac{(\mu\beta_1)^2}{\mathfrak{C}_2(1+\mu\beta_1)^2}\|(\bar{\boldsymbol{e}}_u,\bar{e}_p)\|_D\leq
\|(\tilde{\boldsymbol{e}}_u,\tilde{e}_p)\|_D
\leq\frac{\mathfrak{C}_3(1+\mu\beta_1)^2}{(\mu\beta_1)^2}\|(\bar{\boldsymbol{e}}_u,\bar{e}_p)\|_D,
\end{align}
where $\|\cdot\|_D$ is defined in (\ref{eqn:norm associate with a3}).  The constants $\mathfrak{C}_2,\mathfrak{C}_3,\beta_1$, and $\mu$ are defined in (\ref{eqn:Cauchy-Schwarz inequality for a3}), (\ref{eqn:Cauchy-Schwarz inequality for a4}), (\ref{eqn:equivalence relation in T}), and (\ref{eqn:inf-sup in com space}), respectively.
\end{lemma}
\begin{proof}
It follows from (\ref{eqn:error
in com space easily}) and (\ref{eqn:the third error problem}) that

\vspace{-10pt}
\begin{align}\label{eqn:a4=a3}
a_3((\bar{\boldsymbol{e}}_u,\bar{e}_p),(\hat{\boldsymbol{v}},\hat{q}))
=a_2((\tilde{\boldsymbol{e}}_u,\tilde{e}_p),(\hat{\boldsymbol{v}},\hat{q})), \quad
\forall~(\hat{\boldsymbol{v}},\hat{q})\in W_{k+d}.
\end{align}
Using (\ref{eqn:a4=a3}) and Lemma~\ref{lem:a4 inf sup}, we obtain

\vspace{-10pt}
\begin{align*}
\frac{(\mu\beta_1)^2}{(1+\mu\beta_1)^2}\|(\bar{\boldsymbol{e}}_u,\bar{e}_p)\|_D&
\leq\sup_{(\hat{\boldsymbol{v}},\hat{q})\in W_{k+d}\backslash\{0\}}\frac{a_3((\bar{\boldsymbol{e}}_u,\bar{e}_p),
(\hat{\boldsymbol{v}},\hat{q}))}{\|(\hat{\boldsymbol{v}},\hat{q})\|_D}\\
&=\sup_{(\hat{\boldsymbol{v}},\hat{q})\in W_{k+d}\backslash\{0\}}\frac{a_2((\tilde{\boldsymbol{e}}_u,\tilde{e}_p),(\hat{\boldsymbol{v}},\hat{q}))}
{\|(\hat{\boldsymbol{v}},\hat{q})\|_D}\\
&\leq\sup_{(\hat{\boldsymbol{v}},\hat{q})\in W_{k+d}\backslash\{0\}}\frac{\mathfrak{C}_2\|(\tilde{\boldsymbol{e}}_u,\tilde{e}_p)\|_D
\|(\hat{\boldsymbol{v}},\hat{q})\|_D}{\|(\hat{\boldsymbol{v}},\hat{q})\|_D}\\
&=\mathfrak{C}_2\|(\tilde{\boldsymbol{e}}_u,\tilde{e}_p)\|_D,
\end{align*}
which implies the first inequality in (\ref{eqn:fourth error estimations}).

Similarly, using (\ref{eqn:a4=a3}) and Lemma \ref{lem:a3 inf sup}, we have

\vspace{-10pt}
\begin{align*}
\frac{(\mu\beta_1)^2}{(1+\mu\beta_1)^2}\|(\tilde{\boldsymbol{e}}_u,\tilde{e}_p)\|_D&
\leq\sup_{(\hat{\boldsymbol{v}},\hat{q})\in W_{k+d}\backslash\{0\}}\frac{a_2((\tilde{\boldsymbol{e}}_u,\tilde{e}_p),
(\hat{\boldsymbol{v}},\hat{q}))}{\|(\hat{\boldsymbol{v}},\hat{q})\|_D}\\
&=\sup_{(\hat{\boldsymbol{v}},\hat{q})\in W_{k+d}\backslash\{0\}}\frac{a_3((\bar{\boldsymbol{e}}_u,\bar{e}_p),(\hat{\boldsymbol{v}},\hat{q}))}
{\|(\hat{\boldsymbol{v}},\hat{q})\|_D}\\
&\leq\sup_{(\hat{\boldsymbol{v}},\hat{q})\in W_{k+d}\backslash\{0\}}\frac{\mathfrak{C}_3\|(\bar{\boldsymbol{e}}_u,\bar{e}_p)\|_D\|(\hat{\boldsymbol{v}},\hat{q})\|_D}
{\|(\hat{\boldsymbol{v}},\hat{q})\|_D}\\
&=\mathfrak{C}_3\|(\bar{\boldsymbol{e}}_u,\bar{e}_p)\|_D,
\end{align*}
which implies the second inequality in (\ref{eqn:fourth error estimations}).
\end{proof}

Combining Theorem \ref{theo:diag velocity approx error} and Lemma
\ref{lem:e_h tilde estimate e_h hat}, we obtain the following lower
and upper bounds for the estimator
$\|(\bar{\boldsymbol{e}}_u,\bar{e}_p)\|_D$.
\begin{theorem}\label{theo:diag pressure approx error}
Let $(\boldsymbol{u},p),(\hat{\boldsymbol{u}},\hat{p})$ and
$(\bar{\boldsymbol{e}}_u,\bar{e}_p)$ be the solutions of
(\ref{eqn:variational form}),(\ref{eqn:error problem}), and (\ref{eqn:the third error problem}), respectively.
There are constants $\bar{\mathfrak{C}}_*=\frac{\mu^6\beta_1^4}
{2\mathfrak{C}_1^2\mathfrak{C}_2(1+\mu)^2(1+\beta_1\mu)^4\sqrt{1+\beta_2}}$, and
$\bar{\mathfrak{C}}^*=\frac{\mathfrak{C}_1\mathfrak{C}_2
\mathfrak{C}_3\sqrt{1+\beta_1}(1+\mu)^2(1+\mu\beta_1)^2}
{\mathfrak{c}_1\beta_1^2\sqrt{\beta_1}\mu^4}$ such that

\vspace{-10pt}
\begin{align}
\bar{\mathfrak{C}}_*\|(\tilde{\boldsymbol{e}}_{u},\tilde{e}_p)\|_D
+\frac{1}{2\sqrt{d}}\|\nabla\cdot\hat{\boldsymbol{u}}\|
\leq\|(\boldsymbol{u}-\hat{\boldsymbol{u}},p-\hat{p})\|_V\leq \bar{\mathfrak{C}}^*\|(\tilde{\boldsymbol{e}}_{u},\tilde{e}_p)\|_D
+\frac{1}{\mathfrak{c}_1}\|\nabla\cdot\hat{\boldsymbol{u}}\|
+\frac{C_{\mathcal{T}}}{\mathfrak{c}_1}osc(\boldsymbol{f}),
\end{align}
where $\|\cdot\|_V$ and $\|\cdot\|_D$ are defined in (\ref{eqn:norm
V}) and (\ref{eqn:norm associate with a3}),
respectively. The constants $\mathfrak{C}_1, \mathfrak{C}_2, \mathfrak{C}_3, \mathfrak{c}_1, \mu, \beta_1,\beta_2$, and $C_{\mathcal{T}}$ are defined in  (\ref{eqn:B continous}), (\ref{eqn:Cauchy-Schwarz inequality for a3}), (\ref{eqn:Cauchy-Schwarz inequality for a4}), (\ref{eqn:a1 orign inf-sup}), (\ref{eqn:inf-sup in com space}), (\ref{eqn:equivalence relation in T}), and Lemma~\ref{lem:quasi-interpolant v}, respectively.
\end{theorem}

\section{Adaptive Algorithm}
\label{sec:Adaptive finite element method} In this
section, we construct an adaptive FEM to solve (\ref{eqn:original
eqn 1})-(\ref{eqn:original
eqn 3}) based on the local and global \textit{a posteriori} error
estimators, denoted by $\eta_{L,T}$ and $\eta_G(\mathcal{T}_{m})$, defined in (\ref{eqn:local error estimator}) and (\ref{eqn:global error estimator}), which produce a sequence of discrete solutions $(\hat{\boldsymbol{u}}_m,\hat{p}_m)$
in nested spaces $V_{k,m}$ over triangulation $\mathcal{T}_m$.
The index $m$ indicates the underlying mesh with size $h_m$. Assume that an initial mesh
$\mathcal{T}_0$, a D\"{o}fler parameter $\theta\in (0,1)$, and a
targeted tolerance $\varepsilon$ are given.

Actually, a common adaptive refinement scheme involves a
loop structure of the form:

\vspace{-10pt}
\begin{equation*}
{\verb"SOLVE"}\rightarrow \verb"ESTIMATE"\rightarrow
\verb"MARK"\rightarrow \verb"REFINE"
\end{equation*}
with the initial triangulation $\mathcal{T}_0$ of $\Omega$ (cf.
\cite{Nochetto2009,Huang2011}). \verb"SOLVE" refers to solving the FEM
scheme (\ref{eqn:approximation problem}) on a relatively coarse mesh
$\mathcal{T}_m$. \verb"ESTIMATE" relies on an efficient and reliable
\textit{a posteriori} error estimate, and the local and global estimators
are defined in (\ref{eqn:local error estimator}) and (\ref{eqn:global error estimator}). With the
help of the error estimators, \verb"MARK" determines the
elements to be refined, hence creating a subset $\mathcal{S}_m$ of $\mathcal{T}_m$
for refinement. Finally, \verb"REFINE" generates a finer
triangulation $\mathcal{T}_{m+1}$ by dividing those elements in
$\mathcal{S}_m$, and an updated numerical solution will be computed
on $\mathcal{T}_{m+1}$.

For the first and the last step, there have been rapid advances for
solving the linear system (\ref{eqn:approximation problem}) and refinement
implementation, respectively in recent years, and we refer to
\cite{BOFFI1994,BOFFI1997,Verfurth2013} for the details. Here,
we focus on the interplay between the error estimator and the
marking strategy. The error estimator consists of local and global
estimates for a given triangulation. The local estimator provides
the information for the marking strategy to determine the triangles to
be refined, while the global error estimator provides the measure
for the reliable stop condition of the loops.

Recall that $\{\varphi_j\}_{j=1}^{N_{v}}$ and
$\{\psi_j\}_{j=1}^{N_{p}}$ are the bases in $W_{k+d}$ for
velocity and pressure, respectively. The matrix form of the third error problem is: Find
$(\bar{\boldsymbol{e}}_u,\bar{e}_p)\in W_{k+d}$ with
$\bar{\boldsymbol{e}}_u=\sum_{j=1}^{N_v}\bar{x}_{u,j}\varphi_j$
and $\bar{e}_p=\sum_{j=1}^{N_p}\bar{x}_{p,j}\psi_j$
satisfying (\ref{eqn:xu hat}) and (\ref{eqn:xp hat}). The
definitions of matrix $D_v$ and $B$ are similar to (\ref{eqn:matrix
form of a3}). The elements of $D_v$ and $B$ are as follows

\vspace{-10pt}
\begin{align*}
&(D_v)_{j,j}=a(\varphi_j,\varphi_j),\quad j=1,\cdots,N_v,\\
&B_{\ell,j}=-b(\psi_j,\varphi_{\ell}),\quad \ell=1,\cdots,N_v \mbox{~and~} j=1,\cdots, N_p.
\end{align*}
Let $D_p$ be the diagonal matrix with the same diagonal as
$B^TD_v^{-1}B$. Then the elements of $D_p$ are

\vspace{-10pt}
\begin{align*}
(D_p)_{j,j}=\sum_{\ell=1}^{N_v}\frac{B_{\ell,j}^2}{(D_v)_{\ell,\ell}}, \quad j=1,\cdots, N_p.
\end{align*}
From (\ref{eqn:xp hat}), we can get

\vspace{-10pt}
\begin{align*}
\bar{x}_{p,j}=\frac{F_{p,j}+\sum\limits_{\ell=1}^{N_v}\frac{B_{\ell,j}F_{v,\ell}}{(D_v)_{\ell,\ell}}}
{c_s\sum\limits_{\ell=1}^{N_v}\frac{B_{\ell,j}^2}{(D_v)_{\ell,\ell}}},\quad
j=1,\cdots, N_p,
\end{align*}
where $F_{v,\ell}$ and $F_{p,j}$ can be find in (\ref{eqn:Fvj}) and
(\ref{eqn:Fpj}). For any $T\in \mathcal{T}_m$, set
$\Lambda_T^p=\{j~\big|~ supp(\psi_j)\cap T\neq
\emptyset,j=1,\cdots,N_p\}$, then

\vspace{-10pt}
\begin{align*}
\bar{e}_{p,T}:=(\bar{e}_p)_{|T}=
\sum\limits_{j\in\Lambda_T^p}\bar{x}_{p,j}\psi_{j}.
\end{align*}
Following \cite{Verfurth2013}, the local error estimator for pressure can be defined as

\vspace{-10pt}
\begin{align*}
\eta_{L,T}^p=\|\bar{e}_{p,T}\|_{0,T}
=\|\sum\limits_{j\in\Lambda_T^p}\bar{x}_{p,j}\psi_{j}\|_T,\quad T\in \mathcal{T}_m.
\end{align*}
From (\ref{eqn:xu hat}), we can get

\vspace{-10pt}
\begin{align*}
\bar{x}_{u,\ell}=\frac{1}{(D_v)_{\ell,\ell}}(F_{v,\ell}
-\sum\limits_{j=1}^{N_p}B_{\ell,j}\bar{x}_{p,j}),\quad \ell=1,\cdots, N_v.
\end{align*}
Set  $\Lambda_T^v=\{j~\big|~ supp(\varphi_j)\cap T\neq \emptyset,j=1,\cdots,N_v\}$, then

\vspace{-10pt}
\begin{align*}
\bar{\boldsymbol{e}}_{u,T}:=(\bar{\boldsymbol{e}}_u)_{|T}=
\sum\limits_{j\in\Lambda_T^v}\bar{x}_{u,j}\varphi_{j}.
\end{align*}
From (\ref{eqn:local norm D}), the local error estimator for velocity can be defined as

\vspace{-10pt}
\begin{align*}
\eta_{L,T}^v=|\bar{\boldsymbol{e}}_{u,T}|_{D,T}=
\sqrt{\sum\limits_{j\in\Lambda_T^v}|\bar{x}_{u,j}\varphi_{j}|_{1,T}^2},
\quad T\in \mathcal{T}_m.
\end{align*}
The local error estimator for divergence term can be defined as

\vspace{-10pt}
\begin{align*}
\eta_{L,T}^d=\|\nabla\cdot\hat{\boldsymbol{u}}\|_{0,T}.
\end{align*}
Then, the local error estimator can be defined as

\vspace{-10pt}
\begin{align}\label{eqn:local error estimator}
\eta_{L,T}=\sqrt{(\eta_{L,T}^p)^2+(\eta_{L,T}^v)^2+(\eta_{L,T}^d)^2}.
\end{align}

Recall the third error problem (\ref{eqn:the third error problem}) and
the associated norm (\ref{eqn:norm associate with a3}), we define the
global error estimator as

\vspace{-10pt}
\begin{align}\label{eqn:global error estimator}
\eta_G(\mathcal{T}_{m})=\sqrt{\sum\limits_{k\in\mathcal{T}_m}\eta^2_{L,T}}
=\sqrt{\|(\bar{\boldsymbol{e}}_u,\bar{e}_p)\|_D^2+\|\nabla\cdot\hat{\boldsymbol{u}}\|^2}.
\end{align}
Based on Theorem~\ref{theo:diag pressure approx error}, the global
error estimator $\eta_G(\mathcal{T}_{m})$ provides an estimate of the discretization
error $\|(\boldsymbol{u}-\hat{\boldsymbol{u}},p-\hat{p})\|_V$, which is frequently
used to judge the quality of the underlying discretization. The
local error estimator $\eta_{L,T}$ is an estimate of the error on
element $T$. All elements $T\in \mathcal{T}_{m}$ are marked for
refinement, if $\eta_{L,T}$ exceeds the certain tolerance. Denote the set of all marked
elements by $\mathcal{S}_{m}\subset \mathcal{T}_m$ . The global error estimator associated with
$\mathcal{S}_{m}$ is denoted by $\eta_G(\mathcal{S}_{m})$.

The algorithm of adaptive FEM is listed in
Algorithm~\ref{alg:adaptive FEM}.

\vspace{-10pt}
\begin{algorithm}[h]
\caption{Adaptive FEM}\label{alg:adaptive FEM}
\begin{flushleft}
{\bf Input}: Construct an initial mesh $\mathcal{T}_{0}$. Choose a
parameter $0<\theta<1$ and a tolerance $\varepsilon$.

{\bf Output}: Final triangulation $\mathcal{T}_{M}$ and the
finite element approximation $(\hat{\boldsymbol{u}}_{M},\hat{p}_M)$ on $\mathcal{T}_{M}$.

Set $m=0$ and  $\eta_G(\mathcal{T}_{m})=1$.

{\bf While} $\eta_G(\mathcal{T}_{m})>\varepsilon$
\begin{itemize}
\item [1.] (\verb"SOLVE") Solve the FEM scheme
(\ref{eqn:approximation problem}) on $\mathcal{T}_m$.

\item [2.] (\verb"ESTIMATE") Compute the local error estimator
as defined in (\ref{eqn:local error estimator}) for all elements $T\in\mathcal{T}_{m}$.

\item [3.] (\verb"MARK") Construct a subset
$\mathcal{S}_{m}\subset \mathcal{T}_m$ with least number of elements
such that
\begin{align*}
\eta_G^2(\mathcal{S}_{m})\geq \theta \eta_G^2(\mathcal{T}_{m}).
\end{align*}

\item [4.] (\verb"REFINE") Refine elements in $\mathcal{S}_{m}$
together with the elements, which must be refined to make
$\mathcal{T}_{m+1}$ conforming.

\item [5.] Set $m=m+1$.
\end{itemize}

{\bf End}

Set $(\hat{\boldsymbol{u}}_{M},\hat{p}_M)=(\hat{\boldsymbol{u}}_m,\hat{p}_m)$ and $\mathcal{T}_M=\mathcal{T}_m$.

\end{flushleft}

\end{algorithm}

In the \verb"MARK" step, we adopt the D{\"o}rfler marking strategy
which is a mature strategy and is widely used in the adaptive algorithm
\cite{Dorfler1996}. Recently, it has been shown that D{\"o}rfler marking with minimal
cardinality is a linear complexity problem
\cite{Pfeiler2020}. In this marking strategy the local error
estimators $\{\eta_{L,T}\}_{T\in\mathcal{T}_m}$ are sorted in
descending order. The sorted local error estimators are denoted by
$\{\widetilde{\eta}_{L,T}\}_{T\in\mathcal{T}_m}$. Then, the set of
elements marked for refinement is given by
$\{\tilde{\eta}_{L,T}\}_{T\in\mathcal{S}_m}$, where $\mathcal{S}_m$
contains the least number of elements such that

\vspace{-10pt}
\begin{align*}
\eta_G^2(\mathcal{S}_{m})=\sum_{T\in\mathcal{S}_m}\widetilde{\eta}^2_{L,T}\geq \theta \eta^2_G(\mathcal{T}_{m}).
\end{align*}
Generally speaking, a small value of
$\theta$ leads to a small set $\mathcal{S}_m$,
while a large value of $\theta$ leads to a large set
$\mathcal{S}_m$. In \cite{Dorfler1996}, $\theta$ is suggested to be adopted in $[0.5, 0.8]$. We emphasize that many auxiliary
elements are refined to eliminate the hanging nodes, which
may have been created in the \verb"MARK" step. There are many mature
toolkits to process the hanging nodes \cite{Verfurth2013}.

Finally, to show the effectiveness of the global error
estimator defined in (\ref{eqn:global error estimator}), we
introduce the effective index as follows

\vspace{-10pt}
\begin{align}\label{eqn:effective index}
\kappa_{eff}=\frac{\eta_G(\mathcal{T}_{m})}{\|(\boldsymbol{u}-\hat{\boldsymbol{u}},p-\hat{p})\|_V},
\end{align}
which is the ratio between the global error estimator and the FEM approximation error. According to Theorem \ref{theo:diag pressure approx error}, the effective index is bounded from both above and
below.

\section{Numerical experiments}
\label{sec:Numerical experiments} In this section, we present
two-dimensional numerical examples to demonstrate the efficiency and reliability of
our adaptive FEM. All these simulations have been implemented on a
3.2GHz quad-core processor with 16GB RAM by Matlab.\\

\noindent{\bf Example 1.} This example is taken from page 113 in
\cite{Verfurth1996}. The solution is singular at the origin.
Let $\Omega$ be the L-shape domain $(-1,1)^2 \backslash[0,1)\times
(-1,0]$, and select $\boldsymbol{f}=0$. Then, use $(r,
\varphi)$ to denote the polar coordinates. We impose an appropriate
inhomogeneous boundary condition for $\boldsymbol{u}$ so that

\vspace{-10pt}
\begin{align*}
u_1(r,\varphi)&=r^\lambda((1+\lambda)sin(\varphi)\Psi(\varphi)
+cos(\varphi)\Psi^{\prime}(\varphi)),\\
u_2(r,\varphi)&=r^\lambda(sin(\varphi)\Psi^{\prime}(\varphi)
-(1+\lambda)cos(\varphi)\Psi(\varphi)),\\
p(r,\varphi)&=-r^{\lambda-1}[(1+\lambda)^2\Psi^{\prime}(\varphi)
+\Psi^{\prime\prime\prime}]/(1-\lambda),
\end{align*}
where

\vspace{-10pt}
\begin{align*}
\Psi(\varphi)&= \sin ((1+\lambda) \varphi) \cos (\lambda \omega)
/(1+\lambda)-\cos ((1+\lambda) \varphi) -\sin ((1-\lambda) \varphi)
\cos (\lambda \omega) /(1-\lambda)+\cos ((1-\lambda) \varphi),\\
\omega&= \frac{3 \pi}{2}.
\end{align*}
The exponent $\lambda$ is the smallest positive solution of

\vspace{-10pt}
\begin{align*}
\sin (\lambda \omega)+\lambda \sin (\omega)=0,
\end{align*}
thereby, $\lambda\approx 0.54448373678246$.

We emphasize that $(\boldsymbol{u},p)$ is analytic in $\overline{\Omega}\backslash\{0\}$, but both $\nabla \boldsymbol{u}$ and $p$ are singular at the origin; indeed, here $\boldsymbol{u}\notin [H^2(\Omega)]^2$ and $p\notin H^1(\Omega)$. This example reflects the typical (singular) behavior that solutions of the two-dimensional Stokes problem exhibit in the vicinity of reentrant corners in the computational domain.

We denote the finite element spaces by $V_1$ and $W_{3}$ in the approximation problem and the error problem, respectively. The finite element  space
$V_1$ consists of velocity space and pressure space. The
velocity space is the space of continuous piecewise quadratic
polynomials and the pressure space is the space of continuous
piecewise linear polynomials associated with $\mathcal{T}$. It is characterized in terms of
Lagrange basis.  The
hierarchical basis of any component with respect to velocity in any element
$T\in\mathcal{T}$ will be

\vspace{-10pt}
\begin{align*}
&\lambda_1\lambda_2\lambda_3, \lambda_2^2\lambda_3, \lambda_2\lambda_3^2, \lambda_1^2\lambda_3, \lambda_1\lambda_3^2, \lambda_1^2\lambda_2, \lambda_1\lambda_2^2,
\lambda_2^2\lambda_3^2, \lambda_1^2\lambda_3^2, \lambda_1^2\lambda_2^2, \lambda_1^2\lambda_2\lambda_3, \lambda_1\lambda_2^2\lambda_3, \lambda_1\lambda_2^2\lambda_3,
\lambda_1\lambda_2\lambda_3^2.
\end{align*}
The
hierarchical basis of pressure in any element
$T\in\mathcal{T}$ will be $\lambda_1\lambda_2\lambda_3$,
where $\lambda_i (i=1,2 ,3)$ are
Lagrange bases of the three vertices in $T$, respectively. These bases of $W_3$ in any element are shown in Figure~\ref{fig:basis ex1}.

\vspace{-10pt}
\begin{figure}[htbp]
  \centering
  \includegraphics[width=0.30\textwidth]{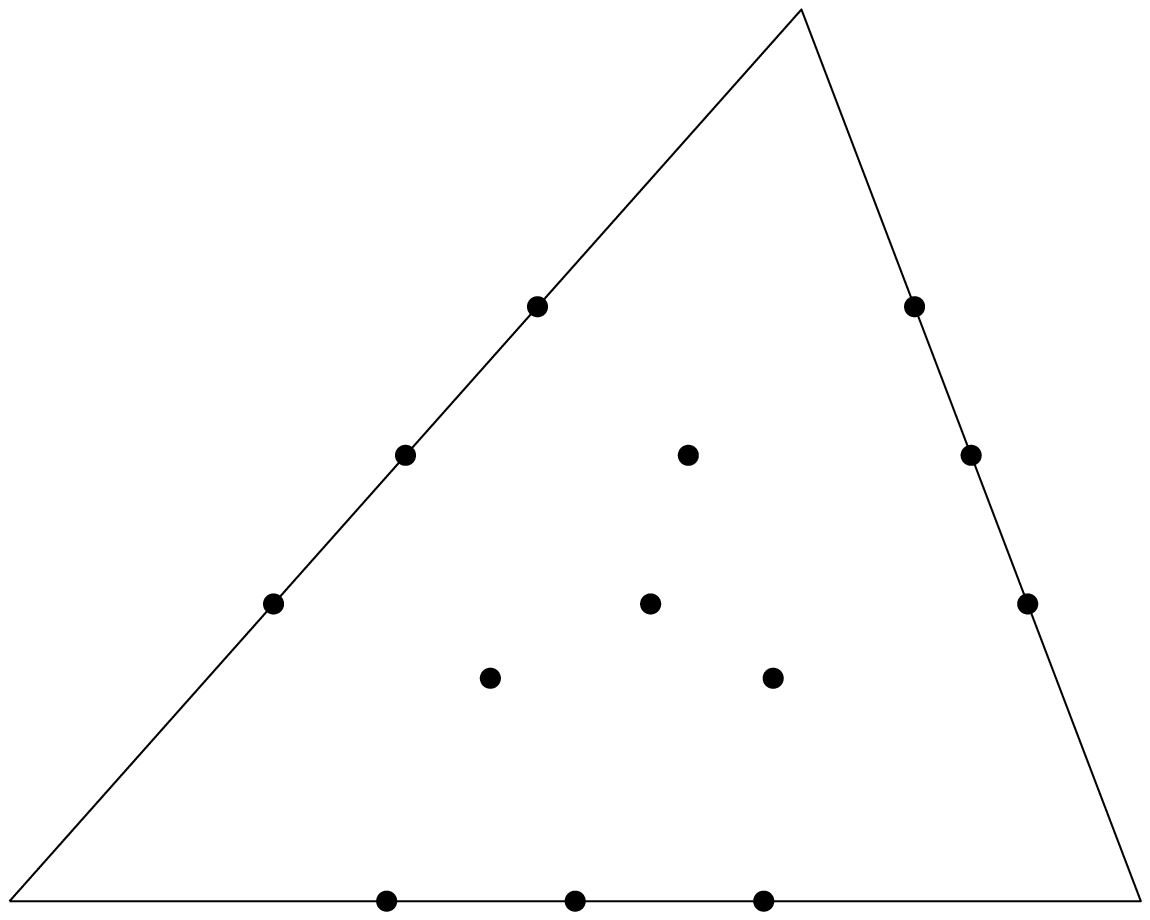}
  \includegraphics[width=0.30\textwidth]{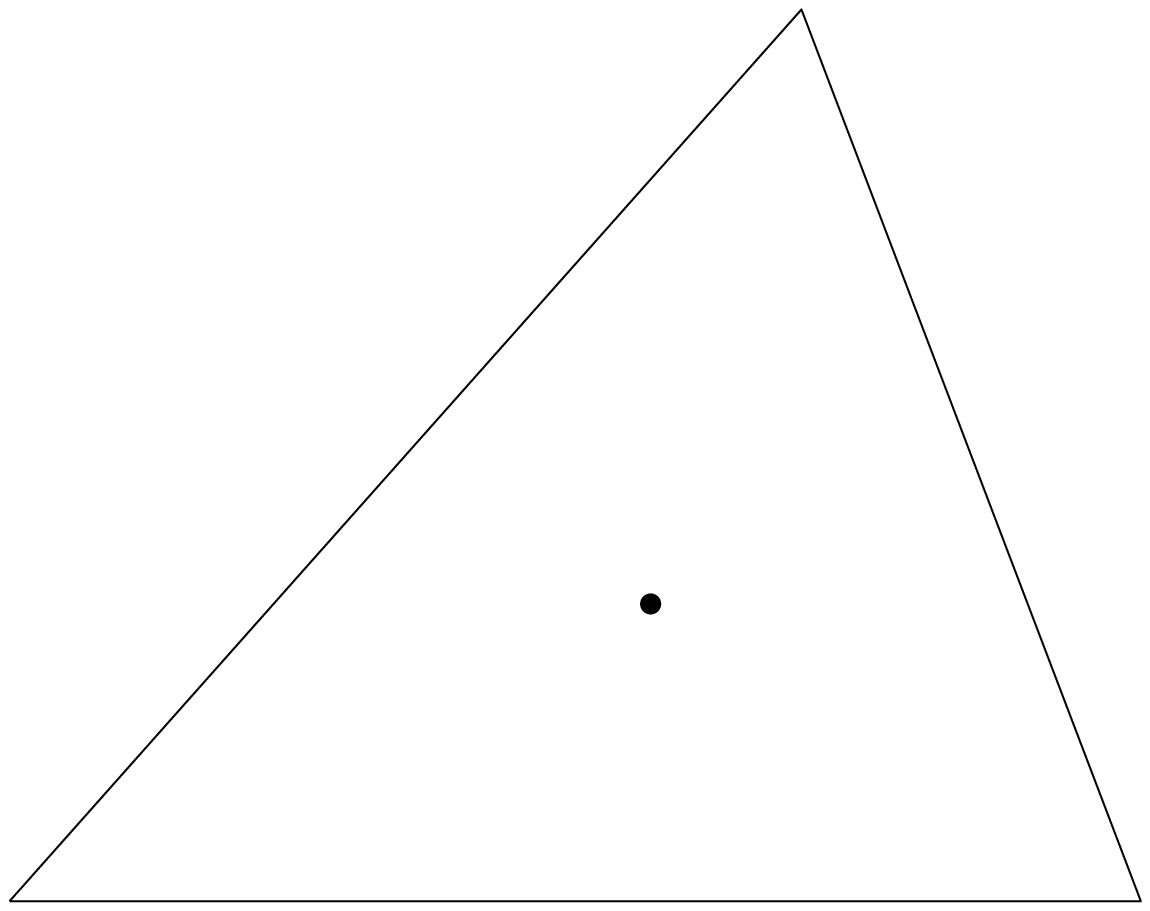}
\caption{The basis of velocity(left) and pressure(right) in any element for Example 1.}
 \label{fig:basis ex1}
\end{figure}

Figure~\ref{fig:convergence order ex1}(a) shows that for Example 1,
$\eta_G(\mathcal{T}_m)$ has different convergent rates with respect
to the degrees of freedom ($d.o.f$) for different $\theta$. Table
\ref{tab:computation cost ex1 FEM} shows the computation cost for
different $\theta$ when
$\|(\boldsymbol{u}-\hat{\boldsymbol{u}},p-\hat{p})\|_{V}<0.25$. From the comparison
in Table  \ref{tab:computation cost ex1 FEM}, adaptive FEM is
much faster than the uniform refinement. In the \verb"MARK" step, we
set the D\"{o}fler parameter as $\theta=0.7$. In the \verb"REFINE"
step, the refinement process is implemented using the MATLAB
function REFINEMESH. The key is dividing the marked element into four
parts by regular refinement (dividing all edges of the selected triangles in half).

Figure~\ref{fig:convergence order ex1}(b) shows the
convergent rates of $\|(\boldsymbol{u}-\hat{\boldsymbol{u}},p-\hat{p})\|_V$ and $\eta_G(\mathcal{T}_{m})$ for
Algorithm~\ref{alg:adaptive FEM}. The $x$-axes denotes the
$d.o.f$ in log scale, while $y$-axes
denotes the errors in log scale.
The squared line denotes the error $\|(\boldsymbol{u}-\hat{\boldsymbol{u}},p-\hat{p})\|_V$ of
the uniform refinement. The asterisk line and the circled line denote the
error $\|(\boldsymbol{u}-\hat{\boldsymbol{u}},p-\hat{p})\|_V$ and $\eta_G(\mathcal{T}_{m})$ of adaptive FEM,
respectively. It is obvious that
$\|(\boldsymbol{u}-\hat{\boldsymbol{u}},p-\hat{p})\|_V$ and $\eta_G(\mathcal{T}_m)$ have the same convergence
order and have a higher
convergence order than the uniform refinement.

\begin{figure}[htbp]
  \centering
  \includegraphics[width=0.45\textwidth]{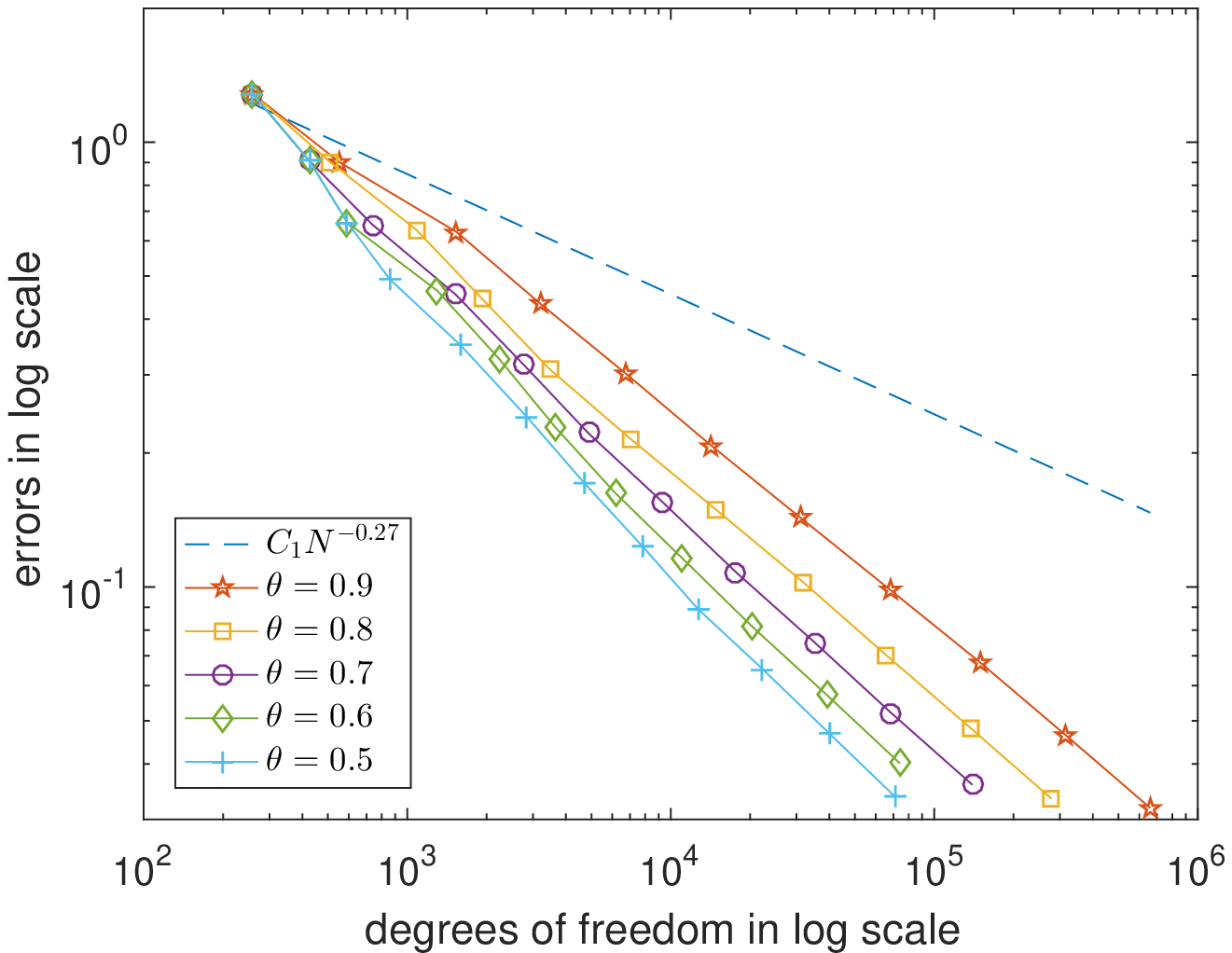}
  \includegraphics[width=0.45\textwidth]{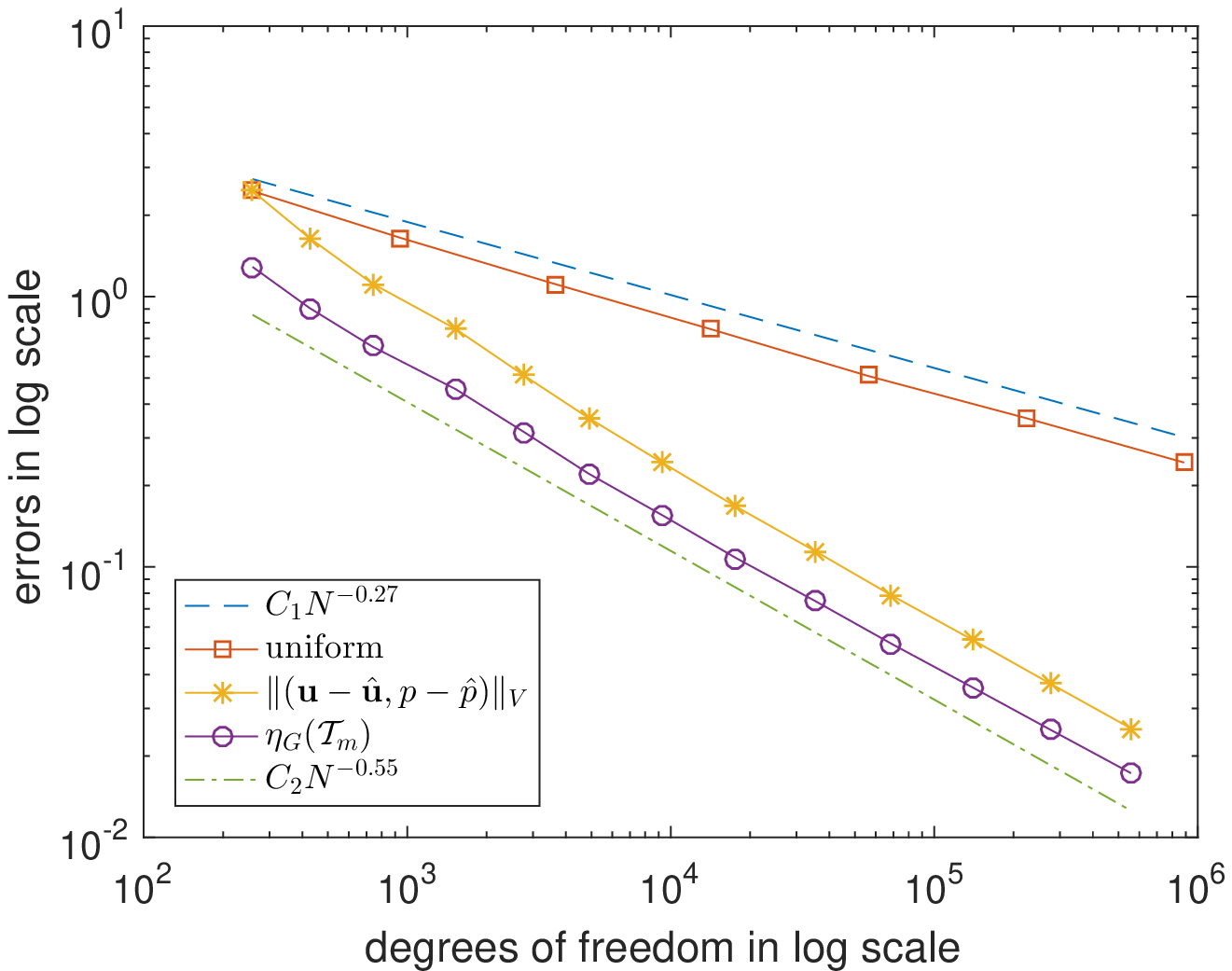}
    \hspace{88mm} (a)    \hspace{70mm} (b)
\caption{The convergent rates of adaptive FEM for Example 1.}
 \label{fig:convergence order ex1}
\end{figure}

\begin{table}[htbp]
\caption{Computation cost for different $\theta$ in Example 1.} \label{tab:computation cost ex1 FEM}
\begin{tabular*}{\hsize}{@{}@{\extracolsep{\fill}}ccccc@{}}
    \hline
   $\theta$     & refinement steps & $d.o.f$          &$\|(\boldsymbol{u}-\hat{\boldsymbol{u}},p-\hat{p})\|_{V}$ &  time(s) \\
    \hline
        $0.9$    & $7$           & $31063$   &  $0.243$    &  $6.533$  \\

       $0.8$    & $7$            & $14935$   &  $0.243$    &  $4.339$  \\

       $0.7$    & $7$            & $9250$   &  $0.244$    &  $3.299$  \\

       $0.6$    & $7$            & $6148$   &  $0.244$    &  $3.018$  \\

       $0.5$    & $7$            & $4738$   &  $0.245$    &  $2.308$  \\

    \hline
       $uniform$    & $7$            & $887299$  &  $0.243$     &  $130.463 $  \\
    \hline
\end{tabular*}
\end{table}

Table~\ref{tab:ex1 FEM} shows the error
$\|(\boldsymbol{u}-\hat{\boldsymbol{u}},p-\hat{p})\|_V$, the global error estimator
$\eta_G(\mathcal{T}_{m})$, and the effective index $\kappa_{eff}$ of
the adaptive FEM as the $d.o.f$ increases. The results
of effective index $\kappa_{eff}$ defined in (\ref{eqn:effective
index}) are shown in the sixth column. The effective index
$\kappa_{eff}$ is between $0.5$ and $0.7$ with adaptive refinement, which
shows the adaptive FEM is reliable. Figure~\ref{fig:Mesh and error
ex1} shows the initial mesh with $d.o.f=259$, fourth refinement mesh with
$d.o.f=2778$, and seventh refinement mesh with $d.o.f=17707$. It
has inserted refinement elements around the singularity at $(x,y)=(0,0)$ as $d.o.f$
increases to reduce the global error.

\begin{table}[htbp]
\centering \caption{The errors, the global error estimator, and the
effective index of the adaptive FEM in Example 1.} \label{tab:ex1
FEM}
\begin{tabular*}{\hsize}{@{}@{\extracolsep{\fill}}cccccc@{}}
    \hline
   $d.o.f$          &  $\|(\boldsymbol{u}-\hat{\boldsymbol{u}},p-\hat{p})\|_{V}$ & order   &  $\eta_G(\mathcal{T}_{m})$ & order & $\kappa_{eff}$\\
    \hline
      $259$   &  $2.452$E0      &   ---   &  $1.286$E0  &   ---  &
      0.524 \\

      $426$    &  $1.644$E0    &    $0.803$    &  $9.068$E-1  &   $0.702$ &
      0.551 \\

      $738$    &  $1.116$E0    &    $0.705$    &  $6.526$E-1  &   $0.598$ & 0.584\\

     $1523$    &  $7.595$E-1    &    $0.531$    &  $4.550$E-1  &   $0.497$ & 0.599\\

     $2778$    &  $5.101$E-1    &    $0.662$    &  $3.151$E-1  &   $0.611$& 0.617\\

     $4871$    &  $3.559$E-1   &    $0.640$    &  $2.215$E-1  &   $0.627$  &0.622\\

     $9250$    &  $2.436$E-1   &    $0.590$    &  $1.557$E-1  &   $0.549$  &0.639\\

     $17707$    &  $1.670$E-1   &    $0.581$    &  $1.078$E-1  &   $0.565$ &0.645\\

    $35236$    &  $1.145$E-1   &    $0.547$    &  $7.492$E-2  &   $0.529$  &0.653\\

    $68949$    &  $7.846$E-2   &    $0.564$    &  $5.194$E-2  &   $0.545$  &0.661\\

    $138420$    &  $5.381$E-2   &    $0.535$    &  $3.590$E-2  &   $0.524$ &0.667\\

    $277820$    &  $3.687$E-2   &    $0.548$    &  $2.492$E-2  &   $0.529$  &0.675\\

    $557663$   &  $2.529$E-2     &   $0.540$   &  $1.728$E-2  &   $0.525$  & 0.683 \\
    \hline
\end{tabular*}
\end{table}

\begin{figure}[htbp]
  \centering
  \includegraphics[width=0.30\textwidth]{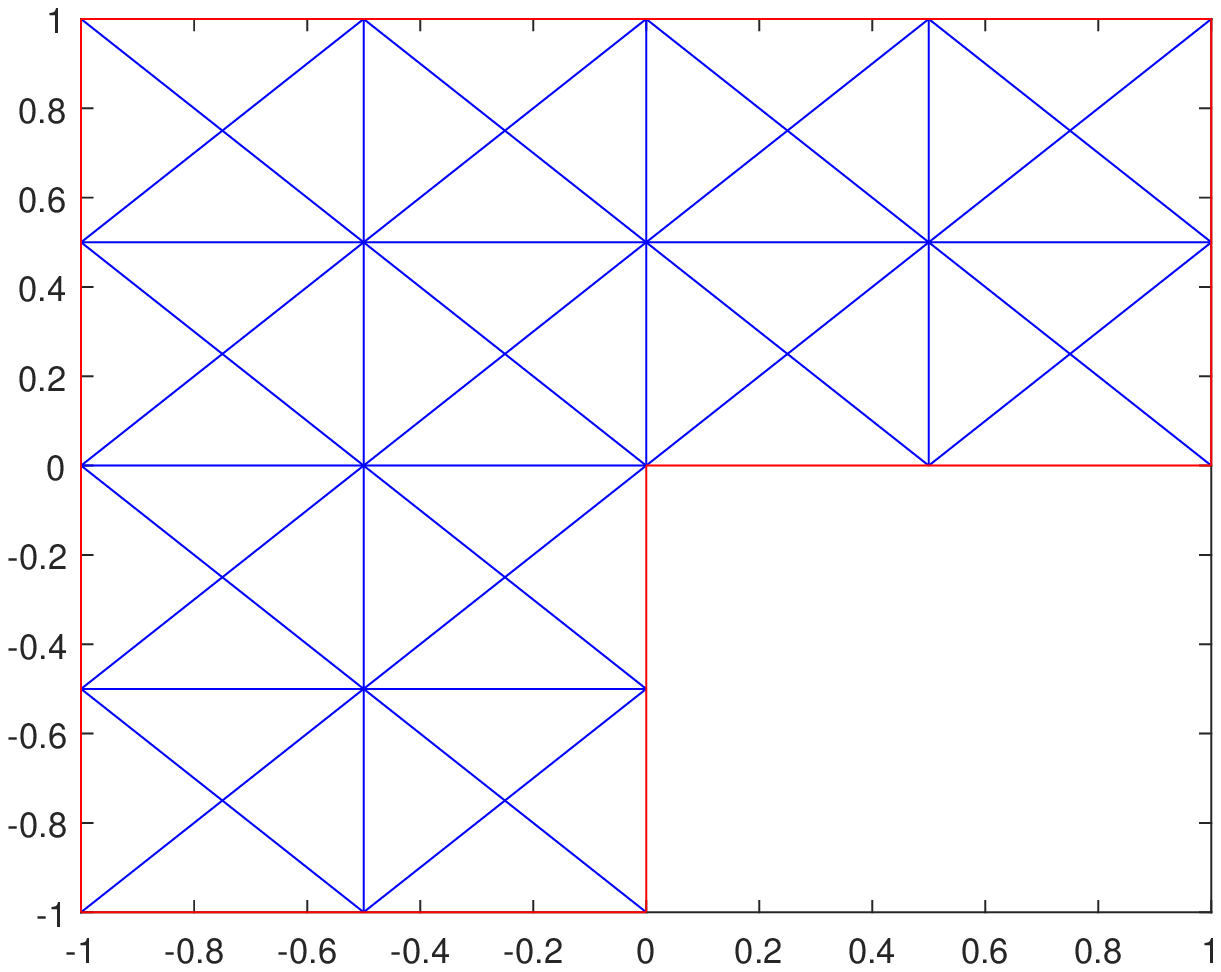}
  \includegraphics[width=0.30\textwidth]{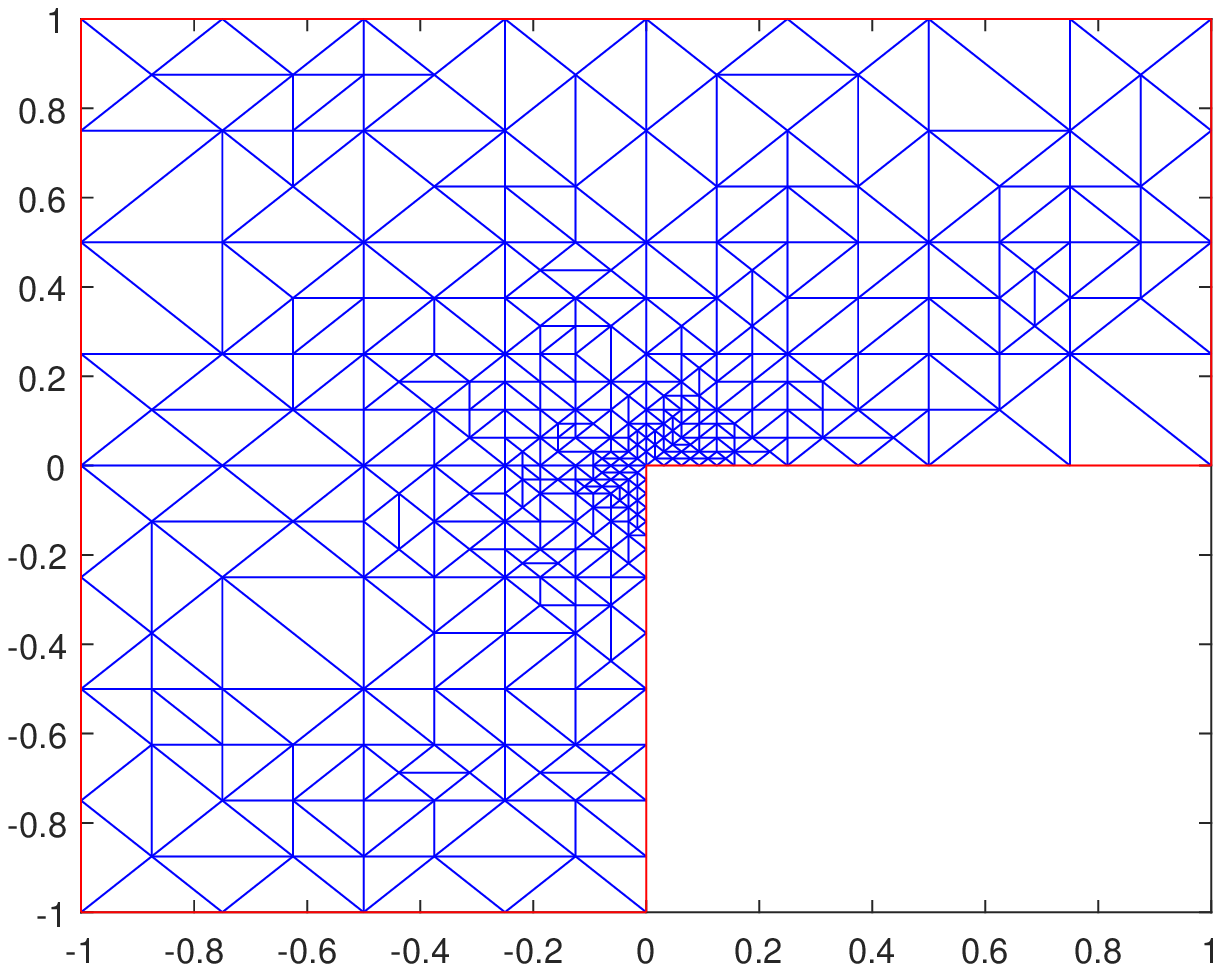}
  \includegraphics[width=0.30\textwidth]{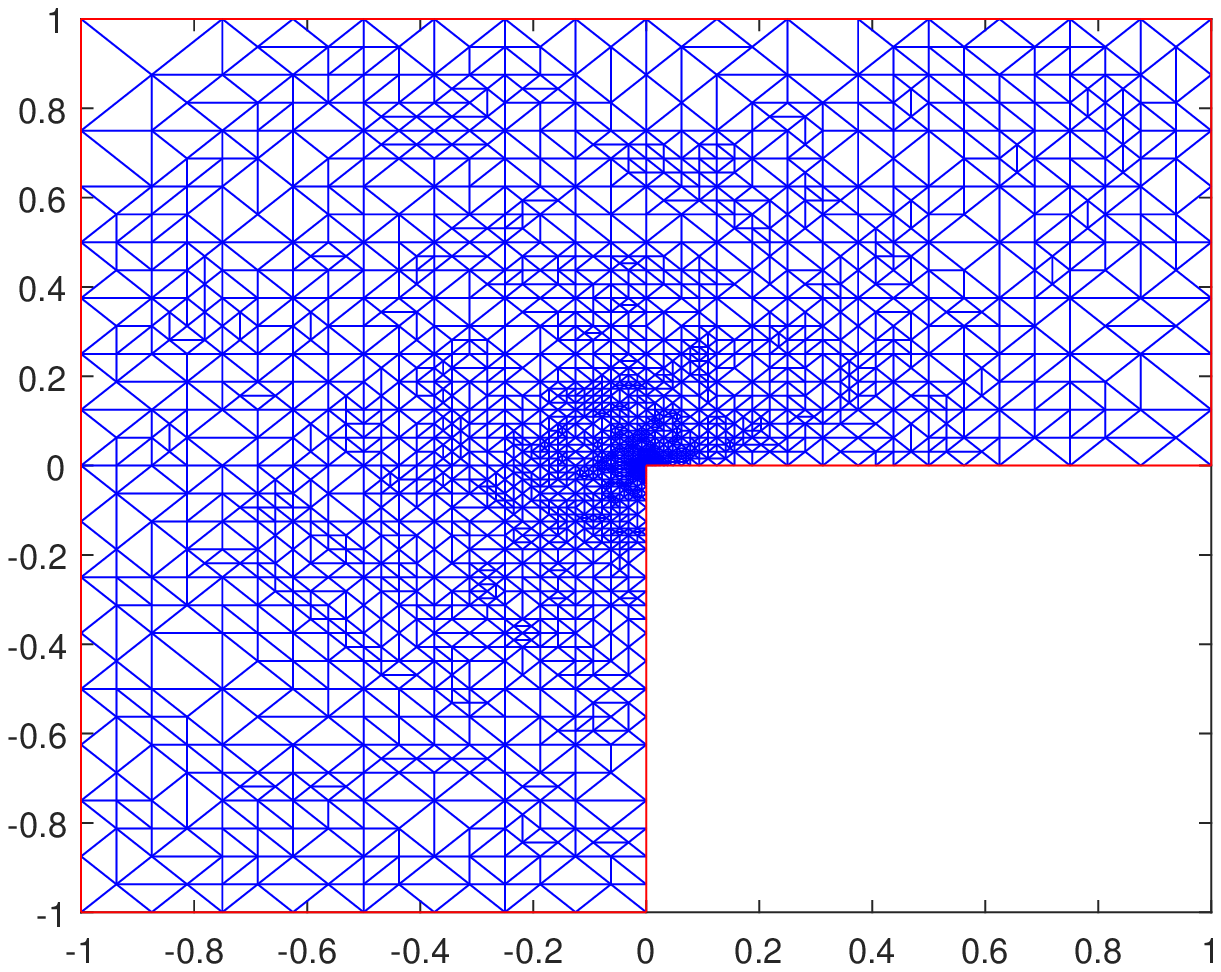}\\
\caption{The meshes with
$d.o.f=259$ (left), $2778$ (middle), and $17707$ (right) for Example 1.}
 \label{fig:Mesh and error ex1}
\end{figure}

\noindent{\bf Example 2.} In this case, we test the lid-driven cavity problem. The domain is taken as the square $\Omega=(0,1)\times(0,1)$, we set $\boldsymbol{f}=\boldsymbol{0}$, and the boundary conditions $\boldsymbol{u}=0$ on $[\{0\}\times(0,1)]\cup[(0,1)\times\{0\}]\cup[\{1\}\times(0,1)]$ and $\boldsymbol{u}=(1,0)^T$ on $(0,1)\times\{1\}$. This problem has a corner singularity. The tangential component of velocity $\boldsymbol{u}\cdot\boldsymbol{\tau}$ has a discontinuity at the two top corners, where $\boldsymbol{\tau}$ denotes the unit tangential vector on the boundary. We use the proposed adaptive FEM algorithm to solve this problem. The finite element space, D\"{o}fler parameter, and refinement criterion are the same as Example 1. Figure~\ref{fig:mesh ex2} shows that the refinement of mesh focuses on the two top corners. In Figure~\ref{fig:pressure ex2}, we depict the discrete pressure field obtained using the initial and adapted meshes where we note the improvement in the quality of the computed solution since the singular nature of the pressure is better captured in the adapted mesh.

\begin{figure}[htbp]
  \centering
  \includegraphics[width=0.40\textwidth]{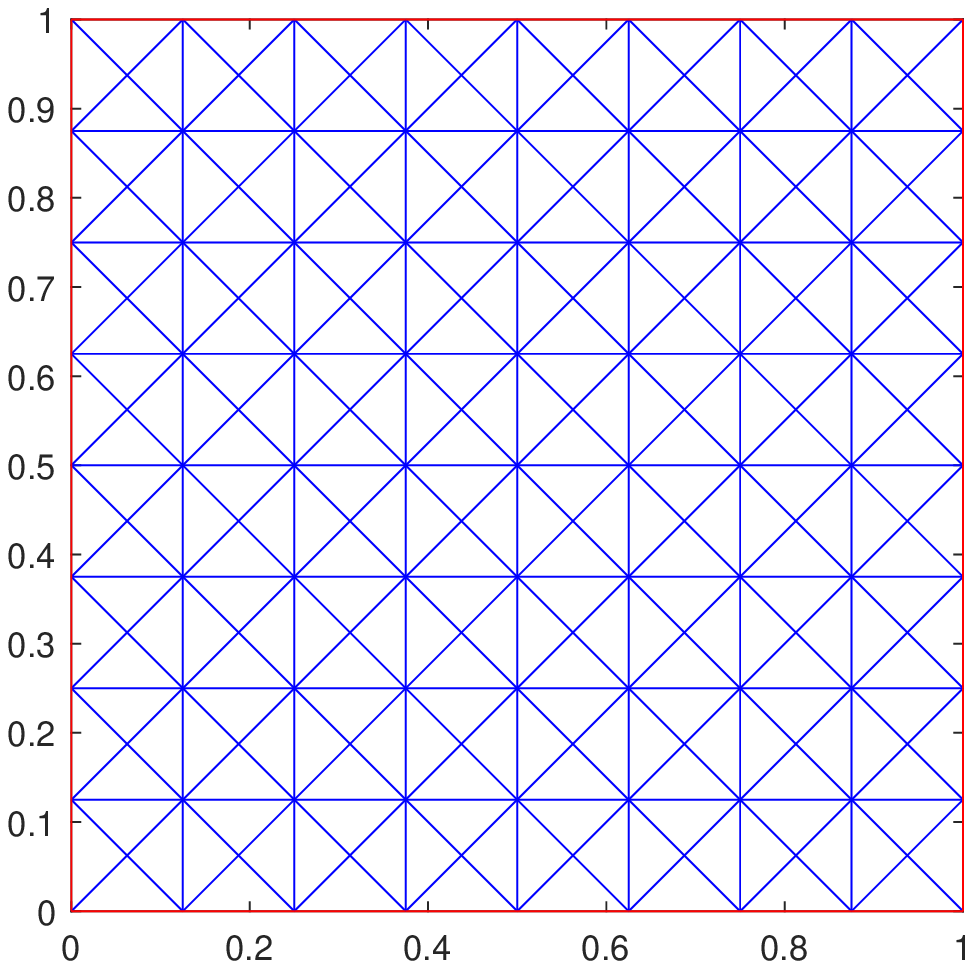}
  \includegraphics[width=0.40\textwidth]{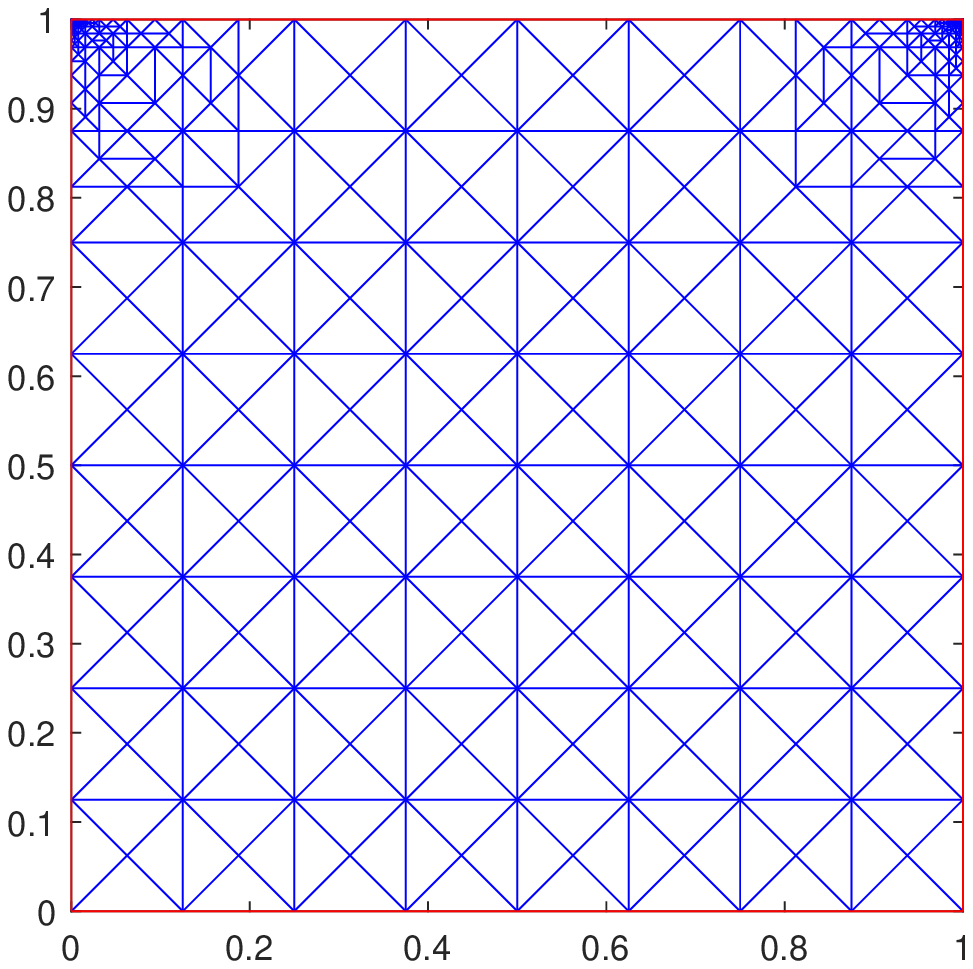}\\
\caption{The initial mesh and tenth refinement mesh for Example 2.}
 \label{fig:mesh ex2}
\end{figure}

\begin{figure}[htbp]
  \centering
  \includegraphics[width=0.40\textwidth]{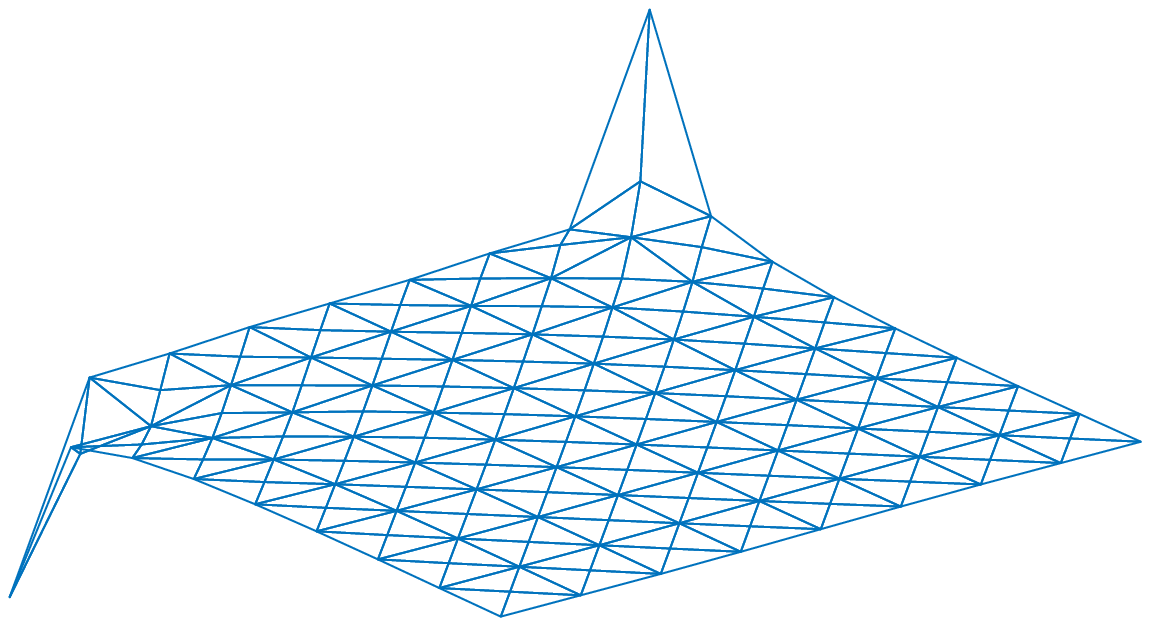}
  \includegraphics[width=0.40\textwidth]{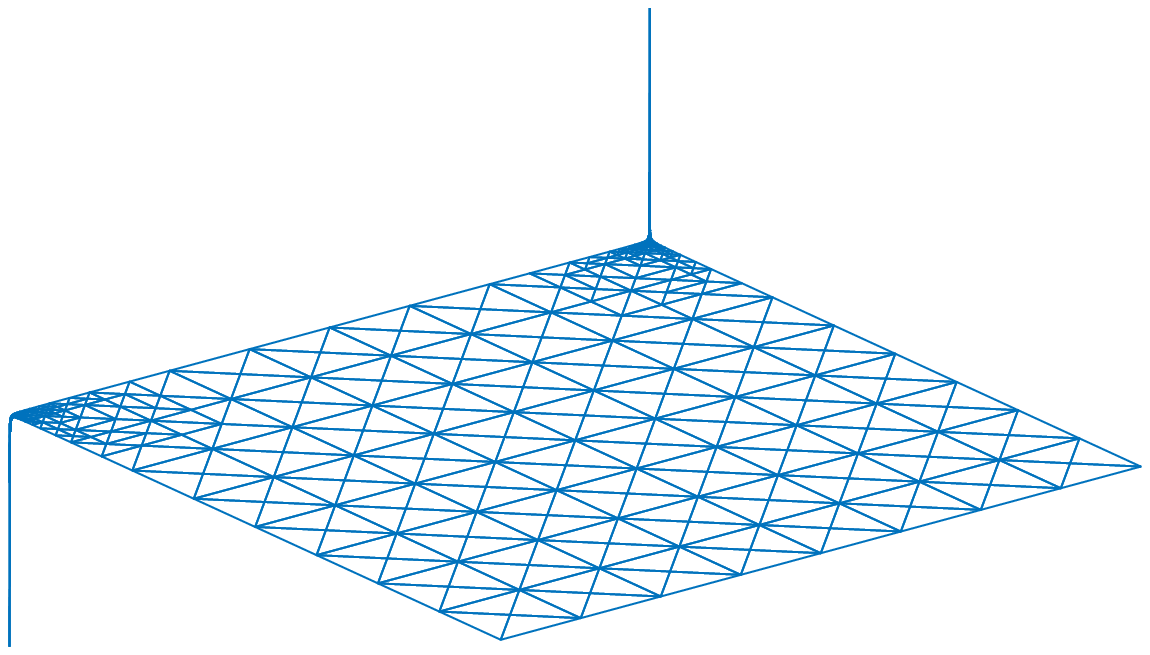}\\
\caption{The pressures in initial mesh and tenth refinement mesh for Example 2.}
 \label{fig:pressure ex2}
\end{figure}

\section{Conclusion}\label{sec:Conclusion}
In this paper, we present an adaptive FEM for solving the Stokes
problem with Dirichlet boundary condition.  Based on auxiliary subspace techniques, we proposed a hierarchical basis \textit{a posteriori} error estimator, which is most efficient and robust. We need to solve only two global diagonal linear systems. In theory,  The estimator is proved
to have global upper and lower bounds without saturation assumption. Numerical experiments are
shown to illustrate the efficiency and reliability of our adaptive
algorithm.

\section*{Acknowledgments}
The work of J.C. Zhang was supported by the Natural Science Foundation of Jiangsu Province (grant no.BK20210540)
, the Natural Science Foundation of the Jiangsu Higher Education Institutions of China (grant no.21KJB110015). The work of R. Zhang was supported by the National Key Research and Development Program of China (grant no.2020YFA0713601).

\section*{Declarations}
\textbf{Conflict of interest} The authors have no conflicts of interest to declare.\\
~\\
\textbf{Data Availability} Data sharing is not applicable to this article as no datasets were generated or analysed during the current study.

\section*{Appendix A.}
\setcounter{equation}{0}
\renewcommand\theequation{A.\arabic{equation}}
\renewcommand\thefigure{A.\arabic{figure}}
\setcounter{figure}{0}
\textit{The proof of Lemma~\ref{lem:inf-sup one layer in com space}}.\\
\begin{proof}
The idea of proof is similar to \cite{BOFFI1994} for $d=2$ and \cite{BOFFI1997} for $d=3$. Next, we will give proofs for $d=2$ and $d=3$, respectively.

\textbf{Case $d=2$}: The idea is to consider a macroelement partition of the domain $\Omega$ in such a way that each macroelement contains exactly three triangles. By virtue of Remark 3.3 in \cite{BOFFI1994}, it suffices to prove the inf-sup condition for only one macroelement. We consider a macroelement $\Omega_i=a\cup b\cup c$ as in Figure~\ref{fig:2d_inf_sup}

\begin{figure}[htbp]
  \centering
  \includegraphics[width=0.4\textwidth]{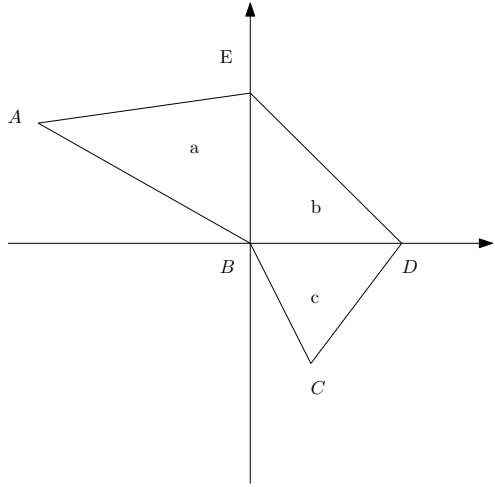}
\caption{The macroelement partition containing three triangles.}
 \label{fig:2d_inf_sup}
\end{figure}

Let us introduce some notations. We denote by $\lambda_{AB}^a$ by the barycentric coordinate related to the triangle a, which vanishes on the edge $AB$ (analogous notations for the other cases). we denote by $L_{i,x}^a$ the $i$-th Legendre polynomial with respect to the measure $\mu_{a,x}$ defined as

\vspace{-10pt}
\begin{align}
\label{eqn:measure a}
\int_{x_A}^0 f(x)d\mu_{a,x}=\int_{a}\lambda_{AB}^a\lambda_{AE}^af(x)dxdy\quad \forall f(x):[x_A,0]\rightarrow \mathbb{R},
\end{align}
where $x_A$ is the $x$-coordinate of the vertex $A$. A similar definition will hold for $L_{i,y}^c$ using $\lambda_{BC}\lambda_{CD}$. On the triangle $b$ we shall use both $L_{i,x}^b$ (using $\lambda_{ED}\lambda_{BD}$) and $L_{i,y}^b$ (using $\lambda_{BE}\lambda_{ED}$). These Legendre polynomials are defined up to a constant factor so that we can normalize them by imposing that they assume the same value at the origin. This is possible by virtue of Proposition 2.1 in \cite{BOFFI1994}.

Our approach to the stability condition will be related to the modified inf-sup condition that can be written as

\vspace{-10pt}
\begin{align*}
\sup_{\hat{\boldsymbol{v}}\in \overline{WV}_{k+j+1}}\frac{b(\hat{\boldsymbol{v}},\hat{q})}{\| \hat{\boldsymbol{v}}\|}\geq\mu\|\nabla \hat{q}\|,\quad \forall \hat{q}\in
\overline{P}_{k+j},
\end{align*}
which implies the standard one \cite{Verfurth1984}.

For every fixed $\hat{q}\in \overline{P}_{k+j}$ we want to construct $\hat{\boldsymbol{v}}\in \overline{WV}_{k+j+1}$ such that

\vspace{-10pt}
\begin{align}\label{eqn:inf-sup condition 1}
-\int_{\Omega}\hat{\boldsymbol{v}}\cdot\nabla\hat{q}dx&\geq c_1\|\nabla\hat{q}\|^2,\\
\label{eqn:inf-sup condition 2}
\|\hat{\boldsymbol{v}}\|_{0,\Omega}&\leq c_2\|\nabla\hat{q}\|.
\end{align}

Define:

\vspace{-10pt}
\begin{align*}
&\hat{\boldsymbol{v}}(x,y)=(\hat{v}_1(x,y),\hat{v}_2(x,y)),\\
&\hat{v}_1(x,y)_{|_a}=-\lambda_{AB}^a\lambda_{AE}^a\|\nabla \hat{q}\|L_{k-1,x}^a\cdot sign(H_a),\\
&\hat{v}_2(x,y)_{|_a}=-\lambda_{AB}^a\lambda_{AE}^a\frac{\partial \hat{q}}{\partial y},\\
&\hat{v}_1(x,y)_{|_b}=-\lambda_{ED}^b\lambda_{BD}^b\|\nabla \hat{q}\|L_{k+d-1,x}^b\cdot sign(H_b)
-\lambda_{ED}^b\lambda_{EB}^b\frac{\partial \hat{q}}{\partial x},\\
&\hat{v}_2(x,y)_{|b}=-\lambda_{EB}^b\lambda_{BD}^b\frac{\partial \hat{q}}{\partial y}
-\lambda_{EB}^b\lambda_{ED}^b\|\nabla \hat{q}\|L_{k+d-1,y}\cdot sign(K_b),\\
&\hat{v}_1(x,y)_{|_c}=-\lambda_{BC}^c\lambda_{CD}^c\frac{\partial \hat{q}}{\partial x},\\
&\hat{v}_2(x,y)_{|_c}=-\lambda_{BC}^c\lambda_{CD}^c\|\nabla \hat{q}\|L_{k+d-1,y}\cdot sign(K_c),
\end{align*}
where $sign(x)$ is sign function defined as

\vspace{-10pt}
\begin{align*}
sign(x)=
\begin{cases}
1,& x>0,\\
0,& x=0,\\
-1,& x<0.
\end{cases}
\end{align*}
and

\vspace{-10pt}
\begin{align*}
&H_a=\int_a\lambda_{AB}^a\lambda_{AE}^aL_{k+d-1,x}\cdot\frac{\partial \hat{q}}{\partial x},\quad H_b=\int_b\lambda_{ED}^b\lambda_{BD}^bL_{k+d-1,x}\cdot\frac{\partial \hat{q}}{\partial x},\\
&K_a=\int_b\lambda_{EB}^b\lambda_{ED}^bL_{k+d-1,y}\cdot\frac{\partial \hat{q}}{\partial y},\quad K_b=\int_c\lambda_{BC}^c\lambda_{CD}^cL_{k+d-1,y}\cdot\frac{\partial \hat{q}}{\partial y}.
\end{align*}

First of all, we observe that $\hat{\boldsymbol{v}}$ is an element of $\overline{WV}_{k+j+1}$, by the virtue of the fact that the tangential components of $\nabla \hat{q}$ along the interface $EB$ and $BD$ are continuous.

It is easy to verify that $\hat{\boldsymbol{v}}$ satisfies (\ref{eqn:inf-sup condition 2}). In order to check the validity of (\ref{eqn:inf-sup condition 1}), define
$|\!|\!|\nabla \hat{q}|\!|\!|=-\int_{\Omega}\hat{\boldsymbol{v}}\cdot\nabla \hat{q}$. Then

\vspace{-10pt}
\begin{align}
\label{eqn:new norm equl zero total}
|\!|\!|\nabla \hat{q}|\!|\!|^2=&\int_a\lambda_{AB}^a\lambda_{AE}^a(\frac{\partial \hat{q}}{\partial y})^2+\|\nabla\hat{q}\|(|H_a|+|H_b|) \notag\\
&+\int_b\big(\lambda_{ED}^b\lambda_{EB}^b(\frac{\partial \hat{q}}{\partial x})^2+\lambda_{ED}^b\lambda_{BD}^b(\frac{\partial \hat{q}}{\partial y})^2\big)\\
&+\|\nabla\hat{q}\|(|K_a|+|K_b|)+\int_c\lambda_{BC}^c\lambda_{CD}^c(\frac{\partial \hat{q}}{\partial x})^2.\notag
\end{align}

We verify that the expression $\|\frac{\partial \hat{q}}{\partial x}\|_H:=|H_a|+|H_b|$ is a norm of $\frac{\partial \hat{q}}{\partial x}$ in $a\cup b$ and $\|\frac{\partial \hat{q}}{\partial y}\|_K:=|K_a|+|K_b|$ is a norm of $\frac{\partial \hat{q}}{\partial y}$ in $b\cup c$.

Step 1. We will show $|H_a|+|H_b|$ vanishes only when $\frac{\partial\hat{q}}{\partial x}$ equals zero. From (\ref{eqn:measure a})

\vspace{-10pt}
\begin{align*}
0=\|\frac{\partial \hat{q}}{\partial x}\|_H=|\int_{x_A}^0L_{k+d-1,x}\cdot\frac{\partial \hat{q}}{\partial x}|+|\int_0^1L_{k+d-1,x}\cdot\frac{\partial \hat{q}}{\partial x}|
\end{align*}
From the orthogonality of Legendre polynomials $L_{i,x}^a, L_{i,x}^b$ and noting that $\frac{\partial\hat{q}}{\partial x}$ is a homogeneous polynomial of degree $k+d-1$, we have $\frac{\partial \hat{q}}{\partial x}=0$ in $a\cup b$.

Step 2. We will get $\|k\frac{\partial \hat{q}}{\partial x}\|_H=|k|\|\frac{\partial \hat{q}}{\partial x}\|_H$ from

\vspace{-10pt}
\begin{align*}
\|k\frac{\partial \hat{q}}{\partial x}\|_H=|\int_{x_A}^0L_{k+d-1,x}\cdot k\frac{\partial \hat{q}}{\partial x}|+|\int_0^1L_{k+d-1,x}\cdot k\frac{\partial \hat{q}}{\partial x}|.
\end{align*}

Step 3.  We will show that $\|\frac{\partial \hat{q}_1}{\partial x}+\frac{\partial \hat{q}_2}{\partial x}\|_H\leq
\|\frac{\partial \hat{q}_1}{\partial x}\|_H+\|\frac{\partial \hat{q}_2}{\partial x}\|_H$.

\vspace{-10pt}
\begin{align*}
\|\frac{\partial \hat{q}_1}{\partial x}+\frac{\partial \hat{q}_2}{\partial x}\|_H&=|\int_{x_A}^0L_{k+d-1,x}\cdot(\frac{\partial \hat{q}_1}{\partial x}+\frac{\partial \hat{q}_2}{\partial x})|+|\int_0^1L_{k+d-1,x}\cdot(\frac{\partial \hat{q}_1}{\partial x}+\frac{\partial \hat{q}_2}{\partial x})|\\
&\leq |\int_{x_A}^0L_{k+d-1,x}\cdot\frac{\partial \hat{q}_1}{\partial x}|
+|\int_{x_A}^0L_{k+d-1,x}\cdot\frac{\partial \hat{q}_2}{\partial x}|
+|\int_0^1L_{k+d-1,x}\cdot\frac{\partial \hat{q}_2}{\partial x}|
+|\int_0^1L_{k+d-1,x}\cdot\frac{\partial \hat{q}_1}{\partial x}|\\
&=\|\frac{\partial \hat{q}_1}{\partial x}\|_H+\|\frac{\partial \hat{q}_2}{\partial x}\|_H.
\end{align*}

Similarly, $\|\frac{\partial \hat{q}}{\partial y}\|_K=|K_b|+|K_c|$ is a norm of $\frac{\partial \hat{q}}{\partial y}$ in $b\cup c$. From the equivalence of norms on a finite dimensional space, there exists a constant $C_a,C_b,C_c,C_H, C_K>0$ such that

\vspace{-10pt}
\begin{align*}
&\int_a\lambda_{AB}^a\lambda_{AE}^a(\frac{\partial \hat{q}}{\partial y})^2\geq C_a\int_a(\frac{\partial \hat{q}}{\partial y})^2,\quad
\|\frac{\partial \hat{q}}{\partial x}\|_H\geq C_H\sqrt{\int_{a}(\frac{\partial \hat{q}}{\partial x})^2+\int_{b}(\frac{\partial \hat{q}}{\partial x})^2},\\
&\int_b\big(\lambda_{ED}^b\lambda_{EB}^b(\frac{\partial \hat{q}}{\partial x})^2+\lambda_{ED}^b\lambda_{BD}^b(\frac{\partial \hat{q}}{\partial y})^2\big)
\geq C_b\int_b\big((\frac{\partial \hat{q}}{\partial x})^2+(\frac{\partial \hat{q}}{\partial y})^2\big)\\
&\int_c\lambda_{BC}^c\lambda_{CD}^c(\frac{\partial \hat{q}}{\partial x})^2\geq C_c \int_c(\frac{\partial \hat{q}}{\partial x})^2,\quad \|\frac{\partial \hat{q}}{\partial y}\|_H\geq C_K\sqrt{\int_{b}(\frac{\partial \hat{q}}{\partial y})^2+\int_{c}(\frac{\partial \hat{q}}{\partial y})^2}.
\end{align*}
Set $c_1=\min\{C_a,C_b,C_c,C_H, C_K\}$ and obtain (\ref{eqn:inf-sup condition 1}).

\textbf{Case $d=3$}: We use the macroelement described by Stenberg in \cite{Stenberg1987} in order to check the inf-sup condition. Let $\mathcal{M}$ be a macroelement partition of the domain decomposition of $\mathcal{T}$. For a macroelement $M\in\mathcal{M}$ we introduce the following usual notation:

\vspace{-10pt}
\begin{align*}
&WV_M=\{\hat{\boldsymbol{v}}_{|M}~|~\hat{\boldsymbol{v}}\in W_{k+d+1}\}\cap[H_0^1(M)]^3,\\
&WP_M=\{\hat{q}_{|M}~|~\hat{q}\in WP_{k+d}\}.
\end{align*}

Consider a generic macroelement $M\in\mathcal{M}$. Let $T_0\in\mathcal{T}$ be a tetrahedron of $M$ and denote by $x_0$ the internal vertex of $T_0$ which also belongs to the other element of $M$. There are three edges $e_i, i=1,\cdots,3$ of $T_0$ meeting at $x_0$. Thanks to the fact that $x_0$ is internal, none of the edges $e_i$ lie on the boundary $\partial \Omega$.

Let $\hat{q}\in WP_M$ be given and suppose that

\vspace{-10pt}
\begin{align}\label{eqn:3d divergence free}
\int_M \hat{q}\nabla\cdot\hat{\boldsymbol{v}}=0,\quad \forall\hat{\boldsymbol{v}}\in WV_M.
\end{align}
We shall prove that $\nabla \hat{q}$ vanishes on $T_0$, thus obtaining H1 condition described in Theorem 2.1 in \cite{BOFFI1997} by virtue of the fact that $T_0$ is arbitrary and $q$ is continuous.

First, we concentrate our attention on the edge $e_1$ and fix an $(x,y,z)$-coordinate system in such a way that $e_1$ lies in the direction of the $x$-axis. We consider the collection $\mathcal{A}=\{T_0,\cdots,T_n\}$ of those elements of $\mathcal{T}$ which share the edge $e_1$ in common with $T_0$ (including $T_0$ itself). It is clear that $T_i\in M$ and that exactly two faces of $T_i$ touch other elements of $\mathcal{A}$ of every $i$.

Define $\hat{\boldsymbol{v}}$ in the following way:

\vspace{-10pt}
\begin{align*}
&\hat{\boldsymbol{v}}_{|T_i}=\big(\lambda_i\kappa_i\frac{\partial \hat{q}}{\partial x},0,0\big),\\
&\hat{\boldsymbol{v}}_{|T}=(0,0,0),\quad \mbox{~if~} T\in\mathcal{T}, ~T\neq T_i, ~\forall i,
\end{align*}
where $\lambda_i$ and $\kappa_i$ are the equations of the two faces of $T_i$ which are not in common with any other element of $\mathcal{A}$, normalized in order to assume the same value at the opposite vertex. It is worthwhile to observe that these two vertices are $x_0$ and the other extreme of the edge $e_1$.

It is easily verified that $\hat{\boldsymbol{v}}$ is a polynomial of degree $k+1$ and that it is continuous in $M$. The continuity of $\hat{q}$ in $M$ ensures that the gradient of $\hat{q}$ is continuous between two elements in all the directions which are contained in the plane of the interface; the $x$-axis is the direction of $e_1$ which is the edge of all common faces among the elements of $\mathcal{A}$. Moreover, $\hat{\boldsymbol{v}}$ vanishes at the boundary of $M$; hence, the following inclusion holds:

\vspace{-10pt}
\begin{align*}
\hat{\boldsymbol{v}}\in WV_M.
\end{align*}

Suppose now that (\ref{eqn:3d divergence free}) hold.

\vspace{-10pt}
\begin{align*}
0=\int_M \hat{q}\nabla\hat{\boldsymbol{v}}=-\int_{M}\nabla\hat{q}\cdot\hat{\boldsymbol{v}}
=-\sum_{i=1}^n\int_{T_i}\lambda_i\kappa_i(\frac{\partial\hat{q}}{\partial x})^2.
\end{align*}
It follows that the component of $\nabla\hat{q}$ in the direction of the $x$-axis vanishes in $T_i$ for every $i$.

The same argument applies to the edge $e_2$ and $e_3$, giving the result that $\nabla\hat{q}$ vanishes on $T_0$ in the direction of $e_i$, for $i=1,\cdots,3$. These three directions being independent, the final result

\vspace{-10pt}
\begin{align*}
\nabla\hat{q}=(0,0,0),\quad \mbox{~in~} T_0
\end{align*}
is obtained and the lemma is proved. Then the H1 condition of Theorem 2.1 in \cite{BOFFI1997} is proved and the H2-H3 conditions follow immediately from the regularity assumption of $\mathcal{T}$.

\end{proof}

\bibliographystyle{spmpsci}
\bibliography{tex}

\begin{thebibliography}{10}
\providecommand{\url}[1]{{#1}}
\providecommand{\urlprefix}{URL }
\expandafter\ifx\csname urlstyle\endcsname\relax
  \providecommand{\doi}[1]{DOI~\discretionary{}{}{}#1}\else
  \providecommand{\doi}{DOI~\discretionary{}{}{}\begingroup
  \urlstyle{rm}\Url}\fi

\bibitem{Antonietti2013}
Antonietti, P.F., da~Veiga, L.B., Lovadina, C., Verani, M.: {Hierarchical A
  Posteriori Error Estimators for the Mimetic Discretization of Elliptic
  Problems}.
\newblock SIAM Journal on Numerical Analysis \textbf{51}(1), 654--675 (2013)

\bibitem{Araya2008}
Araya, R., Barrenechea, G.R., Poza, A.: {An adaptive stabilized finite element
  method for the generalized Stokes problem}.
\newblock Journal of Computational and Applied Mathematics \textbf{214}(2),
  457--479 (2008)

\bibitem{Araya2005}
Araya, R., Poza, A.H., Stephan, E.P.: {A hierarchical a posteriori error
  estimate for an advection-diffusion-reaction problem}.
\newblock Mathematical Models and Methods in Applied Sciences \textbf{15}(7),
  1119--1139 (2005)

\bibitem{Araya2012}
Araya, R., Poza, A.H., Valentin, F.: {On a hierarchical error estimator
  combined with a stabilized method for the Navier-Stokes equations}.
\newblock Numerical Methods for Partial Differential Equations \textbf{28}(3),
  782--806 (2012)

\bibitem{Araya2018}
Araya, R., Rebolledo, R.: {An a posteriori error estimator for a LPS method for
  Navie-Stokes equations}.
\newblock Applied Numerical Mathematics \textbf{127}, 179--195 (2018)

\bibitem{Arnold2013}
Arnold, D.N.: {Spaces of Finite Element Differential Forms. In: Brezzi, F.,
  Colli-Franzone, P., Gianazza, U.P., Gilardi, G. (eds.) Analysis and Numerics
  of Partial Differential Equations}, \emph{Springer INdAM Series}, vol.~4.
\newblock Springer, Milan (2013)

\bibitem{Babuska1972}
Babu\v{s}ka, I., Aziz, A.: {Survey lectures on the mathematical foundations of
  the finite element method, in The Mathematical Foundations of the Finite
  Element Method with Applications to Partial Differential Equations}.
\newblock Academic Press, New York, NY (1972)

\bibitem{BOFFI1994}
Boffi, D.: {Stability of higher order triangular Hood-Taylor methods for the
  stationary Stokes equations}.
\newblock Mathematical Models and Methods in Applied Sciences \textbf{04}(2),
  223--235 (1994)

\bibitem{BOFFI1997}
Boffi, D.: {Three-Dimensional Finite Element Methods for the Stokes Problem}.
\newblock SIAM Journal on Numerical Analysis \textbf{34}(2), 664--670 (1997)

\bibitem{Brezzi2006}
Boffi, D., Brezzi, F., Demkowicz, L.F., Dur{\'{a}}n, R.G., Falk, R.S., Fortin,
  M.: {Mixed Finite Elements, Compatibility Conditions, and Applications},
  \emph{Lecture Notes in Mathematics}, vol. 1939.
\newblock Springer Berlin Heidelberg, Berlin, Heidelberg (2008)

\bibitem{Bornemann1996}
Bornemann, F.A., Erdmann, B., Kornhuber, R.: {A posteriori error estimates for
  elliptic problems in two and three space dimensions}.
\newblock SIAM Journal on Numerical Analysis \textbf{33}(3), 1188--1204 (1996)

\bibitem{Dorfler1996}
D{\"{o}}rfler, W.: {A Convergent Adaptive Algorithm for Poisson's Equation}.
\newblock SIAM Journal on Numerical Analysis \textbf{33}(3), 1106--1124 (1996)

\bibitem{Dorfler2002}
D{\"{o}}rfler, W., Nochetto, R.H.: {Small data oscillation implies the
  saturation assumption}.
\newblock Numerische Mathematik \textbf{91}(1), 1--12 (2002)

\bibitem{Eijkhout1991}
Eijkhout, V., Vassilevski, P.: {The Role of the Strengthened
  Cauchy-Buniakowskii-Schwarz Inequality in Multilevel Methods}.
\newblock SIAM Review \textbf{33}(3), 405--419 (1991)

\bibitem{Giani2021}
Giani, S., Grubi{\v{s}}i{\'{c}}, L., Hakula, H., Ovall, J.S.: {A Posteriori
  Error Estimates for Elliptic Eigenvalue Problems Using Auxiliary Subspace
  Techniques}.
\newblock Journal of Scientific Computing \textbf{88}(3) (2021)

\bibitem{HainReduced2019}
Hain, S., Ohlberger, M., Radic, M., Urban, K.: {A hierarchical a posteriori
  error estimator for the Reduced Basis Method}.
\newblock Advances in Computational Mathematics \textbf{45}, 2191--2214 (2019)

\bibitem{Hakula2017}
Hakula, H., Neilan, M., Ovall, J.S.: {A Posteriori Estimates Using Auxiliary
  Subspace Techniques}.
\newblock Journal of Scientific Computing \textbf{72}(1), 97--127 (2017)

\bibitem{Huang2011}
Huang, W., Russell, R.D.: {Adaptive Moving Mesh Methods}, \emph{Applied
  Mathematical Sciences}, vol. 174.
\newblock Springer, New York (2011)

\bibitem{Nochetto2009}
Nochetto, R.H., Siebert, K.G., Veeser, A.: {Theory of adaptive finite element
  methods: An introduction}.
\newblock Multiscale, Nonlinear and Adaptive Approximation: Dedicated to
  Wolfgang Dahmen on the Occasion of his 60th Birthday pp. 409--542 (2009)

\bibitem{Pfeiler2020}
Pfeiler, C.m., Praetorius, D.: {D{\"{o}}rfler marking with minimal cardinality
  is a linear complexity problem}.
\newblock Mathematics of Computation \textbf{89}(326), 2735--2752 (2020)

\bibitem{Scott1990}
Scott, L.R., Zhang, S.: {Finite Element Interpolation of Nonsmooth Functions
  Satisfying Boundary Conditions}.
\newblock Mathematics of Computation \textbf{54}(190), 483 (1990)

\bibitem{Stenberg1987}
Stenberg, R.: {On some three-dimensional finite elements for incompressible
  media}.
\newblock Computer Methods in Applied Mechanics and Engineering \textbf{63}(3),
  261--269 (1987)

\bibitem{Verfurth1984}
Verf{\"{u}}rth, R.: {Error estimates for a mixed finite element approximation
  of the Stokes equations}.
\newblock RAIRO. Analyse num{\'{e}}rique \textbf{18}(2), 175--182 (1984)

\bibitem{Verfurth1996}
Verf{\"{u}}rth, R.: {A review of a posteriori error estimation and adaptive
  mesh-refinement techniques}.
\newblock Teubner, Stuttgart (1996)

\bibitem{Verfurth2013}
Verf{\"{u}}rth, R.: {A Posteriori Error Estimation Techniques for Finite
  Element Methods}.
\newblock Oxford University Press, Oxford (2013)

\bibitem{Zienkiewicz1986}
Zienkiewicz, O., Graig, A.: {Adaptive refinement, errorestimates, multigrid
  solution and hierarchical finite element method concepts}.
\newblock Accuracy Estimates and Adaptive Refinements in Finite Element
  Computations, John Wiley and Sons pp. 25--59 (1986)

\bibitem{Zienkiewicz1982}
Zienkiewicz, O., Kelly, D., Gago, J., Babu\v{s}ka, I.: {Hierarchical finite
  element approaches, error estimates and adaptive refinemen}.
\newblock The Mathematics of Finite Elements and Applications IV, Academic
  Press pp. 313--346 (1982)

\end{thebibliography}

\end{document}